\documentclass[12pt, psamsfonts, reqno]{amsart}
\usepackage{graphics}
\usepackage{amsmath}
\usepackage{amsthm}
\usepackage{amssymb}
\usepackage{amscd}
\usepackage{amsfonts}
\usepackage{amsbsy}
\usepackage{epsfig}

\newtheorem{theorem}{Theorem}
\newcounter{letter}

\newtheorem{maintheorem}[letter]{Theorem}
\newtheorem{proposition}[theorem]{Proposition}
\newtheorem{example}[theorem]{Example}
\newtheorem{definition}[theorem]{Definition}
\newtheorem{corollary}[theorem]{Corollary}
\newtheorem{lemma}[theorem]{Lemma}
\newtheorem{remark}[theorem]{Remark}
\newtheorem{conjecture}[theorem]{Conjecture}

\def\eps{\varepsilon}

\def\R{{\mathbb R}}
\def\T{{\mathbb T}}

\def\N{{\mathbb N}}
\def\P{{\mathbb P}}

\def\O{{\mathcal O}}

\def\E{{\mathbb E}}
\def\S{{\mathbb S}}
\def\B{{\mathcal B}}
\def\A{{\mathcal A}}
\def\D{{\mathcal D}}

\def\orb{\mathop{\rm orb}\nolimits}
\def\fullorb{\mathop{\rm fullorb}\nolimits}
\def\diam{\mathop{\rm diam}\nolimits}
\def\dist{\mathop{\rm dist}\nolimits}
\def\bas{\mathop{\rm Bas}\nolimits}

\def\supp{\mathop{\rm supp}\nolimits}
\def\Crit{\mathop{\rm Crit}\nolimits}
\def\Feig{\mathop{\rm Feig}\nolimits}

\def\ie{{\em i.e., }}
\def\eg{{\em e.g.~}}
\def\fp{\,\mathfrak{fp}}
\renewcommand{\a}{g}
\def\Inte{\mathop{\rm Int}\nolimits}
\def\Cl{\mathop{\rm Cl}\nolimits}
\def\Dis{\mathop{\rm Dis}\nolimits}
\def\LY{\mathop{\rm LY}\nolimits}
\def\Prox{\mathop{\rm Prox}\nolimits}
\def\Asymp{\mathop{\rm Asymp}\nolimits}

\def\AsPer{\mathop{\rm AsPer}\nolimits}

\def\cyc{\mathop{\rm cyc}\nolimits}

\def\CIInf{C^2_{\text{\rm nf}}}
\def\CIIInf{C^3_{\text{\rm nf}}}
\begin{document}

\bibliographystyle{plain}
\title[version of \today]
{On the Lebesgue measure of Li-Yorke pairs for interval maps}
\author{Henk Bruin
and V\'{\i}ctor Jim\'enez L\'opez}
\thanks{HB gratefully acknowledges the support of EPSRC
grant EP/F037112/1 and also the hospitality of the University of Murcia and
Delft University of Technology.
\\ \indent
 VJL was partially supported by MICINN
(Ministerio de Ciencia e Innovacion) and FEDER (Fondo Europeo de
Desarrollo Regional), grant MTM2008-03679/MTM, and Fundaci\'on
S\'eneca (Agencia de Ciencia y Tecnolog\'{\i}a de la Regi\'on de
Murcia, Programa de Generaci\'on de Conocimiento Cient\'{\i}fico
de Excelencia, II PCTRM 2007-10), grant 08667/PI/08.}

\subjclass[2000]{Primary: 37E05, Secondary: 26A18, 37B20,  37C70}
\keywords{Li-Yorke pair, Li-Yorke chaos, scrambled set, interval
map, distal pair, asymptotic pair, attractor, strongly wandering
set}
\date{\today}

\begin{abstract}
We investigate the prevalence of Li-Yorke pairs
for $C^2$ and $C^3$ multimodal maps $f$ with non-flat critical points. We
show that every measurable scrambled set has zero Lebesgue
measure and that all strongly wandering sets have
zero Lebesgue measure, as does the set of pairs of asymptotic (but not
asymptotically periodic) points.

If $f$ is topologically mixing and has no Cantor attractor, then
typical (w.r.t.\ two-dimen\-sional Lebesgue measure) pairs are
Li-Yorke; if additionally $f$ admits an absolutely continuous
invariant probability measure (acip), then typical pairs have a
dense orbit for $f \times f$.
These results make use of so-called nice neighborhoods of the
critical set of general multimodal maps, and hence uniformly expanding
Markov induced maps, the existence of either is proved in this paper as well.

For the setting where $f$ has a Cantor attractor, we present a
trichotomy explaining when the set of Li-Yorke pairs and distal
pairs have positive two-dimensional Lebesgue measure.
\end{abstract}

\maketitle

\section{Introduction}\label{sec:intro}

In interval dynamics there are many ways to deal with the notion
of asymptotic complexity (``chaos''). Probably it is pointless to
try and decide which is the best of them, but in applications
there are two of them which are by far the most popular, see \eg
\cite{Day, MMN}. One is topological chaos, that is, the existence
of an uncountable scrambled set in the sense of the famous Li and
Yorke paper \cite{LY}. The other one is ergodic chaos, that is,
the existence of an invariant probability measure absolutely
continuous with respect to the Lebesgue measure (acip). Neither of
them is without drawbacks. (To keep this introduction at
the expository level, we have deferred most of definitions to the
subsequent sections.)

There are easy conditions implying  the existence of Li-Yorke
chaos and its stability under small perturbations. One such
condition is the existence of a periodic orbit of period not a
power of two \cite{Blo}. Nevertheless, this chaos need not be
``observable'': for instance, the orbits  of almost all points
(in the sense of Lebesgue measure) may be attracted by
this periodic orbit \cite{Guc}.

On the other hand, ergodic chaos ensures complicated dynamics for
a set of points with positive measure. For instance, if a smooth
multimodal map (with non-flat critical points) $f$ has an acip,
then, as a simple consequence of the zero measure of Cantor metric
attractors \cite{vSV} and Proposition~\ref{propAttr} below, there
is a positive measure set of points whose orbit under $f$ is dense
in some interval. However, the converse is not true, even for the
family of logistic maps \cite{Joh,Lyu}, so acips can only exist
under additional conditions (for instance, hyperbolic repelling
periodic points and $|Df^n(f(c))|\to \infty$ for any critical
point $c$ of the map $f$ \cite{BRSS}).

A map $f$ is infinitely renormalizable if there is an infinite
collection of nested cycles of periodic intervals; the
intersection of these cycles is a Cantor set, called a solenoidal
attractor (because the suspension over this attractor is a
topological solenoid). The Feigenbaum map, or more correctly
Coullet-Tresser-Feigenbaum, (see \eg \cite[pp.
151-152]{dMvS}) is the best known example of this. A solenoidal
attractor is Lyapunov stable and as shown in \cite{BrJi} (see
Proposition~\ref{PropApproxPer} below), points in the basin of
such an attractor are approximately periodic:

\begin{definition}\label{DefApproxPer}
A point $x$ is \emph{approximately periodic} if for
every $\eps>0$ there is a periodic point $p$ such that $\limsup_{n
\to\infty} |f^n(x)-f^n(p)| < \eps$.
\end{definition}

Hence, up to small
errors, almost all points eventually behave as periodic points.

\begin{remark}
An intrinsic characterization of adding machines (namely any
system in which every point is regularly recurrent) is presented
in \cite{BK}. In our setting, it applies to the solenoidal
attractor itself, whereas approximate periodicity gives
information on a neighborhood of the solenoidal attractor.
\end{remark}

However, a multimodal (even polynomial) map may have a dense orbit
while, simultaneously, almost all orbits are attracted by a Cantor
set (a so-called wild attractor) \cite{BKNS}. As we will explain
below, in such a case it is still possible but not necessary that
a.e.\ point is approximately periodic.

We see that there is a variety of smooth multimodal maps featuring
a certain degree of ``observable'' dynamical complexity which is,
however, not strong enough to be realized by an acip. It is natural to
return to the Li-Yorke notion of chaos and investigate to
what extent it can be used to measure this complexity. This is
what we intend in the present paper.

\begin{definition} \label{def:basic}
Let $f:I =[0,1]\to I$ be a continuous map. A pair of points
$(x,y)$ is called:
\begin{itemize}
\item \emph{distal} if $\liminf_{n\to\infty} |f^n(y)-f^n(x)| > 0$;\\
\item \emph{asymptotic} if $\lim_{n\to\infty} |f^n(y)-f^n(x)| = 0$;\\
\item \emph{Li-Yorke} if it is neither asymptotic nor distal, that is,
$$
0 = \liminf_{n\to\infty} |f^n(y)- f^n(x)| < \limsup_{n\to\infty}
|f^n(y)- f^n(x)|.
$$
\end{itemize}
We denote the set of distal, asymptotic and Li-Yorke pairs by
$\Dis$, $\Asymp$ and $\LY$, respectively.
A pair that is not distal (hence asymptotic or Li-Yorke)
is called {\em proximal}. The set of Li-Yorke
pairs with $\limsup_{n\to\infty} |f^n(y)- f^n(x)| \geq \eps$ is
denoted as $\LY_\eps$.
\end{definition}

Note that $f^n(x) \neq f^n(y)$ for all $n \geq 0$ whenever $(x,y)$
is a Li-Yorke or distal pair.

\begin{definition}
Let $f:I \to I$ be a continuous map.
\begin{itemize}
 \item A set $S \subset X$ is called {\em scrambled} if any two
distinct points in $S$ form a Li-Yorke pair.
 \item The map $f$ is called \emph{chaotic (in the sense of Li-Yorke)} if it has
an uncountable scrambled set.
\end{itemize}
 \iffalse $\bullet$ A map $f$ is
{\em spatio-temporally chaotic} if for every $x \in X$ and
neighborhood $U \owns x$, there is $y \in U$ such that $(x,y) \in
\LY$.  \fi

The above definitions can be strengthened by assuming that there
is a positive lower bound of $\limsup |f^n(x) - f^n(y)|$
independently of $x,y$:
\begin{itemize}
 \item If there is $\eps > 0$ such that for every $x \neq y \in S$,
$\limsup |f^n(x) - f^n(y)| \geq \eps$ independent of $x,y \in
S$, then $S$ is called {\em $\eps$-scrambled}.
 \item If there is $\eps > 0$ such that for every $x \in X$
and every neighborhood $U \owns x$, there is $y \in U$ such that
$(x,y) \in \LY_\eps$, then $f$ is {\em Li-Yorke sensitive}.
\end{itemize}
\end{definition}

We emphasize that $\eps$-scrambled is a stronger property than
just scrambled. For example, \cite[Proposition 5]{BHS} shows the
possibility of having scrambled sets that are not $\eps$-scrambled
for any $\eps > 0$.

Our concrete aim is to investigate the Lebesgue measure of the
above sets for $C^2$ (or sometimes $C^3$) multimodal maps from the
interval $I$ into itself with non-flat critical points (denoted by
$\CIInf(I)$ and $\CIIInf(I)$, respectively). The advantage of
working in this setting is that there are many tools at hand to
deal with measure-theoretic properties. Remarkably the most
important of these tools is purely topological: maps from
$\CIInf(I)$ have no wandering intervals, see \cite[Theorem A, p.
267]{dMvS}.

As it happens, some of the results in the paper are based on a
generalization of this property which is of interest in itself.

\begin{definition}\label{def:wandering}
A point $x$ is \emph{asymptotically periodic}, written $x\in \AsPer$,
if there is a periodic point $p$ such that
$\lim_{n\to \infty} |f^n(x)-f^n(p)|=0$.
A measurable set $W$ is \emph{strongly wandering} if
$f^n(W) \cap f^m(W) = \emptyset$ for all $n > m \geq 0$ and $W$
contains no asymptotically periodic points.
A \emph{wandering interval} is just an interval which is strongly wandering.
\end{definition}

The notion of strongly wandering set was introduced by Blokh and
Lyubich in \cite{BL} (in a slightly different way). Under the
assumption of negative Schwarzian derivative, the non-existence of
strongly wandering sets of positive measure was proved in the
unimodal case (and stated in the multimodal case without
inflection points) in \cite{BL}. Here we prove it for maps from
$\CIInf(I)$, for which inflection points are now allowed.

\begin{maintheorem} \label{mainthm:Strong}
 Let $f\in \CIInf(I)$. Then every strongly wandering set
 has zero Lebesgue measure.
\end{maintheorem}

Concerning the size of Li-Yorke chaos, the first natural question
is whether smooth multimodal maps may have scrambled sets of
positive Lebesgue measure. There is extensive literature on the
subject. Examples of continuous maps possessing scrambled sets of
positive or even full measure are well known: \cite{Kan, Smi2,
Mis, BrHu}. In fact, if $f$ is chaotic (respectively, has a dense
orbit), then it is topologically conjugate to a map having a
positive (respectively, full) measure scrambled set \cite{JaSm,
Smi3} (respectively, \cite{SmSt}).

Of course, it is not possible that the whole interval is scrambled
(in fact, a scrambled set cannot be residual on any subinterval of
$I$, see \cite{Ged}), although there are maps  with scrambled
forward invariant Cantor sets \cite{HuYe}. However such maps
cannot be multimodal (see Proposition~\ref{prop:noinvscramb}).

It is worth emphasizing that maps of type $2^\infty$ in the
Sharkovskiy ordering (that is, those having periodic points of
periods all powers of 2, but no other periods) can possess
scrambled sets of positive measure, but not of full measure
because any Li-Yorke chaotic map of type $2^\infty$ has a wandering
interval. Indeed, if a map of type $2^\infty$ has no wandering
intervals, then all points are approximately periodic \cite{Smi3}.
However the following result is well known (see \eg \cite[p.
144]{BlCo}):

\begin{proposition}\label{prop:approximately}
If $f:I\to I$ is continuous and $x,y$ are approximately periodic
 points, then $(x,y)$ is either asymptotic or distal.
\end{proposition}

\noindent Hence a scrambled set can contain at most one
approximately periodic point. Finally, positive measure scrambled
sets may also exist for $C^\infty$ maps (with flat critical
points) or $C^1$ maps with non-flat critical points, but a $C^1$
map cannot have a full measure scrambled set \cite{Jim1, Jim2,
BaJi2}.

Nevertheless, it is a widely held view that
these are rather pathological examples. For instance, it is known
that neither maps from $\CIInf(I)$ with hyperbolic periodic points
and whose critical points satisfy the Misiurewicz condition, nor
maps in $\CIIInf(I)$ with negative Schwarzian derivative and
having no wild attractors, may possess scrambled sets of positive
measure \cite{BaJi1, BrJi}. Our next result confirms these expectations:

\begin{maintheorem} \label{mainthmC2}
 If $f\in \CIIInf(I)$, then it has
 no measurable scrambled sets of positive Lebesgue measure.
\end{maintheorem}

It seems rather paradoxical that scrambled sets have zero measure
even in the case when there is an acip, but this is not really so.
The key point is \emph{measurability}. For instance, one can
easily derive from \cite{Smi1} that the full logistic map
$f(x)=4x(1-x)$ possesses a non-measurable scrambled set with full
\emph{exterior} Lebesgue measure. Since the sets $\Dis, \Asymp,
\LY, \LY_\eps, \AsPer$ are all Borel, hence measurable sets
\cite{Jim3, BrJi}, the moral is that we should measure these sets
rather than scrambled sets. The idea of passing to the square $I
\times I$ to study topological or measure-theoretic properties of
Li-Yorke chaos is due to Lasota and was first used by Pi\'orek
\cite{Pio}.

To begin with, we prove that there almost no ``non-trivial''
asymptotic pairs.
\begin{maintheorem} \label{mainthm:Asymp}
 If $f\in \CIInf(I)$, then $\Asymp \setminus (\AsPer\times \AsPer)$ has zero
 (two-dimensional) Lebesgue measure.
\end{maintheorem}

We can consider Theorem~\ref{mainthm:Asymp} as a weak form of
sensitivity to initial conditions in the absence of periodic
attractors. It is really quite weak: it applies in particular to
the infinitely renormalizable case where almost all points are
attracted by a solenoidal set, so there is no sensitivity to
initial conditions in the standard Guckenheimer sense \cite{Guc}.
What Theorem~\ref{mainthm:Asymp} emphasizes is that there are no
``privileged'' routes (that is, with positive measure) to
measure-theoretic attractors.

The Li-Yorke property describes how chaotically pairs of points
behave with respect to each other, and is hence a property of the
Cartesian product $(I^2, f_2)$, for $I^2 = I \times I$ and
$f_2(x,y) := (f\times f)(x,y)=(f(x),f(y))$. If $(x,y)$ is
Li-Yorke, then $\orb((x,y))$ accumulates on, but does not converge
to, the diagonal of $I^2$. Hence LY is a weaker property than
$(x,y)$ having a dense orbit in $I^2$. We want to find conditions
ensuring that $\LY$ has positive (or full) mass w.r.t.\
two-dimensional Lebesgue measure $\lambda_2$. Recall the a map $g$
is called \emph{topologically mixing} if every iterate $g^n$ has a
dense orbit. According to Proposition~\ref{propAttr}, the orbit of
Lebesgue a.e.\ point is either attracted by a periodic point
or a solenoidal set, or eventually falls into an interval $K$ such
that $f^r(K)=K$ for some $r$ and the restriction $g=f^r|_K$ is
topologically mixing. As we explained earlier, approximately
periodic points cannot be used to produce Li-Yorke pairs. On the
other hand, all sets from Definition~\ref{def:basic} are the same
for $f$ as for any of its iterates $f^r$. Thus we can restrict
ourselves to the case where the map $f$ itself is topologically
mixing.

Now, if $f$ is topologically mixing, then either almost all points
have a dense orbit, or the orbits of almost all points are
attracted by finitely many pairwise disjoint wild Cantor
attractors. It turns out that in the first case $\lambda_2$-a.e.
pair $(x,y)$ is Li-Yorke.

\begin{maintheorem}\label{mainthmLY}
 Let $f\in \CIIInf(I)$ be a topologically mixing map having no
 Cantor attractors. Then the Cartesian product $(I^2,\lambda_2,f_2)$
 is ergodic and for every $x\in I$ there is a
 full measure set $C_x\subset I$ such that
  $$
  \liminf_{n\to\infty} |f^n(y)- f^n(x)|=0, \;\;\;\limsup_{n\to\infty}
  |f^n(y)- f^n(x)|\geq \diam(I)/2,
  $$
 for every $y\in C_x$. In particular, $\LY_{1/2}$ has full measure
 and $f$ is Li-Yorke sensitive.
\end{maintheorem}

Weaker versions of this result under additional Misiurewicz or
negative Schwarzian conditions were proved earlier \cite{BaJi1,
BrJi}. Ergodicity of $(I^2, \lambda_2, f_2)$ is parallel to
Lebesgue measure $\lambda$ being weak mixing, although in its
standard definition, weak mixing applies to invariant measures
only, see Subsection~\ref{subsec:ergodic}. If $f$ admits an acip,
then we can say more: $\lambda_2$-a.e. pair $(x,y)$ has a dense
orbit in $I^2$, see Corollary~\ref{cor:density}. However, it seems
possible that there are cases where $f$ admits no acip,
$\lambda_2$-a.e.\ pair $(x,y)$ is Li-Yorke, but has no dense orbit
in $I^2$.

It remains to consider the case where $f\in \CIInf(I)$ is a
topologically mixing map having a wild attractor $\A$. Let
$\bas(\A)$ be the set of points whose orbit is attracted by $\A$.
We already know (Theorem~\ref{mainthm:Asymp}) that $\Asymp$ has
zero measure.  Also, if some $x \in \bas(\A)$ is approximately
periodic, then by Proposition~\ref{PropApproxPer}, $\A$ is
conjugate to an adding machine, so all points in $\bas(\A)$ are
approximately periodic. Finally, recall that every pair of
approximately periodic points is either asymptotic or distal. Our
result in this area classifies which behaviors can (and indeed do)
occur:
\iffalse We conclude that one of the following alternatives
must occur: (a) a.e. pair of points in $\bas(\A)$ is distal and
every point in $\bas(\A)$ is approximately periodic; (b) a.e. pair
of points in $\bas(\A)$ is distal and no point in $\bas(\A)$ is
approximately periodic; (c) a.e. pair of points in $\bas(\A)$ is
Li-Yorke; (d) both $\Dis$ and $\LY$ have positive measure in
$\bas(\A)\times \bas(\A)$. We remark that in case (a), points
$x\in \bas(\A)$ have a \emph{target} point $t_x \in \A$ such that
$|f^n(y)-f^n(t_x)| \to 0$. In case (b), such target points $t_x
\in \A$ do not exist in general, cf.\ Remark 2 in \cite{BrHaw}.

It turns out that all these cases do really occur. In fact the
existence of a polynomial unimodal map having a wild attractor of
type (a) (so-called a \emph{wild adding machine}) is already known
\cite{Btree}. We provide examples of the other three types to
conclude:
\fi
\begin{maintheorem}\label{mainthm:Classification}
Let $f\in \CIInf(I)$ be a topologically mixing map with a wild
attractor $\A$. Then one of the following alternatives must occur:
\begin{itemize}
\item[(a)] Lebesgue a.e.\ pair of points in $\bas(\A)$ is distal and
every point in $\bas(\A)$ is approximately periodic;
\item[(b)] Lebesgue a.e.\ pair of points in $\bas(\A)$ is distal and
no point in $\bas(\A)$ is approximately periodic;
\item[(c)] Lebesgue a.e.\ pair of points in $\bas(\A)$ is Li-Yorke;
\item[(d)] Both $\Dis$ and $\LY$ have positive  Lebesgue measure in
$\bas(\A)\times \bas(\A)$.
\end{itemize}
There are examples of polynomial unimodal maps of all
above types (a)-(d) so that additionally, in cases (b)-(d),
 $\bas(\A)$ contains $\eps$-scrambled sets
for a fixed $\eps > 0$ and $f$ is Li-Yorke sensitive on $\bas(\A)$.
\end{maintheorem}

\section{Preliminaries}\label{sec:prelimnaries}

\subsection{Interval maps}\label{subsec:maps}

A continuous map $f:I \to I$ (for $I=[0,1]$) is called
\emph{multimodal} if $[0,1]$ can be decomposed into finitely many
subintervals on which $f$ is (strictly) monotone. A point $c$
is \emph{critical} if $f'(c) = 0$; the set of critical points is
denoted by $\Crit$. Critical points can be \emph{turning} points
(if $f$ assume a local extremum at $c$) or \emph{inflection}
points, and hence there can be more critical points than maximal
intervals of monotonicity. A differentiable map having exactly one
critical (turning) point is called \emph{unimodal}.
Near a turning
point $c$, there is a largest interval  $[a,b]$ such that
$f(a)=f(b)$ and $f$ is monotone on each of the intervals $[a,c]$
and $[c,b]$. Then there is a continuous involution
$\tau_c:[a,b]\to [a,b]$ such that $f(\tau_c(x))= f(x)$ and
$\tau_c(x)\neq x$ for every $x \neq c$.

We say that $f \in C^k_{nf}(I)$ if $f:I \to I$ has a finite
critical set $\Crit$, is $C^k$ and each critical point is {\em
non-flat}, \ie for each $c \in \Crit$, there exist a $C^k$
diffeomorphism $\varphi_c$ with $\varphi_c(c)=0$ and an $\ell_c
\in (1, \infty)$, called the {\em critical order} of $c$, such
that $f(x)=\pm|\varphi_c(x)|^{\ell_c} + f(c)$ for $x$ close to
$c$. It follows that if $n$ is the smallest integer $\geq \ell_c$,
then $D^m f(c)= 0$ for $1 \leq m < n$, but $D^n f(c) \neq 0$.
Conversely, if $f$ is $C^{k+1}$ near $c$ and $D^n f(c) \neq 0$ for
some $2\leq n\leq k+1$, then $c$ is non-flat.

We can always enlarge
the domain of $f$ (without adding new critical points) and
rescale such that
\begin{equation} \label{eq:nicely_scaled}
 \text{$f(\partial I)\subset \partial I$ \;\;\; and \;\;\; $\Crit\cap \partial
 I=\emptyset$.}
\end{equation}
Observe that this operation does not change the zero or positive
measure quality of the sets from Definition~\ref{def:basic}.
Hence, except when $f$ is topologically mixing, we will
always assume that \eqref{eq:nicely_scaled} holds. Also, we will
assume without loss of generality that
\begin{equation}
 \label{nonperiodic}
 \text{$\Crit$ does not contain any periodic point,}
\end{equation}
because if $c$ is a periodic critical point, then we can modify
slightly $f$ near the orbit of $c$ so that the resulting map has a
periodic orbit containing no critical points and attracting the
same points as previously attracted by the orbit of $c$.

Recall that the map $f$ is topologically mixing if every iterate
$f^n$ has a dense orbit. This excludes the case that there is a
proper compact subinterval $J$ of $I$ such that $f^r(J)\subset J$
for some $r\geq 1$ (\ie $f$ is \emph{non-renormalizable}). The map
$f$ is \emph{topologically exact}  (also called \emph{locally
eventually onto}) if for every non-degenerate interval $J \subset
I$, there is $n$ such that $f^n(J) = I$. For multimodal maps,
topologically mixing and topologically exact are equivalent, see
\eg \cite[pp. 157--158]{BlCo}.

\subsection{Attractors of interval maps}\label{subsec:attractors}

 The \emph{orbit} $\{f^n(x)\}_{n=0}^\infty$ of a point $x$
is denoted by $\orb(x)$. More generally, the orbit of a set $A$ is
$\orb(A) := \bigcup_{n=0}^\infty f^n(A)$. The \emph{$\omega$-limit
set} $\omega(x) := \bigcap_{n \in \N} \Cl\bigcup_{m \geq n}
f^n(x)$ of $x$ is the set of limit points of the orbit of $x$. If
$\A$ is a subset of $I$, then we call $\bas(\A) = \{ x \in I :
\omega(x) \subset \A\}$ the {\em basin (of attraction)} of $\A$. A
periodic orbit $\O$ is called \emph{attracting} if its basin
contains an open set. The union of the components of $\bas(\O)$
intersecting $\O$ is called the \emph{immediate basin} of $\O$. If
$\bas(\O)$ contains a neighborhood of $\O$, then $\O$ is called a
\emph{two-sided} attracting periodic orbit; otherwise it is called
a \emph{one-sided} attracting periodic orbit. If $p$ is a periodic
point of period $r$ and $|Df^r(p)|$ is less than, equal to, or
greater than $1$, then the orbit of $p$ is called \emph{hyperbolic
attracting}, \emph{parabolic}, or \emph{hyperbolic repelling}
respectively. Of course, only hyperbolic attracting and parabolic
orbits can be attracting. If $f\in \CIInf(I)$, then it is well
known  that all periodic orbits of sufficiently high period are
hyperbolic repelling \cite{MMS}. Furthermore, maps in $\CIInf(I)$
have no wandering intervals. One of the consequences of this is
the following useful result. Below $\dist(A,B)$ denotes the
distance between the sets $A$ and $B$ (by convention
$\dist(A,\emptyset)=\infty$).

\begin{proposition}[The non-Contraction Principle] \label{prop:contraction}
 Let $f:I\to I$ be a multimodal map without wandering
 intervals. Then for every $\eps>0$ there is $\delta>0$ such that
 if $J$ is an interval such that  $|J|<\delta$ and $\dist(J,p)>\eps$
 for each attracting periodic point $p$, then every
 component of every preimage $f^{-n}(J)$ of $J$ has length less
 than $\eps$.
\end{proposition}

\begin{proof}
This follows easily from the Contraction Principle as stated in \cite[p.
 305]{dMvS}. The principle is a bit of a misnomer, so we added
non- in our version.
\end{proof}

%Recall that a set $A$ is called \emph{(forward) invariant} for $f$
%if $f(A)\subset A$.

\begin{definition}\label{def:attractor}
We call a closed invariant set $\A$ a {\em (measure-theoretic)
attractor} if its basin has positive Lebesgue measure, and there
is no proper subset ${\A}' \subset \A$ with the same properties.
\end{definition}

%It follows from this definition that attractor are closed and
%invariant.
Note that $\A$ need not be an attractor in any
topological sense: the basin $\bas(\A)$ need not contain a
neighborhood of $\A$, nor be of second Baire category.

In order to understand the nature of measure-theoretic attractors
of maps from $\CIInf(I)$, certain types of interval cycles are of
particular interest. We say that a compact interval $K$ is
\emph{periodic (of period r)} if
$K,\ldots, f^{r-1}(K)$ have disjoint interiors and $f^r(K)\subset K$.
We call the union $\bigcup_{i=0}^{r-1}f^i(K)$ a \emph{cycle of
intervals} and denote it by $\cyc(K)$.
%If additionally the
%intervals $K,\ldots, f^{r-1}(K)$ themselves are disjoint and
%$f^r(K)=K$, then we say that $K$ is \emph{strongly periodic}.\comHB{strongly periodic}
%Let $K$ be a periodic interval of period $r$. If $f^r|_K$ is a
%homeomorphism, then we say that $\cyc(K)$ is \emph{neutral}. If
%$\cyc(K)$ is the closure of the immediate basin of a (possibly
%one-sided) periodic orbit and $f^r|_K$ is not a homeomorphism,
%then we call it \emph{elementary}. \comHB{elementary}
%If $f^r|_K$ is topologically
%mixing, then we call $\cyc(K)$ \emph{topologically mixing} as well.
If $K_0\supset K_1\supset \cdots$ is a nested sequence of
periodic intervals of periods $r_0<r_1<\cdots$, then
$S=\bigcap_{i=0}^\infty \cyc(K_i)$ is called a \emph{solenoidal
set} and all points from $S$ are called \emph{solenoidal points}.
If $r_{i+1}=2r_i$ for every $i$ sufficiently large, then we say
that both $S$ and its points are of \emph{Feigenbaum type}\label{feig}. If
$i_0$ is such that $\cyc(K_{i_0})\cap \Crit = S\cap \Crit$ and all
periodic points in $\cyc(K_{i_0})$ are hyperbolic repelling, then
we say that $\cyc(K_{i_0})$ is \emph{solenoidal}.
Thus all solenoidal sets are of Cantor type and all cycles
$\cyc(K_i)$ are solenoidal if $i$ is sufficiently large.

A compact invariant set $A$ is \emph{Lyapunov stable} if for every
neighborhood $U$ of $A$, there exists a neighborhood $V$ of $A$
such that $f^n(V)\subset U$ for every $n$. We have the following
classification of measure-theoretic attractors, see
\cite{mane,keller,BL,Lypre,dMvS,vSV}.

\begin{proposition}\label{propAttr}
If $f\in \CIInf(I)$, then $f$ has countably many attractors $\A$,
which are of the following types:
\begin{enumerate}
\item $\A$ is an attracting periodic orbit;
\item $\A$ is a cycle of intervals on which Lebesgue a.e.\ orbit is dense;.
\item $\A$ is a solenoidal set (the infinitely
 renormalizable case). Then $\A$ is Lyapunov stable and
 the basin $\bas(\A)$ is of second Baire
 category.
\item $\A$ is a minimal Cantor set, but not of the above type.
 In particular, $\A$ is not Lyapunov stable, and
 $\bas(\A)$ is of first Baire category.
\end{enumerate}
For Lebesgue a.e. $x \in I$, either $x$ has a finite orbit (that
is, $x$ is eventually periodic) or $\omega(x)$ is one of the
attractors above, and the number of attractors of type (2)-(4)
together is no more than the number of critical points (because
each of them must contain at least one critical point).
\end{proposition}

More can be said: there can be countably many disjoint cycles, but
all but finitely many of them must be disjoint from the basins of
periodic attractors. If a cycle  is solenoidal, then almost all
its orbits are attracted by the solenoidal set contained in the
cycle. If a cycle $\cyc(K)$ contains a dense orbit, then one of
the following two holds:
\begin{itemize}
 \item The whole cycle is an attractor and almost all its points
are dense in $\cyc(K)$. If $K$ has period $r$  this can still mean
that $K$ consists of two intervals $J$ and $J'$ with a common
boundary point such that $f^r(J) = J'$ and $f^r(J') = J$. In this
case $f^{2r}$ is topologically mixing on $J$. Otherwise $f^r$ is
topologically mixing on $K$.
 \item The cycle contains finitely
many attractors of type (4) attracting the orbits of almost all
points in the cycle (hence almost no point has a dense orbit in
the cycle).
\end{itemize}
An attractor of type (3) or (4) is called a \emph{solenoidal} and
a {\em wild attractor} respectively. Proposition~\ref{propAttr}
implies that the orbit of $\lambda$-a.e.\ $x\notin \AsPer$
accumulates on $\Crit$, so every solenoidal set is in fact a
solenoidal attractor. It is well known that attractors of type (1)
and (3) are uniquely ergodic, and it follows from \cite[Theorem
4]{BSS} that this is also true for attractors of type (4). The
existence of wild attractors has been proved for unimodal maps
only if the combinatorial properties of the map are very specific,
and the critical order $\ell_c$ is sufficiently large ($\ell_c \gg
2$). The prototype is the unimodal Fibonacci map \cite{BKNS}, but
there are other Fibonacci-like combinatorics that allow wild
attractors, see \cite{BTams}. For multimodal maps, there are
combinatorial types that allow wild attractors also if all
critical orders are $\ell_c = 2$ \cite{vS}.

\subsection{Distortion results}\label{subsec:distortion}
In what follows we denote by $|A|$ or $\lambda(A)$ the Lebesgue
measure of a measurable set $A\subset I$ (also, $\lambda_2$ will
denote the two-dimensional Lebesgue measure).  The
\emph{density} of a set $X$ in $J$ is $|X\cap J|/|J|$.
A point $x$ is a \emph{(Lebesgue) density point} of $X$ if
$\lim_{\eps \to 0} |X \cap (x-\eps, x+\eps)|/2\eps = 1$.

Many of the arguments in this paper rely on measuring images under
$f^n$ of neighborhoods $U$ of density points of certain sets. If
$f^n|_U$ is diffeomorphic then the Koebe Principle (see
Proposition~\ref{prop:koebe}) is used to estimate how densities
change, but in general, $U$ can visit several critical points in
its first $n$ iterates. In this case, we need more advanced
techniques and results (relying on work in \cite{BlMi,dMvS, vSV}),
which we summarize below in Theorems~\ref{thm:lemma11} and
\ref{thm:induced}.

We call a sequence $(G_i)_{i=0}^l$ of intervals  a \emph{chain} if
$G_i$ is a maximal interval such that $f(G_i)\subset G_{i+1}$,
$i=0,\ldots,l-1$. If  $(H_i)_{i=0}^l$ and $(G_i)_{i=0}^l$ are
chains and $H_i\subset G_i$ for every $i$, then we will write
$(H_i)_{i=0}^l\subset (G_i)_{i=0}^l$. If $x\in G_0$ (or $J$ is a
subinterval of $G_0$), then we call $(G_i)_{i=0}^l$, or sometimes
just the interval $G_0$, \emph{the pullback (chain) of $G_n$ along
$x,\ldots,f^l(x)$} (or \emph{along $J,\ldots,f^l(J)$}). The
\emph{order} of a chain is the number of intervals $G_i$, $0 \le
i<l$, intersecting $\Crit$.
\begin{remark}
 \label{rem:regular}
Under the hypothesis $f\in \CIInf(I)$, if  $G_l$ is a small
interval not too close to any attracting periodic orbit,
then all intervals $G_i$, $i<l$, are also very small by
Proposition~\ref{prop:contraction}. (Notice that the closure of the
set of attracting periodic points only contains periodic points
because the periods of attracting orbits are bounded and recall
that $\Crit$ only contain non-periodic points, see
assumption \eqref{nonperiodic}.)
\end{remark}

Given intervals $J \subset K$, we say that $J$ is {\em $\xi$-well
inside $K$} if the components $L$ and $R$ of $K \setminus J$
satisfy $|L|,|R|\geq \xi |J|$. If in addition $\xi |L|\leq |R|$
and $\xi |R|\leq |L|$, then we say that $J$ is {\em $\xi$-well
centered in $K$}.

A differentiable map without critical points $f:J\to
\R$ has \emph{distortion bounded by $\kappa>0$} if
\[
\sup_{x,y \in J} \frac{|f'(x)|}{|f'(y)|} \leq \kappa.
\]
We emphasize that if the density of a subset $X$ of $J$ is very
close to $1$, say $>1-\eps$, then the density of $f(X)$ in $f(J)$
is $>1-\kappa\eps$, so it is very close to $1$ as well.

An open subset $V$ of $\R$ is called \emph{nice} if $\orb(\partial V)\cap
V=\emptyset$.
The \emph{first entry map} to a nice set $V$ is the
map $\phi_V:D(V)\rightarrow V$ defined on the domain $D(V) = \cup_{n \geq 1} f^{-n}(V)$
by $\phi_V(x)=f^{r_V}(x)$, where $r_V = \min\{ k > 0 : f^k(x) \in V\}$
is the \emph{first entry time}. The maximal intervals $J$ on which
the first entry time $r_V$ is constant are called \emph{entry
domains}. By convention we assume that $J$ does not intersect $V$:
if $J$ is a  subset of a component of $V$ with first entry time
$r_V(J) \equiv r$, then we prefer to call $J$ a \emph{return
domain} with \emph{return time} $r$. The main reason why nice sets
are ``nice'' is that the return and entry domains are all
disjoint. Furthermore, all components in the backward orbit of a
nice set are nice, and two such intervals are either nested or
disjoint.

\begin{lemma} \label{lem:regularity}
 Let $f\in \CIInf(I)$. Then for every $\xi>0$ there exists
    $\xi' = \xi'(\xi, f)>0$ such that if
    $T$ is a  component of the preimage of an interval $V$
    and $U$ is an interval $\xi$-well inside $V$
    (respectively, $\xi$-well centered in $V$), then the preimage
    $J$ of $U$ in $T$ is $\xi'$-well inside $T$
    (respectively, $\xi'$-well centered in $T$).
%\item[(iii)] If in addition $T$
%    contains no critical points, then there is $\kappa = %\kappa(\xi, f)>0$
%    such that $f|_J$ has distortion bounded by $\kappa$.
\end{lemma}

\begin{proof}
 This follows easily from \cite[Lemmas~3.2 and 3.3]{BlMi}.
\end{proof}

The \emph{Schwarzian derivative} of a $C^3$ map $f$ is defined
for every $x \notin \Crit$ by
\[
Sf(x) = \frac{f'''(x)}{f'(x)} - \frac32\left(\frac{f''(x)}{f'(x)}\right)^2.
\]
A $C^1$ version reads: $f$ has negative Schwarzian derivative if
$1/\sqrt{|f'|}$ is convex on every interval where it is defined.
If $f$ has negative Schwarzian derivative, then we can frequently
estimate distortion using the Koebe Principle, see
 \cite[Section IV.1]{dMvS} or \cite[``Koebe lemma'']{BlMi}:

\begin{proposition}[Koebe Principle for negative Schwarzian derivative maps]
 \label{prop:koebenegative}
  If $f:G\rightarrow f(G)$ is a diffeomorphism with negative
  Schwarzian derivative, $H \subset G$ and $f(H)$ is
  $\xi$-well inside $f(G)$, then $H$ is $\xi^3/(2(3\xi+2)^2)$-well inside
  $G$ and $f|_G$ has distortion bounded by $((1+\xi)/\xi)^2$.
\end{proposition}

However, in this paper we will use also use a $C^2$-version of
these classical results.

\begin{proposition}[$C^2$ Koebe Principle]
 \label{prop:koebe}
Given $f\in \CIInf(I)$, there is a function $Q:(0,\infty) \to
(0,\infty)$ with $\lim_{\eps \to 0} Q(\eps) = 0$ such that the
following holds. Suppose that $H \subset G$ are intervals such
that $f^l|_G$ is a diffeomorphism and
 $f^l(H)$ is $\xi$-well-inside $f^l(G)$ for some $\xi > 0$.
Then there exists $\xi' = \xi'(\xi, f) > 0$ (with
$\xi' \to \infty$ as $\xi \to \infty$), and
\begin{equation}
\kappa = \exp\left( Q(\max_{0 \leq i < l} |f^i(G)|) \cdot
\sum_{i=0}^{l-1} |f^i(H)| \right) \cdot
\left(\frac{1+\xi}{\xi}\right)^2
\end{equation}
such that the distortion of $f^l|_H$ is bounded by $\kappa$ and
$H$ is $\xi'$-well inside $G$.
\end{proposition}

\begin{proof}
 See \cite[Proposition~2]{vSV}.
\end{proof}

We have an obvious bound $\sum_{i=0}^{l-1} |f^i(H)| \leq 1$ when
the $f^i(H)$, $0 \leq i < l$, are pairwise disjoint. This is what
we will use in the following corollary, which uses
Lemma~\ref{lem:regularity} and Proposition~\ref{prop:koebe} as
well.

\begin{corollary}
 \label{cor:disjointness}
Let $f\in \CIInf(I)$. Then for any $\xi>0$ and $k\geq 0$, there
are $\xi'=\xi'(\xi,k,f)>0$, $\sigma=\sigma(\xi,k,f)>0$  such that
the following statement holds: Let $(H_i)_{i=0}^l\subset
(G_i)_{i=0}^l$ be chains such that $(G_i)$ has order at most $k$,
$G_l$ is a small interval close enough to $\Crit$, and the
intervals $H_i$ are pairwise disjoint. If $H_l$ is $\xi$-well
inside $G_l$, then  $H_0$ is $\xi'$-well inside $G_0$. If in
addition $k=0$, then there is $\kappa=\kappa(\xi,f)>0$ such that
$f^l|_{H_0}$ has distortion bounded by $\kappa$.
\end{corollary}

\begin{proof}
By saying that ``$G_l$ is a small interval close enough $\Crit$''
we mean that there is $\eps_0=\eps_0(f)$ such that  $|G_l|<\eps_0$
and $\dist(G_l,c)<\eps_0$ for some non-periodic critical point
$c$, where $\eps_0$ is chosen so that, if $\eps_1$  satisfies
$Q(\eps)<1$ for every $\eps\leq \eps_1$ in
Proposition~\ref{prop:koebe}, then  $|G_i|<\eps_1$ for every $i<l$
(see Remark~\ref{rem:regular}).

The case $k=0$ is just Proposition~~\ref{prop:koebe}. Let us give
the proof for $k=1$; the idea is the same for $k
> 1$. Let  $G_t$ be the interval from the chain containing the
critical point. Now we construct three subchains,
$(H_i)_{i=0}^t\subset (G_i)_{i=0}^t$, $(H_i)_{i=t}^{t+1}\subset
(G_i)_{i=t}^{t+1}$ and $(H_i)_{i=t+1}^l\subset (G_i)_{i=t+1}^l$.
Applying Proposition~\ref{prop:koebe} to the third chain, we find
$\xi_1=\xi_1(\xi,f)$ and $\kappa_1=\kappa_1(\xi,f)$  such that
$H_{t+1}$ is $\xi_1$-well inside $G_{t+1}$. Applying
Lemma~\ref{lem:regularity} to the middle chain, we find
$\xi_2=\xi_2(\xi_1,f)$ such that $H_t$ is $\xi_2$-well inside
$G_t$. Finally, applying again Proposition~\ref{prop:koebe} to the
first chain, we find $\xi=\xi(\xi_2,f)$ such that $H_0$ is
$\xi$-well inside $G_0$.
\end{proof}

The Koebe property refers to distortion control in the presence of
Koebe space. Slightly weaker is the Macroscopic Koebe property,
which refers to the preservation of Koebe space under pullback.
Hence the fact that $H$ is $\xi'$-well inside $G$ in
Proposition~\ref{prop:koebe} is basically a Macroscopic Koebe
statement.

In order to use the above results, we need conditions guaranteeing
the existence of Koebe space at the end of chains. The following
propositions are particularly useful in this regard.

\begin{proposition}
 \label{macroscopickoebe}
 Let $f\in \CIInf(I)$.
 Then for every $\xi>0$ there exists $\xi'=\xi'(\xi,f)>0$ such that if
 $V$  and $U$ are nice intervals,
 $U$ is $\xi$-well inside $V$,
 $x\in V$
 and $f^k(x)\in U$ (with $k\geq 1$ not necessarily minimal),
 then the pullback of $U$ along $x,\ldots,f^k(x)$ is
 $\xi'$-well inside the return domain to $V$ containing $x$.

 In particular, if $U$ is a return domain
 to $V$ which is $\xi$-well inside $V$, then
 all return domains to $U$ are $\xi'$-well inside
 $U$.
\end{proposition}

\begin{proof}
This is Theorem~C(1) from \cite{vSV} (see also the remark below Theorem C(1)
and the erratum to that paper).
The
second statement follows easily from the first one, by fixing an
arbitrary return domain $K$ to $U$ and $x\in K$, and choosing $k$
as the return time of $K$. The interval $K$ is then the pullback
of $U$ along $x,\ldots,f^k(x)$ and $U$ is the return domain to $V$
containing $x$.
\end{proof}

\begin{proposition} \label{prop:well-inside}
Let $f\in \CIInf(I)$ and let $x$ be a recurrent point of $f$ which
is neither periodic nor of Feigenbaum type. Then there are
$\xi_0=\xi_0(f)>0$ and an arbitrarily small nice neighborhood $J$
of $x$ such that the return domain to $J$ containing $x$ is
$\xi_0$-well inside $J$. Assume in addition that $x$ is not solenoidal,
and that $I_0$ is a nice neighborhood of $x$ so small that it
contains no periodic neighborhood of $x$. Let $(I_m)_{m=0}^\infty$
be the sequence of  nice intervals such that $I_m$ the return
domain to $I_{m-1}$ containing $x$. In this case, there are
infinitely many $m$ such that $I_{m+1}$ is $\xi_0$-well inside
$I_m$.
\end{proposition}

\begin{proof}
 This is a mixture of Theorems~A(1) and A'(2) from \cite{vSV}.
\end{proof}

\subsection{Notions from ergodic theory}\label{subsec:ergodic}

Let $X$ be a topological space with  Borel $\sigma$-algebra $\B$
and let $f:X \to X$ be a Borel measurable map. Recall that a
probability measure $\mu$ on $\B$ is called \emph{invariant}
(respectively, \emph{non-singular}) if $\mu(f^{-1}(A))=\mu(A)$
(respectively, $\mu(A)=0$ if and only if $\mu(f^{-1}(A))=0$) for
any $A\in \B$. In what follows we assume that $\mu$ is
non-singular but not necessarily invariant.

\begin{definition}
 Let $(X,\mu,f)$ be defined as above. Write $X^2 := X
 \times X$, $\mu_2 := \mu \times \mu$, $f_2 := f \times f$ and let
 $\B_2$ be the Borel $\sigma$-algebra in $X\times X$. We
 call the system $(X, \mu, f)$
 \begin{itemize}
  \item \emph{conservative} if for every set $A \in \B$ with $\mu(A) > 0$,
    there is $n \geq 0$ such that $\mu(A \cap f^n(A)) > 0$;
  \item \emph{ergodic} if $f^{-1}(A) = A \in \B$ implies $\mu(A) = 0$ or
    $1$. If $f$ is ergodic and conservative, then $\mu$-a.e. orbit is dense in $\supp(\mu)$.
  \item \emph{exact} if $f^{-n}(f^n(A)) = A \in \B$ for every
    $n \geq 0$ implies $\mu(A) = 0$ or $1$;
   \item \emph{mixing} if $\mu$ is invariant and
    $$
     \lim_{n\to \infty} \mu(A \cap f^{-n}(B)) =\mu(A)\mu(B)
    $$
    for all sets $A,B \in \B$;
  \item \emph{weak mixing}\label{def:weak_mixing} if $\mu$ is invariant,
    but $1$ is the only
    eigenvalue corresponding to a measurable eigenfunction of the
    operator $P:L^1(\mu) \to L^1(\mu)$ defined by $P\psi = \psi \circ
    f^{-1}$. Weak mixing invariant measures have, in fact, several
    equivalent definitions. One of them is that the Cartesian product
    $(X \times Y, \mu \times \nu,f \times g)$ is ergodic for every
    ergodic measure preserving system $(Y,g,\nu)$. (In particular,
    weak mixing implies that $(X^2,\mu_2,f_2)$ is ergodic.)
  \item For non-invariant measures, we can still speak of mildly mixing:
    A non-singular probability measure is \emph{mildly mixing} if for
    every set $A \in \B$ of positive measure
    $$
     \liminf_{n\to \infty} \mu(A \cap f^{-n}(A)) > 0.
    $$
    Mild mixing implies that $(X^2,\mu_2,f_2)$ is ergodic. (If $f$ is
    invertible, then mild mixing is equivalent to
    $(X^2,\mu_2,f_2)$ being ergodic. In this case $f$ also preserves a
    probability
    measure equivalent to $\mu$, but neither of these stronger
    statements holds in general if $f$ is non-invertible. See
   \cite{HawSil} for more results.)
 \end{itemize}
\end{definition}

\begin{lemma} \label{exact-ergodic}
 If $(X,\mu,f)$ is exact, then $(X^2,\mu_2,f_2)$ is ergodic.
\end{lemma}

\begin{proof}
 Assume by contradiction that $(X^2,\mu_2,f_2)$ is not
 ergodic, so there is $U \in \B_2$ such that $f_2^{-1}(U) = U$ and $0
 < \mu_2(U) < 1$. Then there is $a \in X$ such that $0 < \mu(U_a) <
 1$ for $U_a = \{ y \in X : (a,y)  \in U\}$. Let $V_a = X \setminus
 U_a$. Then $f^n(U_a) \cap f^n(V_a) = \emptyset$ for all $n \geq
 0$, so $(f^{-n} \circ f^n)(U_a) = U_a$ for all $n$, contradicting
 that $\mu$ is exact.
\end{proof}

\begin{lemma}\label{lem:stronger}
 If $X$ is separable and $(X,\mu,f)$
 is ergodic and conservative, then $\mu$-a.e.
 $x$ has a dense orbit in $X$.
\end{lemma}

\begin{proof}
 Let $\{ U_n \}_{n \in \N} \subset \B$ be a countable basis of
 $X$, and set $Y_n := \{ x \in X : f^k(x) \in U_n \text{
 infinitely often}\}$. Then $f^{-1}(Y_n) = Y_n$, so $\mu(Y_n) =
 0$ or $1$ for each $n$. If $\mu(Y_n) = 1$ for each $n$, then $Y
 := \bigcap_n Y_n$ has full measure, and every $x \in Y$ has a dense
 orbit.

 So assume that $n$ is such that $\mu(Y_n) = 0$. For $Z = X
 \setminus Y_n$, we can write $Z = \bigcup_{k \in \N} Z_k$ where $Z_k
 = \{ x \in X : k = \max\{ i : f^i(z) \in U_n\} \}$. Since
 $\mu(Z) = 1$, there is $k$ such that $\mu(Z_k) > 0$, and hence
 $\mu(f^k(Z_k)) > 0$. But then $f^k(Z_k) \subset U_n$ is a set of
 positive measure that never visits $U_n$ again, and this
 contradicts that $\mu$ is conservative.
\end{proof}

The following diagram summarizes the implications between these various notions
of mixing. The notions on the bottom line can be defined for
non-invariant measures $\mu$, but any implication to the top line requires
that $\mu$ is $f$-invariant and in particular conservative.

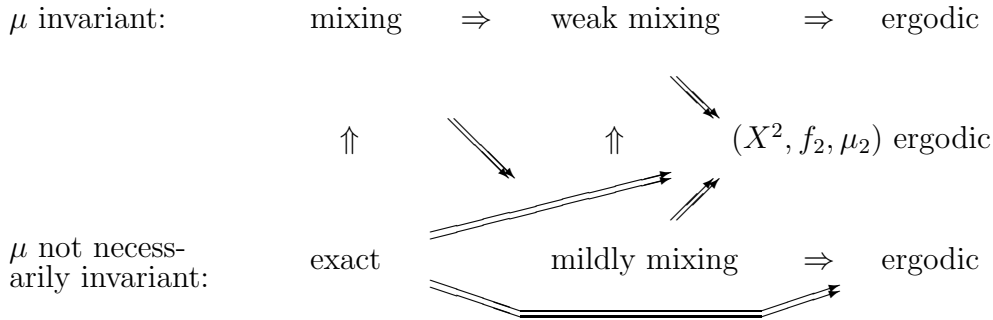
\begin{figure}[ht]
\unitlength=8mm
\begin{picture}(16,6)(0,0)
\put(-1,5){$\mu$ invariant:}
\put(-1,1.2){$\mu$ not necess-}\put(-1,0.7){arily invariant:}
%\put(-1,0.5){but conservative}
\put(4,5){mixing} \put(6.5,5){$\Rightarrow$} \put(8,5){weak mixing}
 \put(12.2,5){$\Rightarrow$} \put(13.5,5){ergodic}

\put(4,1){exact} %\put(6.5,1){$\Rightarrow$}
\put(8,1){mildly mixing}
\put(12.2,1){$\Rightarrow$} \put(13.5,1){ergodic}

 \put(8.9,3){$\Uparrow$} \put(4.5,3){$\Uparrow$}
 \put(10,1.8){\vector(1,1){0.7}}  \put(10.1,1.8){\vector(1,1){0.7}}
\put(6,1.6){\vector(4,1){4}}\put(6,1.5){\vector(4,1){4}}
\put(6.4,3.5){\vector(1,-1){1}}\put(6.3,3.5){\vector(1,-1){1}}
 \put(10,4.2){\vector(1,-1){0.7}} \put(10.1,4.2){\vector(1,-1){0.7}}
 \put(11,3){$(X^2,f_2, \mu_2)$ ergodic}

 \put(6,0.8){\line(3,-1){1.5}}  \put(7.5,0.3){\line(1,0){4}}
 \put(11.5,0.3){\vector(3,1){1.3}}
 \put(6,0.7){\line(3,-1){1.5}}  \put(7.5,0.2){\line(1,0){4}}
 \put(11.5,0.2){\vector(3,1){1.3}}
\end{picture}
%\caption{Implications between various notions of mixing.}
\label{fig:implications}
\end{figure}

\section{Inducing to critical neighborhoods and strongly wandering
sets}\label{sec:strongly_wandering}

In this section, we prove that strongly wandering sets have zero
measure as a consequence of Theorem~\ref{thm:lemma11} below.
Theorem~\ref{thm:lemma11} is a refinement of Theorem~1 of
\cite{caili} (where the slightly stronger hypothesis $f\in
\CIIInf(I)$ is used), which in turns improves Theorem~D of
\cite{vSV}.

\begin{theorem} \label{thm:lemma11}
 Let $f\in \CIInf(I)$.
 Then there are positive constants $\xi=\xi(f)$, $\kappa=\kappa(f)$,
 $\delta=\delta(f)$ and, for every $\eps>0$ and $c \in \Crit$, open
 intervals $c\in U_c\subset V_c$, with $|V_c|<\eps$,
 such that the following conditions hold:
 \begin{itemize}
  \item[(i)] The set $U=\bigcup_{c\in
    \Crit} U_c$ is nice and $U_c$ is $\xi$-well inside $V_c$
    for every $c\in \Crit$.
  \item[(ii)] If  $J$ is an entry domain of the first entry map
    $\phi$ to $U$, say $\phi|_J=f^j|_J$ and $\phi(J)=U_c$, then
    there is $K\supset J$ such that $f^j|_K$ is a diffeomorphism
    and $f^j(K)=V_c$. Moreover, $\phi|_J$ has distortion bounded by $\kappa$.
  \item[(iii)] If $c\in \Crit$,
    then there are an open interval
    $W_c\subset U_c$ and $k_c \in \N$ such that $|W_c|>\delta|U_c|$,
    $f^{k_c}|_{W_c}$ is a diffeomorphism with distortion bounded by $\kappa$, and
\begin{itemize}
\item either $W_c\subset
    f^{k_c}(W_c)\subset U_c$ (if $c$ is of Feigenbaum type);
\item or
    $k_c=1$ and $\partial W_c\cap D(U)=\emptyset$
    (if $c$ is not of Feigenbaum type).
 \end{itemize}
 \end{itemize}
\end{theorem}

The existence of the sets $U_c \subset V_c$ from
Theorem~\ref{thm:lemma11} suggests the construction of an induced
Markov map $F:U \to U$ with all branches mapping monotonically
onto a component of $U$. The first, very detailed, constructions
of such induced maps go back to Jakobson \cite{Jak}, but the
abstract statement for unimodal maps having no periodic or Cantor
attractors comes from Martens' PhD. thesis \cite{martens}. During
the writing of this paper, we learnt about a multimodal $C^3$
version by Cai and Li \cite{caili}, see Proposition~\ref{prop:caili}.
Their construction precludes
the existence of parabolic points, hence we improve it slightly by
showing the version below, where $\Crit'$ denotes the set of
critical points  interior to metric attractors of type (2) in
Proposition~\ref{propAttr}.

Let $E := \bigcup_{n \geq 0} f^{-n}(A)$ where $A$ is the union of
all attractors of type (2) in Proposition~\ref{propAttr}, that is,
the union of cycles of intervals in which $\lambda$-a.e.\ orbit is
dense.

\begin{theorem}\label{thm:induced}
Let $f \in \CIIInf(I)$ and let $U_c,V_c$, $c\in \Crit'$, be
defined as in Theorem~\ref{thm:lemma11} for $\eps$ sufficiently
small and let $U'=\bigcup_{c\in \Crit'} U_c$. Then for
Lebesgue a.e.\ $x\in E$ we can find $k_x \in \N$ and intervals
$G_x \supset H_x \owns x$ (with either $H_x=H_y$ or $H_x\cap
H_y=\emptyset$) such that $f^{k_x}:G_x \to V_c$ is diffeomorphic
for some $c \in \Crit'$ and $f^{k_x}(H_x) = U_c$.

Let $F:U' \to U'$ be defined
by $F|_{H_x} = f^{k_x}:H_x \to U_c$ for the appropriate $c \in
\Crit'$. Then all iterates of $F$ is well defined $\lambda$-a.e \ and its
branches have uniformly bounded distortion, \ie there is
$\kappa=\kappa(f)>0$ such that for every $n \in \N$ and every interval $H$
on $F^n|_H: H\subset U\to U_c$ is a diffeomorphism, the
distortion of  $F^n|_H$ is bounded by $\kappa$.
\end{theorem}

%It is here and in Proposition~\ref{prop:limsup_full_equiv} (to
%guarantee that wild attractors have measure zero) that we need the
%$C^3$ assumption. We expect that these result are true in the
%$C^2$ setting as well, but we have no proof.

We postpone the somewhat technical proofs of
Theorems~\ref{thm:lemma11} and \ref{thm:induced} to the appendix.

\begin{proof}[\textbf{Proof of Theorem~\ref{mainthm:Strong}}]
Let $f\in \CIInf(I)$ and assume that $W$ is a strongly wandering
set of positive measure. Let $x$ be a density point of $W$. It is
not restrictive to assume that $x$ belongs to the interior of a
cycle $\cyc(K)$ which either contains a dense orbit \iffalse (and
possibly $K = \cyc(K) = I$)\fi or is of solenoidal type. We can
also assume that $\orb(x)$ accumulates on $\Crit$.

Let $\tau>0$ be such that $(x-\tau,x+\tau)\subset \cyc(K)$. We can
assume that the density of $W$ in any subinterval of
$(x-\tau,x+\tau)$ containing $x$ is larger than $1-\eps_0$, with
$\eps_0>0$ so small that if $d=\#\Crit$, then
$1-\delta^{-d}\kappa^{2d+1}\eps_0>1/2$, where $\delta=\delta(f)$
and $\kappa=\kappa(f)$ are the numbers from
Theorem~\ref{thm:lemma11}. Finally we fix $\eps>0$ such that if
$U$ is defined as in Theorem~\ref{thm:lemma11}, then all
components of $U$ and entry domains to $U$ have length less than
$\tau$ (by the non-Contraction Principle).

Let $J_1$ be either the entry domain to $U$  containing $x$, say
$\phi|_{J_1}=f^{j_1}|_{J_1}$, $f^{j_1}(J_1)=U_{c_1}$, or the
component $U_{c_1}$ of $U$ containing $x$ (then we take $j_1=0$).
Since $J_1\subset (x-\tau,x+\tau)\subset \cyc(K)$ and $\cyc(K)$ is
invariant,  the density of $W$ in $J_1$ is $>1-\eps_0$ and
$U_{c_1}\subset \cyc(K)$.

Let $W_{c_1}$ be the interval from Theorem~\ref{thm:lemma11}, part
(iii). Since $|W_{c_1}| > \delta |U_{c_1}|$ and the density of
$f^{j_1}(W)$ in $U_{c_1}$ is $>1-\kappa\eps_0$, the density of
$f^{j_1}(W)$ in $W_{c_1}$ is $>1-\delta^{-1}\kappa\eps_0$ and the
density of $f^{j_1+k_1}(W)$ in $f^{k_1}(W_{c_1})$ is
$>1-\delta^{-1}\kappa^2\eps_0$, where $k_1 = k_{c_1}$ is as in
Theorem~\ref{thm:lemma11}(iii). If $c_1$ is a Feigenbaum critical
point, then $W_{c_1} \subset f^{k_{c_1}}(W_{c_1}) \subset
U_{c_1}$, so the density of $f^{j_1}(W)$ in $f^{k_{c_1}}(W_{c_1})$
is also $>1-\delta^{-1}\kappa\eps_0$. Therefore the densities of
both $f^{j_1}(W)$ and $f^{j_1+k_1}(W)$ in $f^{k_{c_1}}(W_{c_1})$
are $>1/2$, hence  $f^{j_1}(W) \cap f^{j_1+k_1}(W)\neq \emptyset$,
contradicting that $W$ is strongly wandering.

Now we deal with the non-Feigenbaum case. Here $k_1=1$,
$f|_{W_{c_1}}$ is a diffeomorphism and $\partial W_{c_1}\cap
U=\emptyset$. On the other hand, $U_{c_1}\subset \cyc(K)$, so
$f(W_{c_1})\subset \cyc(K)$ as well. Moreover, the choice of
$\cyc(K)$ guarantees that the orbits of almost all its points
accumulate on $\Crit$. Hence $f(W_{c_1})$ is the pairwise disjoint
union (up to a measure zero set) of entry domains to $U$ and
components of $U$. Since the density of $f^{j_1+1}(W)$ in
$f(W_{c_1})$ is $>1-\delta^{-1}\kappa^2\eps_0$, it is also
$>1-\delta^{-1}\kappa^2\eps_0$ in one of these entry domains or
components, call it $J_2$.

Now we repeat the argument. Say that $\phi|_{J_2}=f^{j_2}|_{J_2}$,
$f^{j_2}(J_2)=U_{c_2}$, or $J_2$ is a component $U_{c_2}$ of $U$
(when we take $j_2=0$). Then $f^{j_1+j_2+1}(W)$ has density
$>1-\delta^{-2}\kappa^3\eps_0$ in $U_{c_2}$. If $c_2$ is a
critical point of Feigenbaum type, then we get a contradiction as
before. If not, then we find an entry domain to $U$ or a component
of $U$, call it $J_3$, such that $f^{j_1+j_2+2}(W)$ has density
$>1-\delta^{-2}\kappa^4\eps_0$ in $J_3$. Proceeding in this way,
we find intervals $J_1,J_2,\ldots, J_{d+1}$, each of them either
an entry domain to $U$ or a component of $U$, say
$f^{j_i}(J_i)=U_{c_i}$, such that the density of $f^{t_i}(W)$ in
$U_{c_i}$ is $>1-\delta^{-(i-1)}\kappa^{2i-1}\eps_0$ for every
$i=1,2,\ldots, d+1$, with $t_i=j_1+\cdots+j_i+i-1$. Find $i<i'$
such that $c_{i}=c_{i'}$. Then the densities of $f^{t_i}(W)$ and
$f^{t_{i'}}(W)$ in $U_{c_i}$ are $>1/2$, and we have the required
contradiction.
\end{proof}

\section{$\limsup$ fullness of Lebesgue measure}\label{sec:limsupfull}

The following definition was first used in Barnes \cite{Bar}.

\begin{definition}\label{def:limsup_full}
The system $(X,\mu,f)$  is called \emph{ $\limsup$ full} if
$\limsup_{n\to \infty} \mu(f^n(A))=1$ for any $A \in {\mathcal B}$
with $\mu(A) > 0$.
\end{definition}

\begin{theorem}\label{thm:limsup_full}
 Let $f\in \CIIInf(I)$ be a topologically mixing map having no
 Cantor attractors. Then $f$ is $\limsup$ full with respect to
 Lebesgue measure.
\end{theorem}

\begin{proof}
It suffices to show that if $f\in \CIIInf(I)$ (and satisfying
(\ref{eq:nicely_scaled})) has an invariant interval $I'$ such that
$f|_{I'}$ is topologically mixing and has no Cantor attractors,
then we have $\limsup_n \lambda(f^n(A))=\lambda(I')$ for any
measurable set $A\subset I'$ of positive measure.

Fix a critical point $c$ interior to $I'$ and let $\eps>0$ be
small enough so that Theorem~\ref{thm:induced} holds and the
corresponding interval $U_c$ lies in $I'$. Let $x$ be a density
point of $A$. Since $f$ has no Cantor attractors in $I'$, there is
no loss of generality in assuming $\omega(x) = I'$ and (after
replacing if necessary $A$ by some of its iterates)  $x\in U_c$.
Since all iterates of the induced map $F$ from
Theorem~\ref{thm:induced} are well defined for $\lambda$-a.e.\
point in $U_c$, we can assume that this is the case for $x$. This
implies that there are intervals $x\in H_n\subset U_c$ with
$\bigcap_n H_n = \{ x \}$ and integers $k_n$ such that
$f^{k_n}|_{H_n}:H_n \to U_{c'}$ are diffeomorphisms with uniform
distortion bound for some $c' \in \Crit'$.  Hence $\lambda(f^{k_n}
(A \cap H_n) ) \to \lambda(U_{c'})$ as $n \to \infty$. Observe
that $U_{c'}\subset I'$, so there is $j$ such that $I' =
f^j(U_{c'})$ (because $f|_{I'}$ is topologically mixing). Then
also $\lambda(f^{k_n+j}(A \cap H_n)) \to \lambda(I')$, as
required.
\end{proof}

\begin{proposition}\label{prop:limsup_full_equiv}
 Let $f\in \CIIInf(I)$ be topologically mixing.
 Then the following statements are equivalent:
 \begin{itemize}
   \item[(i)] $f$ has an acip $\mu$;
   \item[(ii)] $\liminf_n \lambda(f^n(A))>0$
     for every measurable set $A$ of positive Lebesgue measure;
   \item[(iii)] $\liminf_n \lambda(f^n(A))=1$
     for every measurable set $A$ of positive Lebesgue measure.
 \end{itemize}
 In this case $\mu$ is equivalent to
 $\lambda$ (that is, $\mu(A)$ if and only if $\lambda(A)=0$) and
 $\lambda_2$ is ergodic and conservative.
\end{proposition}

\begin{remark}
Under $C^3$ assumptions, no transitive Cantor set can have
positive measure, \cite[Theorem~E(1), cf. also Remark 1]{vSV}, so
in particular a Cantor attractor has Lebesgue measure $0$, and
cannot support an acip. We expect this to be true in the $C^2$
setting as well, but we have no proof.
\end{remark}

\begin{proof}
 The implication (iii)$\Rightarrow$(ii) is trivial.

 We prove (ii)$\Rightarrow$(i). According to \cite{Str}, the
 existence of an acip for $f$ is
 equivalent to the existence of $\delta>0$ and $0<\alpha<1$ such
 that $\lambda(A)<\delta$ implies $\lambda(f^{-n}(A))<\alpha$ for
 every $n$ and $A$.
Hence, if $f$ does not admit an acip, then there
 are sets $\{ A_k \}$ and integers $(n_k)_{k \geq 1}$ with $n_k \to \infty$,
 $\lambda(A_k)\to 0$ and $\lambda(f^{-n_k}(A_k))\geq 1-2^{-k-1}$.
Take $A = \bigcap_k f^{-n_k}(A_k)$, then $\lambda(A) \geq \frac12$
and $\liminf_n \lambda(f^n(A)) = 0$, contrary to condition (ii).

For the implication (i)$\Rightarrow$(iii), assume that $f$ admits
an acip $\mu$. Since (wild) Cantor attractors have zero Lebesgue
measure, $\mu$ cannot be supported on them. Then
Theorem~\ref{thm:limsup_full} applies.

 We claim that $\mu$ is equivalent is to $\lambda$. Assume the
 contrary to find a measurable set $A$ such that
 $\mu(A)=0<\lambda(A)$. Let $B= \bigcup_{n=0}^{\infty} f^{-n}(A)$.
 Then $\mu(B)=0$ but, by  Theorem~\ref{thm:limsup_full},
 $\lambda(B)=1$. This is impossible.

 The equivalence of $\mu$ and $\lambda$ (or just the absolute
 continuity of $\mu$)  and  Theorem~\ref{thm:limsup_full}
 imply that if $\mu(A)>0$, then $\limsup_n \mu(f^n(A))=1$. Since
 $\mu$ is invariant, the sequence $\{\mu(f^n(A))\}_{n \in \N}$ is
 non-decreasing. Therefore $\lim_n \mu(f^n(A))=1$, and
 also $\lim_n \lambda(f^n(A))=1$ by the equivalence of $\mu$
 and $\lambda$. Moreover, $\lambda$ is exact, hence $\lambda_2$ is ergodic
 (Lemma~\ref{exact-ergodic}), and $\mu_2$ is conservative (because
 of Poincar\'e recurrence), so $\lambda_2$ is conservative as well
 (because $\mu_2$ and $\lambda_2$ are equivalent).
\end{proof}

\iffalse
\begin{remark}
A question in this context is: Are there topologically mixing interval maps
without Cantor attractors such that almost no pair $(x,y)$ has a
dense orbit in $I^2$?
\end{remark}
\fi

\section{Li-Yorke pairs and scrambled sets}\label{sec:scrambled}

In this section we prove our results on scrambled sets and the
$\lambda_2$-measure  of the sets $\Dis$, $\Asymp$ and $\LY$ in the case
that $f$ had no wild attractor.

\begin{proof}[\textbf{Proof of Theorem~\ref{mainthmC2}}]
 Assume that a scrambled set $S$ has positive measure. We can
 assume that all its points are attracted by the same attractor,
 which must be either a cycle of intervals containing a dense orbit,
 or a minimal Cantor set. Since $f^n$ is one-to-one on $S$ for every $n$,
 the first possibility can be immediately discarded
 by  Theorem~\ref{thm:limsup_full}. Hence we may assume that all points
 of $S$ are attracted by a minimal Cantor set $W$. (In fact
 $W$ must be a wild attractor, because points
 attracted by solenoidal sets are approximately periodic and a
 scrambled set can contain at most one approximately periodic
 point by Proposition~\ref{prop:approximately}.)

 By Theorem~\ref{mainthm:Strong} there are integers $n<m$ such that
 $f^n(S)\cap f^m(S)\neq \emptyset$. Let $x\in f^n(S)\cap
 f^m(S)$. Then there is $y\in f^n(S)$ such that $f^{m-n}(y)=x$.
 Since $(x,y)$ is a Li-Yorke pair, so is $(f^{m-n}(x),
 f^{m-n}(y))=(f^{m-n}(x),x)$. Find a sequence $\{l_k\}$ such that
 $|f^{m-n+l_k}(x)-f^{l_k}(x)|\to 0$. We may assume that
 $(f^{l_k}(x))_{k \in \N}$ accumulates at $p\in W$. Then $f^{m-n}(p)=p$,
 which is impossible because $W$ is infinite and minimal.
\end{proof}

\begin{proposition}\label{prop:noinvscramb}
A multimodal map $f$ has no closed invariant scrambled
set (apart from a singleton set).
\end{proposition}

\begin{proof}
Suppose by contradiction that the closed non-singleton $S$ is
invariant and scrambled. Then clearly it can contain at most one
fixed point. Let $y \in S$ be a non-fixed point. Then because
$(y,f(y))$ is Li-Yorke, there is a sequence $(n_k)_{k \in \N}$
such that $\lim_k |f^{n_k}(y)- f^{n_k}(f(y))| = 0$ and $\{
f^{n_k}(y) \}$ converges. By continuity $\lim_k f^{n_k}(y) =
\lim_k f^{n_k+1}(y) = f(\lim_k f^{n_k}(y))$, so the limit is a
fixed point $p\in S$. Since $(y,f(y))$ is Li-Yorke, $\{f^n(y)\}$
does not converge to $p$.

Since $f$ is multimodal, there are a sequence $(m_k)$ and a number
$\eps>0$ such that $f^{m_k}(y) \to p$ but $|f^{m_k-1}(y)-p|>\eps$.
By taking a subsequence, we can assume that $\{ f^{m_k-1}(y)\}_{k
\in \N}$ converges, say to $q$. But then $p = \lim_k f^{m_k}(y) =
f(\lim_k f^{m_k}(y))) = f(q)$, so $(p,q)$ is not Li-Yorke. This
contradiction proves the proposition.
\end{proof}

\iffalse
\begin{proposition}
 \label{prop:Asymp}
If $f\in \CIInf(I)$, then
 $\lambda_2(\Asymp \setminus (\AsPer\times \AsPer)) = 0$.
\end{proposition}
\fi

\begin{proof}[\textbf{Proof of Theorem~\ref{mainthm:Asymp}}]
 Assume that $\lambda_2(\Asymp \setminus (\AsPer\times \AsPer)) > 0$.
Then there are a point
 $x\in I\setminus \AsPer$ and a Borel set $Y\subset I\setminus
 \AsPer$ of
 positive measure such that $\lim_{n\to \infty}
 |f^n(x)-f^n(y)|=0$ for every $y\in Y$. We can assume that $\orb(x)$
 accumulates on a non-periodic point $u$.

 For every $n \geq 0$, let $d_n:Y\to \R$ be defined by
 $d_n(y)=\sup_{m\geq n} |f^m(x)-f^m(y)|$. Then $(d_n)$ is a
 sequence of Borel measurable maps converging pointwise to zero.
 According to Egorov's theorem we can remove from $Y$ a small set
 (so that the remaining set $Z$ has positive measure) in such a way
 that $\{ d_n|_Z \}_{n \in \N}$ converges uniformly to zero, that is,
 $\diam(f^n(Z))\to 0$ as $n\to \infty$.

Use Theorem~\ref{mainthm:Strong} to find integers $k>m$ such that
 $f^k(Z)\cap f^m(Z)\neq \emptyset$. Recall that $\orb(Z)$ accumulates at
 a point $u$ with $|f^j(u)-u|=\eps>0$ for $j=k-m$. Find
$\delta \in (0, \eps/4)$ such that
 $|f^j(v)-f^j(w)|<\eps/2$ whenever $|v-w|<2\delta$. Next take
 $l>m$ so that $\dist(f^l(Z),u)<\delta$ and $\diam(f^l(Z))<\delta$.
 Then $|f^l(z)-u)|<2\delta<\eps/2$, hence
 $|f^{l+j}(z)-f^j(u)|<\eps/2$ for every $z\in Z$. Thus
 $f^{l+j}(Z)\cap f^l(Z)=\emptyset$, contradicting $f^{m+j}(Z)\cap
 f^m(Z)\neq \emptyset$ and $l>m$.
\end{proof}

\begin{proposition}\label{PropExactDis}
 If $(I,\mu,f)$ is exact, then $\mu_2(\Dis)=0$.
\end{proposition}

\begin{proof}
 Assume by contradiction that $\mu_2(\Dis) > 0$. Write $\Dis_x =
 \{ y \in I : (x,y) \text{ is distal}\}$ and $G_{\Dis} := \{ x :
 \mu(\Dis_x) > 0\}$. Then by Fubini's Theorem,
 $\mu(G_{\Dis}) > 0$. For $x \in G_{\Dis}$, take
 $$
  Y_\eps := \{ y \in \Dis_x : \liminf_{n\to\infty} |f^n(x)-f^n(y)| \geq  \eps\}.
 $$
 Clearly $Y_\delta \subset Y_\eps$ if $\delta > \eps$ and
 $f^n(Y_\eps) \cap f^n(I \setminus Y_\eps) = \emptyset$ for all $n
 \geq 0$. If for some $\eps > 0$, both $Y_\eps$ and $I \setminus
 Y_\eps$ have positive measure, then we have a contradiction to
 exactness. The remaining possibility is that there is
 $\eps = \eps(x) > 0$ such that $Y_\eps$ has full measure, whereas
 $\lambda(Y_\delta) = 0$ for all $\delta > \eps$. Now take $\eta >
 0$ such that $G_\eta := \{ x \in G_{\Dis} : \eps(x) > \eta\}$
 has positive measure. Let $R > 1/\eta$, and take distinct
 points $x_0,x_1, \dots, x_R \in G_\eta$. Clearly these points can be chosen
 so that $\liminf_n  |f^n(x_i)-f^n(x_j)|>\eta$
 for all $i \neq j$. Take $N$ minimal such that $|f^n(x_i)-
 f^n(x_j)|> \eta$ for all $i \neq j$ and $n \geq N$. However, by
 the choice of $R$, there is no space in $I$ to fit the points
 $f^N(x_i)$, $i=0,\dots, R$, so that they have pairwise
 distance greater than $\eta$. This contradiction proves that $\mu_2(\Dis)
 = 0$.
\end{proof}

\begin{proof}[\textbf{Proof of Theorem~\ref{mainthmLY}}]
Theorem~\ref{thm:limsup_full} implies that $f$ is $\limsup$ full
w.r.t.\ Lebesgue measure. By \cite[Theorem A]{Bar}, if  $f:I\to I$
is a multimodal surjective map that is $\limsup$ full with respect
to a measure $\mu$, then $f$ is exact. (The proof is stated for
$d$-to-$1$ maps, but applies with minor changes to the ``at most
$d$-to-$1$'' setting as well.) Hence $\lambda$ is exact.  By
Lemma~\ref{exact-ergodic}, $\lambda_2$ is  ergodic for $(I^2,
f_2)$.

Take $x\in I$. We prove that there is a full measure set $A_x$
such that $\limsup_{n\to\infty} |f^n(y)- f^n(x)|\geq \diam(I)/2$
for every $y\in A_x$. If the opposite is true, then there are a
set $A$ of positive measure and an $m \in \N$ such that $|f^n(y)-
f^n(x)|< \diam(I)/2$ for every $y\in A$ and every $n\geq m$. But
this contradicts the conclusion of Theorem~\ref{thm:limsup_full}
applied to $A$, \ie that $\limsup_n \lambda(f^n(A))=1$. Similarly,
we can prove that there is a full measure $B_x$ such that
$\limsup_{n\to\infty} |f^n(y)- f^n(x)|=0$ for every $y\in B_x$.
Indeed, if this were false then there are $m \in \N$, $\eps > 0$
and a set $B$ of positive measure such that $|f^n(y)- f^n(x)| >
\eps$ for every $y\in A$ and every $n\geq m$. Again, this
contradicts that $\limsup_n \lambda(f^n(B))=1$. Hence
 $$
  \liminf_{n\to\infty} |f^n(y)- f^n(x)|=0, \;\;\;\limsup_{n\to\infty}
  |f^n(y)- f^n(x)|\geq \diam(I)/2,
  $$
for every $y\in U_x=A_x\cap B_x$.
\end{proof}

By Proposition~\ref{prop:limsup_full_equiv} and
Lemma~\ref{lem:stronger} we immediately recover a well-known
result on weak-mixing.

\begin{corollary}\label{cor:density}
 If $f\in \CIIInf(I)$ is topologically mixing and has an acip, then
 $\lambda_2$-a.e. $(x,y)$ has a dense orbit in $I^2$.
\end{corollary}

\begin{remark}
Probably the earliest result in this direction dates back to
Ledrappier  who proves in \cite[Theorem 1]{Ledrap} that a certain
class of interval maps is weak Bernoulli. Keller \cite{keller}
proved (weak-)mixing for acips $\mu$ of multimodal maps assuming
negative Schwarzian derivative. he also showed that
$d\mu/d\lambda$ is bounded away from zero on $\supp(\mu)$. One can
prove that $\mu$ is also mixing by means of an induced map $F:U
\to U$ as in Theorem~\ref{thm:induced} which possesses an acip
$\nu$, cf. \cite{Y}. In fact, by taking an appropriate power of
$F^N$, $U$ decomposes into a finite number of $F^N$-invariant
parts on which $\nu$ is invariant and mixing. Pulling back $\nu$
to the original system, we recover $\mu$ and by the topologically
mixing condition, $\mu$ can have only one mixing component.
\end{remark}

Conservativity of $\lambda_2$ is crucial in
Lemma~\ref{lem:stronger}, and even if $\lambda$ itself is
conservative for $(I,f)$, this does not guarantee that $\lambda_2$
is conservative for the Cartesian product. It is for this reason
that Proposition~\ref{PropExactDis} and Theorem~\ref{mainthmLY}
are not just direct consequences of ergodicity of $\lambda_2$ from
Lemma~\ref{exact-ergodic}. The following conjecture suggests
conditions under which $\lambda_2$-a.e.\ pair is Li-Yorke, but has
no dense orbit.

\begin{conjecture}
We think that if $(I,f, \lambda)$ is conservative and has no acip,
but instead the induced time $k_x$ associated to the induced map
in Theorem~\ref{thm:induced} is non-integrable w.r.t.\ Lebesgue,
and in fact the tail $\lambda(\{ x : k_x > s\}) \ge 1/\log s$,
then the product system $(I^2, f_2, \lambda_2)$ is dissipative.
\end{conjecture}

\section{Li-Yorke chaos in the presence of Cantor attractors}\label{sec:attractors}

In this section, we concentrate on $C^2$ unimodal maps with Cantor
attractors $\A = \omega(c)$ for the unique critical point $c$.
If $f$ is infinitely renormalizable (\ie type (3) in
Proposition~\ref{propAttr}), then the situation regarding Li-Yorke
pairs is well-known: there are none. Instead, the attractor
is Lyapunov stable and conjugate to an {\em adding machine}
$(\Omega, g)$. In other words, $\Omega = \{ (\omega_j)_{j \geq 1}
: 0 \leq \omega_j < p_j \}$, for some sequence $(p_i)_{i \geq 1}$
of integers $p_i \geq 2$ (where $p_1 \cdots p_i$ are the periods of
the periodic intervals)
such that is equipped with product topology and the map
$g$ of ``adding $1$ and carry'':
\begin{equation}\label{eq:addingmachine}
\a( \omega_1,\omega_2, \dots) = \left\{ \begin{array}{ll}
0,0,,\dots, 0,\omega_k+1,\omega_{k+1},\omega_{k+2},\dots & \text{
if }
k = \min\{ i : \omega_i < p_i-1\}; \\
0,0,0,0,\dots & \text{ if } \omega_i = p_i-1 \text{ for all } i
\geq 1.
\end{array}\right.
\end{equation}

The following classification is due to \cite[Proposition 5.1]{BrJi}.

\begin{proposition}\label{PropApproxPer}
Let $f:I\to I$ be a continuous map and $x\in I$. The
system $(\omega(x),f)$ is conjugate to some $(p_i)$-adic adding
machine if and only if $x$ is approximately but not
asymptotically periodic, see Definition~\ref{DefApproxPer}.
\end{proposition}

However, there are several constructions leading to {\em strange
adding machines}, \ie (critical) omega-limit sets that are
conjugate to adding machines, but not involving periodic intervals,
see \cite{BKM, BrJi, Btree}. In \cite{Btree} it is shown that a
strange adding machine can still be an attractor, but in this case
it is a wild and not a solenoidal attractor.

The dynamics on such attractors can frequently be understood in
terms of generalized adding machines as done in \cite{BKS} and in the
proof of Theorem~\ref{thm:DistalAttractor} below. Such generalized adding
machines are based on the sequence of cutting time, which we will define now.

Let $Z_n(x)$ be the \emph{$n$-cylinder}, \ie maximal interval
containing $x$ on which $f^n$ is monotone. Unless $x \in \cup_{m <
n} f^{-m}(c)$, $Z_n(x)$ is well-defined, but the critical point itself has two
sets $Z_n^{\pm}(c)$ on either side of it, with $\D_n :=
f^n(Z_n^\pm(c))$ independent of $\pm$.
One can show that for
every $x$ and $n \geq 1$, there is $m \leq n$ such that $f^n(Z_n(x)) =
\D_m$; if $x \in f^{-m}(c)$, then $x$ is the common boundary point of two
$n$-cylinder sets $Z_n^\pm(x)$, but
$f^n(Z_n^\pm(x)) = \D_m$ independently of $\pm$.
We say that $n$ is a {\em cutting time} if $\D_n \owns c$,
and we list them in increasing order $1 = S_0 < S_1 < \dots$.
Write $c_n := f^n(c)$. By a short induction proof one can show
that $\D_1 = [0,c_1]$ and for $n \geq 2$
\begin{equation}\label{eq:beta}
\D_n = [c_n, c_{\beta(n)}] \text{ and } \D_n \subset
\D_{\beta(n)},
\end{equation}
where the map $\beta:\N \to \N$ is defined as $\beta(n) = n -
\max \{ S_k : S_k < n\}$. In the special case that $n = S_k$ is a
cutting time, this means that $S_k - S_{k-1}$ is again a cutting
time, so one can define the {\em kneading map} $Q: \N \to \N \cup
\{ 0, \infty \}$ by
\[
S_k = S_{k-1} + S_{Q(k)},
\]
see \cite{hofbauer, Bknead}.
Here $Q(k) = \infty$ means that $S_k$ does not exist; in this case $c$ is attracted to an orbit of period $S_{k-1}$ or $2S_{k-1}$.
Since we assumed that $f$ has a Cantor attractor, $Q(k) < \infty$
for all $k \in \N$.
Some properties of the kneading map related to $\omega(c)$ are as follows.

\begin{proposition}\label{prop:Qinfty}
(a) The minimal $n$ such that $Z^\pm_n(c) \subset [c, c_{S_k}]$
is $n = S_{Q(k+1)}$. \\
(b) If $f$ is a unimodal map with kneading map $Q(k) \to \infty$,
then the length $|\D_n| \to 0$ and $\omega(c)$ is a minimal Cantor set.
\\[1mm]
(c) If $\max_k k - Q(k) \leq B$, then
$\omega(c) \subset \cup_{n = 1+S_k}^{S_{k+B}} \D_n$ for every $k$. \\[1mm]
(d) The map is renormalizable if and only if there is $k \geq 1$ such
that $k = Q(k+1) \leq Q(k+j)$ for all $j > 1$. In this case $S_k$
is the period of renormalisation.
\end{proposition}

\begin{proof}
By definition of the kneading map, $n = S_{Q(k+1)}$ is the smallest positive iterate such that $f^n([c,c_{S_k}]) \owns c$, implying statement (a).
\\
Under the assumption that $Q(k) \to \infty$, this means that $|c_{S_k} - c| \to 0$
and also $|c_{S_{Q(k)}} - c| \to 0$ as $k \to \infty$.
Since $\D_{S_k} = [c_{S_k}, c_{S_{Q(k)}}]$, the non-Contraction Principle
shows that $|\D_n| \to 0$ as well. Minimality of $\omega(c)$ follows as in \ie
\cite[Proposition 2]{Bknead} (which proves that $c$ is persistently recurrent) and \cite[Lemma 8]{Bknead}.
This proves statement (b).
\\
For statement (c), observe that $\D_{1+S_l} \subset \D_{1+S_k}$
for all $l \geq k+B$ because of the assumption $\max_k \{ k - Q(k) \} \leq B$
and statement (a).
It follows that $\D_{m + S_{k+B}} \subset \D_{m+S_k}$
for all $m \geq 1$, and $\cup_{n \geq 1+S_k} \D_n \subset \cup_{n = 1+S_k}^{S_{k+B}} \D_n$.
By definition, $\omega(c) \subset \Cl \cup_{n \geq 1+S_k} \D_n \subset \Cl \cup_{n = 1+S_k}^{S_{k+B}} \D_n
= \cup_{n = 1+S_k}^{S_{k+B}} \D_n$ as asserted.
\\
Part (d) is \cite[Proposition 1(iii)]{Bknead}.
\end{proof}

Special types of unimodal maps are the Feigenbaum map ($S_k =
2S_{k-1}$) and the Fibonacci map ($S_k = S_{k-1} + S_{k-2}$). We
call $f$ {\em Fibonacci-like} if $\{ k - Q(k)\}_k$ is bounded.

The proof of the existence of wild attractors was first
established in \cite{BKNS} for Fibonacci maps with sufficiently
critical order $\ell$. In \cite{BTams}, this result was extended
to $C^3$ unimodal maps with negative Schwarzian derivative, sufficiently
large critical order and eventually non-decreasing kneading map $Q$ such
that $\limsup_k k-Q(k) \leq B$.
If the cutting times increase
more slowly than Fibonacci-like, then no unimodal
map with finite critical order can have a wild attractor.

Before we continue, let us give a short exposition how the
existence of wild attractor is proved for unimodal maps. There are
two approaches, both based on a random walk on a Markov graph. In
\cite{BKNS, BTams, BrHaw}, this Markov graph of $I$ is based on
preimages of the orientation reversing fixed point $p$. The states
of that Markov system are pairs of intervals $U_k \subset
(Z^+_{S_k}(c) \cup Z^-_{S_k}(c)) \setminus ( Z^+_{S_{k+B}}(c) \cup
Z^-_{S_{k+B}}(c) )$ and $\partial U_k$ belongs to the backward
orbit of $p$. The induced Markov map defined by $G|_{U_k} =
f^{S_k}$ preserves the partition $\{ U_k \}$ of $[\hat p, p]$,
where $f^{-1}(p) = \{ p, \hat p\}$. Viewing the dynamics of $G$ as
a random walk, we define ``random variables'' $\chi_n$ by
\begin{equation}\label{eq:drift}
\chi_n(x) = k \text{ if } G^n(x) \in U_k.
\end{equation}
It is then shown that this process has {\em positive
drift}, \ie the expectations (measured with respect to
Lebesgue measure)
\[
\E(\chi_{n+1} - k | \chi_n = k) \geq \eta > 0
\]
uniformly in $n$ and $l$. The second moments $\E((\chi_{n+1} -
k)^2 | \chi_n = k)$ are shown to be bounded as well. It follows
that $\chi_n(x) \to \infty$ for $\lambda$-a.e.\ $x$ and hence
$G^n(x) \to c$. Since $U_k  \subset (Z^+_{S_k}(c) \cup
Z^-_{S_k}(c))$, $f^{S_k}(U_k) \subset \D_{S_k}$ and since $|\D_n|
\to 0$ as $n \to \infty$, this means for the original map
 that $f^n(x) \to \omega(c)$ for $\lambda$-a.e.\ $x$,
so  $\A = \omega(c)$ is an attractor.

In \cite[Theorem 5.2]{BrHaw}, a further conclusion is drawn from
the positive drift, namely a Borel-Cantelli Lemma argument shows that for
$\lambda$-a.e.\ $x \in \bas(\A)$, there is $k_0 = k_0(x)$ such that
such that for all $k \geq k_0$
\begin{equation}\label{eq:driftCB}
\text{ if } G^m(x) \in U_k,
\text{ then } G^{m+j}(x) \notin U_k
\text{ for } j > k.
\end{equation}

In this paper, we will use a second approach from \cite{Bruin, BKS},
where the Markov system is a disjoint union
 $\hat I = \sqcup_{n \geq 2} \D_n$.
(the {\em Hofbauer tower}) for intervals $\D_n$ defined above.
Let $\hat f : \hat I \to \hat I$ be defined by
$$
\hat f(\D_n) = \left\{ \begin{array}{ll}
\D_{n+1} & \text{ if $n$ is not a cutting time;} \\[1mm]
\D_{1+S_k} \sqcup ( \D_{1+S_{Q(k)}} \setminus \{c_1\}) & \text{ if $n = S_k$ is a cutting time.}
\end{array} \right.
$$
Then $i \circ \hat f = f \circ i$, where $i : \hat I \to I$ is
the inclusion map, and $f$ is continuous, except at the points
$c \in \D_{S_k}$, $k \geq 1$, which are mapped to $c_1 \in \D_{1+S_k}$.
Since these are only countable many points, this has no effect on the
Lebesgue typical behavior.
The collection $\{ \D_n \}_{n \geq 2}$ is a {\em Markov partition}
of $\hat I$, because $\hat f$ maps each $\D_n$ to the union of partition elements
(ignoring again the point $c_1 \in \D_{1+S_{Q(k)}}$).
Starting in some $\D_n$, the subsequence intervals visited are unique
determined up to the moment we reach some $\D_{S_k}$, where we have a choice
between $\D_{1+S_k}$ and $\D_{1+S_{Q(k)}}$.
Therefore it suffices to consider the transitions from interval
$E_l := \D_{1+S_l}$ realized by the $S_{Q(l+1)}$-th image of $f$,
see Figure~\ref{fig:Markov}.
\begin{figure}[ht]
\unitlength=8.8mm
\begin{picture}(20,5)(-1,0.5)
\thicklines
\put(1,2.5){\line(0,1){2}} \put(0.7,4.8){$\D_{1+S_{ Q(l)-1 } }$}
\put(4,2.5){\line(0,1){2}} \put(3.8,4.8){$\D_{1+S_{Q(l)}}$}
\put(7,2.5){\line(0,1){2}} \put(6.7,4.8){$\D_{1+S_{Q(l)+1}}$}

\put(11,2.5){\line(0,1){2}} \put(10.7,4.8){$\D_{1+S_{l-1}}$}
\put(14,2.5){\line(0,1){2}} \put(13.8,4.8){$\D_{1+S_l}$}
%\put(17,2.5){\line(0,1){2}} \put(16.7,4.8){$\D_{S_{l+1}}$}

\thinlines
\put(0.3,3){$\dots$}
 \put(1.5,3){\vector(1,0){2}} \put(2,3.3){$f^{S_{Q^2(l)}}$}
 \put(4.5,3){\vector(1,0){2}} \put(4.8,3.3){$f^{S_{Q(Q(l)+1)}}$}
 \put(7.5,3){\vector(1,0){2}} \put(7.8,3.3){$f^{S_{Q(Q(l)+2)}}$}
\put(9.6,3){$\dots$}

 \put(11.5,3){\vector(1,0){2}} \put(12,3.3){$f^{S_{Q(l)}}$}
 \put(14.5,3){\vector(1,0){2}} \put(15,3.3){$f^{S_{Q(l+1)}}$}
\put(16.6,3){$\dots$}

 \put(10.7,2.5){\line(-2,-1){2}}
 \put(8.7,1.5){\line(-1,0){2}}
 \put(6.7,1.5){\vector(-2,1){2}}
\put(7.2,0.8){$f^{S_{Q(l)}}$}

 \put(13.7,2.5){\line(-2,-1){2}}
 \put(11.7,1.5){\vector(-1,0){2}}
% \put(9.7,1.5){\vector(-2,1){2}}
\put(10.2,0.8){$f^{S_{Q(l+1)}}$}

 \put(6.7,2.5){\line(-2,-1){2}}
 \put(4.7,1.5){\vector(-1,0){2}}
% \put(1.7,1.5){\vector(-2,1){2}}
\put(3.2,0.8){$f^{S_{Q(Q(l)+2)}}$}

 \put(3.7,2.5){\line(-2,-1){2}}
 \put(1.7,1.5){\vector(-1,0){2}}
% \put(1.7,1.5){\vector(-2,1){2}}
\put(0.2,0.8){$f^{S_{Q(Q(l)+1)}}$}

\end{picture}
\caption{The transitions from $\D_{1+S_l}$, $l \geq 1$.
Backward transitions go from $\D_{1+S_l}$ to $\D_{1+S_{Q(l+1)} }$.
\label{fig:Markov} }
\end{figure}
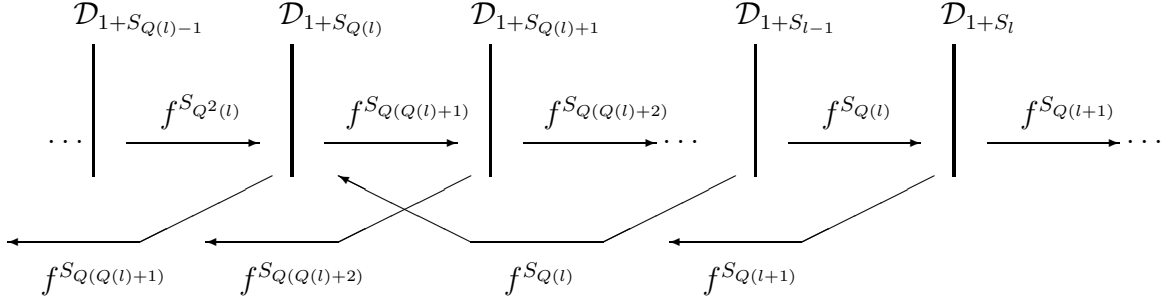
It follows that $\{ E_l \}_{l \geq 1}$
is a Markov partition for the induced map
\[
F: \sqcup_{l \geq 0} E_l \to \sqcup_{l \geq 0} E_l, \quad
F|_{E_l} =  \hat f^{S_{Q(l+1)}}.
\]
One can show that if $f$ is renormalizable of period $S_k$, then
$\sqcup_{l \geq k-1} E_l$ is a trapping region for $F$, see part (d)
of Proposition~\ref{prop:Qinfty}.
Now we need to translate the results on positive drift to the current
set-up. Write $\hat \chi_n(y) = l$ if $F^n(y) \in E_l$.

\begin{lemma}\label{lem:hatchi}
If $\{ \chi_n \}_{n\geq 0}$ has positive drift as in
\eqref{eq:drift} (and bounded second moments), then for
$\lambda$-a.e.\ $y \in \sqcup_{l \geq 1} E_l$, we have $\hat
\chi_n(y) \to \infty$. Furthermore, there is $\hat k_0 = \hat
k_0(y)$ and $C > 0$ such that for all $\hat k \geq \hat k_0$, if
$\hat\chi_n(y) = \hat k$, then $\hat\chi_{n+j}(y) > \hat k$ for
all $j > C\hat k$.
\end{lemma}

\begin{proof}
The positive drift of $\{ \chi_n \}$ ensures that for
$\lambda$-a.e.\ $x$ and $k \in \N$, there is $n_k$ such that
$G^n(x) \in \cup_{j \geq k} U_j$ for all $n \geq n_k$. Since $U_k
\subset Z^+_{S_k}(c) \cup Z^-_{S_k}(c)$, the image $f^{S_k}(U_k)
\subset \D_{S_k}$ and $f^{1+S_k}(U_k) \subset \D_{1+S_k} \cup
\D_{1+S_{Q(k)}} = E_k \cup E_{Q(k)}$. This means that $y :=
f^{1+S_k} \circ G^n(x) \in  E_k \cup E_{Q(k)}$. If consequently
$G^{n+1}(x) \in U_l$, $l > k$, then $f^{S_l}(y) \in E_l \cup
E_{Q(l)}$ but the move from $E_k$ or $E_{Q(k)}$ up to $E_l$ or
$E_{Q(l)}$ in $\sqcup_{j \geq 0} E_j$ involves passages through
the intermediate states $E_i$ as well, but ``lower'' states $E_i$,
$i < k-B\leq \min_{j \geq k} Q(j)$ are avoided. Therefore
$\hat\chi_n(f(x)) \to \infty$ for all $n \to \infty$.

For the second statement, observe that the passage from $E_k$ or
$E_{Q(k)}$ up to $E_l$ or $E_{Q(l)}$ requires $u_n$ iterates of
$F$ for some $l-k-B \leq u_n \leq l-k+B$. Suppose $G^m(x) \in
U_k$, $k \geq k_0(x)$ as in \eqref{eq:driftCB}, then $G^{m+j}(x)
\notin U_k$ for $j \geq k$. Take $\hat k \in \{ k, Q(k) \}$ such
that $y = f^{1+S_k}(G^m(x)) \in E_{\hat k}$. The iterates $m+1,
\dots, m+k$ of $G$ correspond to $\sum_{j = m+1}^{m+k} u_j$
iterates of $F$. Each $u_j \geq 1$, and if $\chi_{m+j}(x) <
\chi_{m+j-1}(x)$, then this single iterate of $G$ corresponds to a
single iterate of $F$, reducing the index of the state by at most
$B$. If $\chi_{m+j}(x) \gg \chi_{m+j-1}(x)$, then one iterate of
$G$ corresponds to many iterates of $F$, but if some of these
iterates brings $y$ above state $E_{\hat k + kB}$, then it would
take more than $k$ steps of $G$ to return, hence this will not
occur. Thus, if the $\sum_{j = m+1}^{m+\hat k} u_j$ iterates of
$F$ (corresponding $k$ iterates of $G$) keep $y$ close to state
$E_{\hat k}$, then $\sum_{j = m+1}^{m+\hat k} u_j \leq 2Bk \leq
3B\hat k$, where the last inequality follows because $k-B \leq
\hat k \leq k$. This proves the second statement for $C = 3B$.
\end{proof}

The dynamics of wild attractors has been investigated in
\cite{BKS}. In that paper, various combinatorial types are
presented for which $(\A, f)$ is semi-conjugate to a
monothetic group $(G,g)$, where $g:G \to G$ is an isometry for
which every orbit is dense. The best known example goes back to
Lyubich and Milnor \cite{LM}; it is the Fibonacci map and its
omega-limit set $\omega(c)$
factorizes over the golden mean circle rotation. In
\cite{BKS}, similar examples are shown factorising onto tori of
any dimension, and even onto a solenoid. On the other hand, \cite{BKS}
gives examples for which $(\A, f$) is weak mixing with respect to the
unique invariant probability measure $\mu$ supported on
$\A$. In \cite{BrHaw}, the simplest such example is shown
to be Lebesgue exact as well.

Let $\fp( x ) := x - \mbox{round}(x) \in [-\frac12,\frac12)$ be
the signed distance of $x$ to the nearest integer.

\begin{theorem}\label{thm:DistalAttractor}
Assume that a unimodal map has a wild attractor with positive
drift. If there exists $\rho$ such that the cutting times satisfy
$$
\sum_k k \max_{i \geq k-B} |\fp( \rho S_k )|  < \infty,
$$
for $B = \limsup_k k-Q(k)$, then $\lambda_2(\Dis) = 1$.
\end{theorem}

\begin{proof} {\bf Step 1: Construction of the factor map.}\\
An {\em enumeration scale} is a symbolic system resembling an
adding machine as in \eqref{eq:addingmachine} based on, in this
case, the sequence of cutting times. Any non-negative integer $n$
can be written in a canonical way as a sum of cutting times: $n =
\sum_j e_j S_j$, where
\[
e_j := \left\{
\begin{array}{ll}
1 &\text{ if } j = \max\{ k; S_k \leq n- \sum_{m > k} e_m S_m \},\\
0 &\text{ otherwise.}
\end{array} \right.
\]
In particular $e_j = 0$ if $S_j > n$. In this way we can code the
non-negative integers as zero-one sequences with a finite
number of ones: $n \mapsto \langle n \rangle \in \{ 0,1 \}^{\N}$.
Let $E_0 = \langle \N \cup \{ 0 \} \rangle$ be the set of such
sequences, and let $E$ be the closure of $E_0$ in the product
topology. This results in
$$
E := \{ e \in \{ 0,1 \}^{\N}\ ;\ e_i = 1 \Rightarrow e_j = 0 \text{
for } Q(i+1) \leq j < i \}.
$$
The condition in this set follows because if $e_i = e_{Q(i+1)} =
1$, then this should be rewritten to $e_i  = e_{Q(i+1)} = 0$ and
$e_{i+1} = 1$. It follows immediately that for each $e \in E$ and
$j \geq 0$,
\begin{equation}\label{rule}
e_0S_0 + e_1 S_1 + \dots + e_j S_j < S_{j+1}.
\end{equation}
We denote by $\a$ the standard addition of $1$ by means of ``add and carry'',
cf.\ \eqref{eq:addingmachine}.
Let $\langle n \rangle$ be the representation of $n \in \N
\cup \{ 0 \}$ in the enumeration scale based on $\{ S_k \}_{k \geq
0}$. Obviously
$\a(\langle n \rangle) = \langle n+1 \rangle$. Under the condition
that $Q(k) \to \infty$, $\a:E \to E$ is continuous, and is
invertible on $E \setminus \{ \langle 0 \rangle \}$, see
\cite{BKS,GLT}.
Since $\sum_k |\fp( \rho S_k )| < \infty$, we can define a continuous
projection $\pi_\rho:E\to \S^1$ by
$$
\pi_\rho(e) = \sum_k e_k \fp(\rho S_k ) \bmod 1,
$$
and $\pi_\rho \circ g = R_\rho \circ \pi_\rho$  for the circle rotation $R_\rho:x \mapsto x+\rho \bmod 1$.
At the same time, the map $P:E \to \A$ defined as the continuous
extension of $P(\langle n \rangle) = f^n(c)$ satisfies $P \circ g = f \circ P$.
\begin{figure}[ht]
\unitlength=10mm
\begin{picture}(6,5)(0,0.7)
\put(2.8,5){$(E,g)$}
\put(3.8,4.8){\vector(1,-1){1}} \put(4.7,4.3){$\pi_\rho$}
\put(2.8,4.8){\vector(-1,-1){1}}  \put(1.8,4.3){$P$}
\put(0.3,3){$(\A,f)$}
\put(4.8,3){$(\S^1,R_\rho)$}
\put(2,3){\vector(1,0){2.5}} \put(2.2,3.2){$\pi = \pi_\rho \circ P^{-1}$}
\put(1,2){$\bigcap$}
\put(0.1,1){$(\bas(\A),f)$}\put(2.8,1.6){\vector(2,1){1.7}}
 \put(3.3,2.1){$\tilde \pi$}
\end{picture}
%\caption{commutative diagram.}
\label{fig:com_diag}
\end{figure}
We know from \cite{BKS} that there is semi-conjugacy $\pi =
\pi_\rho \circ P^{-1}:\A \to \S^1$ such that $\pi \circ f = R_\rho \circ
\pi$, provided $\sum_k |\fp( \rho S_k )| < \infty$.
A more direct way to construct $\pi:\A \to \S^1$ is by setting
$$
\pi(x) = \left\{ \begin{array}{rl}
\rho n \bmod 1 & \text{ if } x = c_n, \\
\lim_{j\to \infty} \rho n_j  \bmod 1 &
\text{ if } x \in \A \setminus \orb(c) \text{ and } (n_j)_{j \in \N}
\text{ is such that } x \in \cap_j \D_{n_j}.
\end{array} \right.
$$
In \cite{BrHaw} it
was shown how to extend the map $\pi_\rho \circ P^{-1}$ to a measurable
factor map $\tilde \pi:\bas(\A) \to \S^1$.
Here we will give a construction of $\tilde \pi$ which is more closely connected
to \cite{BKS}.

%For any $x \in \bas(\A)$, then $\dist(f^k(x), \partial f^k(Z_k(x)) \to 0$
%as $k \to \infty$, see \eg \cite{BTams}.
%We know that $f^k(Z_k(x)) = \D_l$ for some $l \leq k$, and
%$\lim_{l\to \infty} |\D_l| = 0$,
For any $x$ with $\hat \chi_n(x) \to \infty$, the number $b_n(x)$
defined as
$$ \label{eq:bn}
b_n(x) := \max_j \{ j : f^j(Z_j(x)) = \D_n \}
$$
exists. If $x \in f^{-k}(c)$ for some $k \geq 0$, then we need to write
$Z^\pm_n(x)$ for $n > k$, because $x$ is the common boundary point of two
cylinder sets, and $f^n(x) = c_{n-k} \in \D_{n-k}$ for $n$ sufficiently large.
So $b_n(x)$ is well-defined in this case too.
%Since $\sum_k |\fp( \rho S_k )| < \infty$, we can define
Set
\begin{equation} \label{eq:pin}
\tilde \pi_n(x) := -\sum_k \fp( \rho \langle b_n(x)-n \rangle_k S_k )
\bmod 1 = -\rho(b_n(x)-n) \bmod 1.
\end{equation}
If $n+1$ is not a cutting time, then $b_{n+1}(x) = b_n(x) + 1$; in
this case $\tilde \pi_{n+1}(x) = \tilde \pi_n(x)$. If $n+1 = S_k$ is a cutting
time, $f^{b_n(x)}(Z_{b_n(x)}(x)) = \D_{S_k-1}$ and
$f^{b_n(x)+1}(Z_{b_n(x)}(x)) = \D_{S_k}$, but $b_{n+1}(x)$ can be
strictly larger than $b_{n}(x)+1$. In this case, however,
\[
b_{n+1}(x) = b_{n}(x)+1 + \sum_{j \geq k}  d_j S_{Q(j)}
\]
for some non-negative integers $d_j$.
Recall from Lemma~\ref{lem:hatchi} that for
$\lambda$-a.e.\ $x \in \bas(\A)$, there is $\hat k_0 = \hat k_0(x)$ such that
for all $k > \hat k_0$
\begin{equation}\label{eq:driftCBbis}
\text{ if } F^m(x) \in E_k, \text{ then } F^{m+j}(x) \notin E_k
\text{ for all } j > Ck.
\end{equation}
This means that $\sum_{j \geq k} d_j \leq k$, and therefore
%there is $C \in \N$, independently of $k$, such that
$\langle b_n(x)-n \rangle$ and $\langle b_{n+1}(x)- (n+1) \rangle$
are sequences
which differ by at most $C k$ entries and the indices of these
entries are $\geq k-B$. This means that by the definition of
\eqref{eq:pin}
\begin{equation}\label{eq:Ckmax}
|\tilde \pi_n(x) - \tilde \pi_{n+1}(x)| \leq C k
\max_{i \geq k-B}| \fp( \rho S_i ) |,
\end{equation}
which is summable over $k$ by assumption.
It follows that $\{ \tilde \pi_n(x)\}_n$ is a Cauchy sequence in
$\S^1$ for $\lambda$-a.e.\ $x
\in \bas(\A)$, and hence $\tilde \pi(x) := \lim_{n \to \infty} \tilde \pi_n(x)$
exists.

Let us complete Step 1 by showing that $\tilde \pi \circ f = R_\rho \circ \tilde \pi$ for $\lambda$-a.e.\ $x \in \bas(\A)$.
Assume that $x \notin \cup_j f^{-j}(c)$, then $f(Z_j(x)) = Z_{j-1}(f(x))$
for all sufficiently large $j$.
Therefore $b_n(f(x)) = b_n(x) - 1$, so substituting
into \eqref{eq:pin} gives $\tilde \pi_n(f(x)) = \tilde \pi_n(x) + \rho$ for each
$n$. In the limit, $\tilde \pi \circ f = R_\rho \circ \tilde \pi$.

\begin{remark}
It can be shown that $\tilde \pi$ is well-defined on $\A$ and
coincides with $\pi$, but since it plays no role in Theorem~\ref{thm:DistalAttractor}, we will omit the proof.
\end{remark}

{\bf Step 2: The measure of \boldmath $\Dis$. \unboldmath}\\
We will show that if $\tilde \pi(x) \neq \tilde \pi(y)$, then
$(x,y)$ form a distal pair. This is more involved than in
Proposition~\ref{propFibo1} below, because $\tilde
\pi|_{\bas(\A)}$ is not continuous. With the exception of a set of
measure zero, we can assume that $x$ and $y$ satisfy
\eqref{eq:driftCBbis}; let $k_1 = \max\{ \hat k_0(x),\hat k_0(y)
\}$. Suppose that $\tilde \pi(x) \neq \tilde \pi(y)$ and take $N
\in \N$ and $\eta > 0$ such that $|\tilde \pi_n(x) - \tilde
\pi_n(y)|
> 2 \eta$ for all $n \geq N$. \iffalse Let
\[
L_n := \big\{ \big( \sum_k e_k \fp( \rho S_k ) \big) \bmod 1 : e
\in E, e_j = \langle n \rangle_j \text{ for all } j \text{ with }
S_j \leq n \big\}
\]
be the images of the cylinder sets.
Clearly $|L_n| \to 0$ as $n \to \infty$.
\fi
Take $k_2 \geq k_1$ so large that
%$|L_n| < \eta$ for all $n \geq S_{k_0}$ and
$C k \max_{i \geq k-B} | \fp( \rho S_i ) | < \eta$
for all $k \geq k_2$ and $C$ as in Lemma~\ref{lem:hatchi}.

We know that $\A \subset \cup_{n \geq S_{k_2}} \D_n$, but by
Proposition~\ref{prop:Qinfty}, part (c),
$\omega(c) \subset \cup_{n = 1+S_{k_2}}^{S_{k_2+B}}
\D_n$. Take $\eps > 0$ so small that every two intervals
$\D_n,\D_{n'}$, $S_{k_2}< n, n' \leq S_{k_2+B}$ either are at least
$\eps$ apart or their intersection has length at most $\eps$.

\iffalse For every $x \in \bas(\A)$, $f^i(x) \in \mathcal D$ for
$i$ sufficiently large, even though $\mathcal D$ is to some extent
only a one-sided closed neighborhood of $\A$. The reason is that
$f^n$ assumes a local extremum of $c_n$ at $c$, and if $i$ is such
that $f^{i}(x) \in Z_{S_{k_2}}^\pm(c)$, then all iterates
$f^{i+S_{k_2}+j}(x)$ lie on that side of the point $c_{S_{k_2}+j}$
that is covered by $\mathcal D$. \footnote{HB: At the moment I
don't know how to rewrite this paragraph. The point is that if $y$
is a point that basically ``climb'' in the Hofbauer tower, it will
eventually not leave sets of the type $\D$ anymore, and every
point attracted to $\omega(c)$ will eventually enter $\D$, even
though $\D$ only contains a ``one-sided'' open neighbourhood of
$\omega(c)$.} \fi

For $\lambda$-a.e. $x \in \bas(\A)$, we have that
$\hat\chi_n(x)\to \infty$. Hence $f^i(x) \in \mathcal D$ for $i$
sufficiently large, even though $\D$ only contains a
``one-sided'' open neighbourhood of $\omega(c)$. (In fact it is
possible to prove that the same statement is true for all $x \in
\bas(\A)$, but this is enough for our purposes.)

Suppose now by contradiction that $(x,y)$ is not distal. Then
there is $i \geq N$ such that $f^i(x), f^i(y) \in \mathcal D$ and
$|f^i(x)-f^i(y)| < \eps$. So $f^i(x)$ and $f^i(y)$ belong to the
same interval $\D_n$ for some $n \geq S_{k_2}$, and taking $i$
larger if necessary, we can assume that $n$ is a cutting time. By
\eqref{eq:driftCBbis}, $\langle b_n(x)-i \rangle$ and $\langle
b_n(y)-i\rangle$ are sequences which differ by at most $C k$
entries and the indices of these entries are $\geq k-B$. Similar
to \eqref{eq:Ckmax}, we have
$$
|\tilde \pi_n(x) - \tilde \pi_n(y)| \leq C k
\max\{ |\fp( \rho S_i )| : k-B \leq i \leq 2k-B \}
< \eta.
$$
This contradiction to the
choice of $\eta$ and $N$ proves that $(x,y)$ is distal.
Therefore proximal pairs $(x,y)$ can only exist within fibers of $\tilde \pi$.

Finally, if $W = \tilde \pi^{-1}(s)$ with $\lambda(W) > 0$ for some $s
\in \S^1$, then since $\tilde \pi \circ f = R_\rho \circ \tilde \pi$ and
$R_\rho$ is invertible, it follows that $f^m(W) \cap f^n(W) =
\emptyset$ for all $0 \leq m < n$, and this contradicts the
non-existence of strongly wandering sets. Therefore each fiber has
measure zero. This completes the proof.
\end{proof}

\begin{corollary}\label{Cor:d=234}
If $S_k = S_{k-1} + S_{k-d}$ for $d = 2,3,4$, and $f$ is a map
with cutting times $\{ S_k\}_{k \geq 0}$ and sufficiently large
critical order so that $\A$ is a wild attractor with
positive drift, then $\lambda_2(\Dis) = 1$.
\end{corollary}

\begin{proof}
We know from \cite{BKS} that the dynamics  $(\A, f)$  is
semi-conjugate to a minimal rotation on a $d-1$-dimensional torus.
If fact, the characteristic equation  $\lambda^d = \lambda^{d-1}
+ 1$ of the recursive relation $S_k = S_{k-1} + S_{k-d}$ has a
leading root $\rho$ which is a Pisot number, \ie all its algebraic
conjugates lie within the unit disk. It follows easily that $\fp(
\rho S_k )$ is exponentially small in $k$ so that the condition
$\sum_k k |\fp( \rho S_k )| < \infty$ is obviously satisfied, and
Theorem~\ref{thm:DistalAttractor} shows that $\lambda_2(\Dis) = 1$.
By defining $\Pi:\bas(\A) \to \T^{d-1}$ as $\Pi(x) =
(\tilde \pi_\rho(x), \dots , \tilde \pi_{\rho^{d-1}}(x))$
(which is well-defined $\lambda$-a.e.), we obtain a factor
map onto $(\T^{d-1}, R_{\rho,\dots,\rho^{d-1}})$ with Haar
measure, which is the maximal automorphic factor of
$(\bas(\A), \lambda, f)$.
\end{proof}

\begin{proposition}\label{prop:scrambled_fibers}
Under the conditions of Theorem~\ref{thm:DistalAttractor}
and if there is a continuous factor of
$(\A, f)$ onto a $d$-dimensional torus $\T^d$
(with $d \geq 1$),
then the fiber $\tilde \pi^{-1}(\tau) \subset \bas(\A)$ for each
$\tau \in \T^d$ contains an uncountable
$\eps$-scrambled set for some $\eps > 0$.
\end{proposition}

\begin{proof}
Assume for simplicity that $d=1$ and take $\rho \in \R$ such that
$\tilde \pi \circ f = \tilde \pi + \rho \bmod 1$ for $\lambda$-e.a.\ $x \in \A$.
%$\tilde \pi|_{\A} = \pi_\rho \circ P^{-1}$.
Because a zero-dimension set cannot be mapped injectively onto a set
of higher dimension, see \cite{Engel},
the continuity of the factor map $\pi:\A \to \T^d$ implies
that $\pi$ cannot be injective.
Therefore we can find $a \neq \hat a \in \A$
such that $f(a) = f(\hat a)$.
By Proposition~\ref{prop:Qinfty}, part (c),
we can find, for every $k$,
integers $\kappa$ and $\hat \kappa$ with $k < \kappa, \hat \kappa \leq k+B$
and $N \leq \min(S_\kappa, S_{\hat \kappa})$,
such that $\D_{S_{\kappa} - N} \owns a$ and
$\D_{S_{\hat \kappa} - N} \owns \hat a$.

Hence we can find two sequences $(\kappa_j)_{j \in \N}$ and
$(\hat \kappa_j)_{j \in \N}$ with $|\kappa_j - \hat \kappa_j| \le B$, but
possibly $\kappa_{j+1} \gg \kappa_j$, and another sequence $(N_j)_{j \in \N}$
such that $\D_{S_{\kappa_j} - N_j} \owns a$
and $\D_{S_{\hat\kappa_j} - N_j} \owns \hat a$.

Recall the Markov map was defined as
$F: \sqcup_{l \geq 0} E_l \to \sqcup_{l \geq 0} E_l$.
%where $E_l := \D_{1+S_l}$.
For each $j$, we will create loops from $E_{\kappa_j}$ to itself
and  from $E_{\hat \kappa_j}$ to itself, as indicated in
Figure~\ref{fig:loops}. Both loops require the same steps under $F$,
only arranged in a different order, hence they involve the same number
$s_j$ of iterates of $f$.
Because $\limsup_k k - Q(k) \leq B$, both loops involve no more than $2B$ steps,
and the width (highest vertex minus smallest vertex) is less than $2B$ as well.

\begin{figure}[ht]
\unitlength=8mm
\begin{picture}(16,3.5)(1.5,0.0)
\put(4,3){\circle*{0.18}}
\put(3.8,3.4){$E_{\kappa_j}$} \put(6.8,3.4){$E_{\hat \kappa_j}$}
\put(4.2,3){\vector(1,0){0.6}} \put(5,3){\circle{0.18}}
\put(5.2,3){\vector(1,0){0.6}} \put(6,3){\circle{0.18}}
\put(6.2,3){\vector(1,0){0.6}} \put(7,3){\circle*{0.18}}
\put(6.8,2.95){\vector(-2,-1){1.6}} \put(5,2){\circle{0.18}}
\put(4.8,1.95){\vector(-2,-1){1.6}} \put(3,1){\circle{0.18}}
\put(3.2,1){\vector(1,0){0.6}} \put(4,1){\circle*{0.18}}
 \put(3,0.2){loop $E_{\kappa_j} \to E_{\kappa_j}$}
\put(12,3){\circle*{0.18}}
\put(11.8,2.95){\vector(-2,-1){1.6}} \put(10,2){\circle{0.18}}
\put(9.8,1.95){\vector(-2,-1){1.6}} \put(13,1){\circle{0.18}}
\put(8.2,1){\vector(1,0){0.6}} \put(9,1){\circle*{0.18}}
\put(8,0.2){path $E_{\hat \kappa_j} \to E_{\kappa_j}$}
\put(13.8,3.4){$E_{\kappa_j}$} \put(16.8,3.4){$E_{\hat \kappa_j}$}
\put(17,3){\circle*{0.18}}
\put(16.8,2.95){\vector(-2,-1){1.6}} \put(15,2){\circle{0.18}}
\put(14.8,1.95){\vector(-2,-1){1.6}} \put(13,1){\circle{0.18}}
\put(13.2,1){\vector(1,0){0.6}} \put(14,1){\circle*{0.18}}
\put(14.2,1){\vector(1,0){0.6}} \put(15,1){\circle{0.18}}
\put(15.2,1){\vector(1,0){0.6}} \put(16,1){\circle{0.18}}
\put(16.2,1){\vector(1,0){0.6}} \put(17,1){\circle*{0.18}}
\put(13,0.2){loop $E_{\hat \kappa_j} \to E_{\hat \kappa_j}$}
\end{picture}
\caption{Different loops ending at $E_{\kappa_j}$ and $E_{\hat \kappa_j}$
respectively. Both loops contain the same path from
$E_{\hat \kappa_j}$ to $E_{\kappa_j}$, depicted in the middle.
Backward arrows go from $E_l$ to $E_{Q(l+1)}$.}
\label{fig:loops}
\end{figure}
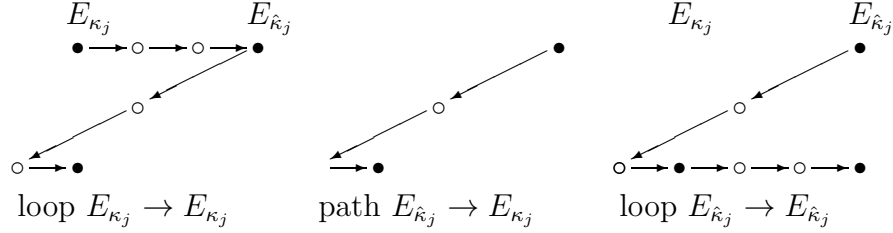

Assume that $\kappa_j < \hat \kappa_j$ and $E_{\kappa_j} \supset E_{\hat \kappa_j}$. (The other three cases can be treated similarly.)
Since $\{ E_l\}_{l \geq 0}$ is a Markov partition for $F$, there are intervals
$J_j \subset E_{\kappa_j}$ and
$\hat J_j \subset E_{\hat \kappa_j}$
such that $f^{s_j}: J_j \to E_{\kappa_j}$ and
$f^{s_j}: \hat J_j \to E_{\hat \kappa_j}$ are diffeomorphic and onto.
There is a similar interval $\hat K_j \subset E_{\hat \kappa_j}$
representing the path from $E_{\hat \kappa_j}$ to $E_{\kappa_j}$,
that is if the path from $E_{\hat \kappa_j}$ to $E_{\kappa_j}$
requires $t_j$ iterates of $f$, then
$f^{t_j}: \hat K_j \to E_{\kappa_j}$ is diffeomorphic and onto.

Combining the two, we find intervals $H_j \subset J_j$
such that
$$
\left\{ \begin{array}{ll}
(a) & f^{s_j}(H_j) = \hat K_j \subset E_{\hat \kappa_j} \subset  E_{\kappa_j}, \\[1mm]
(b) & f^{s_j+t_j}(H_j) =  E_{\kappa_j}, \text{ and } \\[1mm]
(c) & f^{s_j-(N_j+1)}(H_j) \text{ contains, or is close to, } a. \end{array} \right.
$$
Similarly, we can find $\hat H_j \subset \hat J_j$ such that
$$
\left\{ \begin{array}{ll}
(\hat a) & f^{s_j}(\hat H_j) = \hat K_j \subset E_{\hat \kappa_j},\\[1mm]
(\hat b) & f^{s_j+t_j}(\hat H_j) =  E_{\kappa_j}, \text{ and } \\[1mm]
(\hat c) & f^{s_j-(N_j+1)}(\hat H_j) \text{ contains, or is close to, } \hat a.
\end{array} \right.
$$
Let $\Sigma \subset \{ 0, 1\}^\N$ be an uncountable
scrambled subset of the full shift.
The idea is now, for each $\tau \in \T$ and each $\sigma \in \Sigma$,
to find a point $x \in \pi^{-1}(\tau)$ such that
\begin{equation}\label{eq:x}
f^{r_j}(x) \in \left\{  \begin{array}{ll}
\D_{S_{\kappa_j} - N_j} \owns a  & \text{ if } \sigma_j = 0, \\
\D_{S_{\hat \kappa_j}-N_j}\owns \hat a & \text{ if } \sigma_j = 1,
\end{array}\right.
\end{equation}
where the sequence $r_j$ depends of $t$ but not on $\sigma$.

Start with some $y \in \D_2 = E_0$ with $\tilde \pi(y) = \tau+\eps$ where
$\eps$ will be determined later.
Then, when the orbit of $y$ under iteration of $F$ goes from $E_{\kappa_j}$ to
$E_{\hat \kappa_j}$, we insert one of the loops as in Figure~\ref{fig:loops}
according to whether $\sigma_j = 0$ or $1$, and the extra path from $E_{\kappa_j}$ to $E_{\hat \kappa_j}$.
That is, when $q_1$ is such that $f^{q_1}(y) \in E_{\kappa_j}$,
we insert one of the extended loops, both taking $s_1 + t_1$ iterates,
and iterate $r_1 := q_1+s_1-(N_1+1)$ brings the path close to $a$ or $\hat a$,
depending on whether $\sigma_1 = 0$ or $1$.
Then, when $y$ visit $E_{\kappa_2}$, the new extended loop takes $s_1 + t_1$ iterates more to reach it; call this number $q_2$, insert the
appropriate extended loop of $s_2+t_2$ iterates, and find that after
$r_2 = q_2+s_2-(N_2+1)$ iterates,
the path will be close to $a$ or $\hat a$, etc.

Due to the Markov property, there is some $x \in E_0$ whose infinite path
under $F$ is precisely the path we have created, so $x$ satisfies \eqref{eq:x}.
Furthermore, the values $b_n(x) = b_n(y) + \sum_{j, q_j + B < n} s_j + t_j$.
When we compute $\tilde \pi(x)$, the contribution $\eps$ of all the inserted
extended loops bounded by
$$
\sum_j 3B \max_{\kappa_j - B \leq i \leq \kappa_j+B} |\fp(\rho S_i)| < \infty.
$$
Here we used that each loop requires at most $2B$ steps
and each path from $E_{\kappa_j}$ to $E_{\hat \kappa_j}$ (or vice versa)
has at most
$B$ steps, and that all these steps are taken within a distance
$B$ from the step $E_{\kappa_j} \to E_{\kappa_j+1}$.

In particular, $\tilde \pi(x)$ is well-defined, and
$\tilde \pi(x) = \tilde \pi(y) - (\tilde \pi(y) - \pi(x)) = \tilde \pi(y) - \eps = \tau$, so $x \in \tilde \pi^{-1}(\tau)$.

Finally, if $\sigma \neq \sigma'$ both belong to $\Sigma$, there are infinitely many
$j$s such that $\sigma_j \neq \sigma'_j$, and consequently
$\limsup_{j\to\infty} | f^{r_j}(x) - f^{r_j}(x') |
= |a - \hat a| > 0$.
On the other hand, $\lim_{j\to\infty} | f^{q_j}(x) - f^{q_j}(x') | = 0$,
because  $f^{q_j}(x)$ and $f^{q_j}(x')$ both belong to $E_{\kappa_j}$ or to
 $E_{\hat \kappa_j}$
Therefore $\{ x(\sigma) : \sigma \in \Sigma \}$ is an
$|a - \hat a|$-scrambled subset of $\pi^{-1}(\tau)$, as required.
\end{proof}

\begin{proposition}\label{propFibo1}
Let $f$ be a unimodal map with kneading map $Q(k) = \max\{ k-d, 0
\}$ for $d = 2, 3, 4$, and let $\mu$ denote the unique
invariant probability measure supported by $\omega(c)$. Then
\begin{itemize}
\item $\mu_2$-a.e.\ pair of points is distal;
\item if $d = 2$ (the Fibonacci map) then
$\omega(c)$ contains no Li-Yorke pair, and the only asymptotic
pairs $(x,y)$ are such that $f^n(x) = f^n(y) = c$ (only one such pair for
each $n \ge 1$);
\item if $d = 3,4$ then there are uncountably many Li-Yorke pairs in $\omega(c)$.
\end{itemize}
\end{proposition}

\begin{proof}
As in Corollary~\ref{Cor:d=234}, there is a continuous map
$\pi:\omega(c) \to \T^{d-1}$ onto the $d-1$-dimensional torus and
an irrational rotation $R := R_{\rho, \dots, \rho^d}: \T^d \to \T^d$ such
that $\pi \circ f = R \circ \pi$.
If $\pi(x) \neq \pi(y)$ then
$(x,y)$ is distal, as in Corollary~\ref{Cor:d=234}.
It follows from \cite{BKS} that this happens for $\mu_2$-a.e.\ $(x,y)$.

Furthermore, $\pi^{-1}(b)$ consists of at most $d$
points $a_1, \dots , a_d$ for any $b \in \T^{d-1}$.
If there are indeed $d$ distinct points, then there is $n \geq 0$ such that
$f^n(a_1) = f^n(a_2) = \dots
= f^n(a_d) = c$, cf.\ \cite{Btree}.
For the Fibonacci map, this accounts for all non-distal pairs.
For $d = 3,4$, other non-singleton fibers $\pi^{-1}(b)$ are possible,
and Li-Yorke pairs exists within such fibers.
They are related to incidences in the substitution shift description, as
described implicitly in \cite{BD}.
\end{proof}

The question is whether the situation is the same for the ``next''
Fibonacci-like map with $Q(k) = \max\{ k-5,0 \}$. In this case,
the system of $(\omega(c), f)$ with its unique probability measure
is weak mixing, so there is no continuous (or even measurable)
factor map onto a group rotation. The difference with the
previous cases is that the characteristic equation of the
recursive relation
\[
0 = \lambda^5 - \lambda^4 - 1 =
(\lambda^2-\lambda+1)(\lambda^3-\lambda-1)
\]
is reducible, and more decisively, its leading root is not a
Pisot number. The following curiosity about the cutting times
holds in this case:
\begin{equation}\label{eq5Fib}
S_k = S_{k-2} + S_{k-3} + \left\{
\begin{array}{rl}
+ 1 & \text{ if } k \equiv 2,3 \bmod 6; \\
- 1 & \text{ if } k \equiv 5,0 \bmod 6; \\
  0 & \text{ if } k \equiv 1,4 \bmod 6.
\end{array} \right.
\end{equation}
Note that the same algebraic
curiosity holds for any characteristic equation $\lambda^{6m-1} -
\lambda^{6m-2} - 1 = 0$, because in each such case,
$\lambda^2-\lambda+1$ (with solutions $\lambda = \frac{1\pm
i\sqrt{3}}{2}$ on the unit circle) divides the equation. As an
example, the case $m=2$ gives:
\[
\lambda^{11} - \lambda^{10} - 1 = (\lambda^2-\lambda+1)
(\lambda^9-\lambda^7 - \lambda^6 + \lambda^4 +
\lambda^3-\lambda-1),
\]
and one can indeed check that
\[
S_k = S_{k-2} + S_{k-3} - S_{k-5} - S_{k-6} + S_{k-8} + S_{k-9}
 + \left\{
\begin{array}{rl}
+ 1 & \text{ if } k \equiv 2,3 \bmod 6; \\
- 1 & \text{ if } k \equiv 5,0 \bmod 6; \\
  0 & \text{ if } k \equiv 1,4 \bmod 6.
\end{array} \right.
\]

\begin{proposition}\label{propLY}
If $S_k = S_{k-1} + S_{Q(k)}$ for $Q(k) = \max\{0,k-5\}$ and $f$
is a map with cutting times $\{ S_k\}_{k \geq 0}$ and a wild
attractor with positive drift, then $\lambda_2(\LY_\eps) = 1$
for some $\eps > 0$.
\end{proposition}

\begin{proof}%[Proof of the $5$-Fibonacci example]
%In \cite{BKNS}, an induced map $G$ is constructed as in the footnote on
%\pageref{note} for the
%Fibonacci unimodal map, and it is shown if the critical order is
%sufficiently large, then
% $\lim_{n\to\infty} G^n(x) = c$ for $\lambda$-a.e.\ $x$.
% This is the crucial step in showing that $\omega(x) = \omega(c)$
% $\lambda$-a.e.\ for this map.
% In \cite{BTams} the same is shown to happen for
% unimodal maps with different combinatorial types, including
% maps with cutting times satisfying $S_k = S_{k-1} + S_{k-5}$.
As pointed out before, it was shown in \cite{BTams} that $\A = \omega(c)$ is a
wild attractor, and from \cite{BrHaw} it follows that dynamics on the
basin of the attractor is Lebesgue exact. Thus
Proposition~\ref{PropExactDis} implies that $\lambda_2(\LY) = 1$.
In fact, the construction of Proposition~\ref{prop:scrambled_fibers} can also
be used here to show that there is $\eps > 0$ such that $\lambda_2$-a.e.\ pair belongs to $\LY_\eps$.
\end{proof}

\begin{example}\label{ExamLYDIS}
Let $f$ be a unimodal map with cutting times satisfying
$$
S_0 = 1,\ S_1 = 2,\  S_2 = 3,\  S_3 = 4,\  S_4 = 6,\  S_5 = 8,\
S_6 = 10,\  S_7 = 12
$$
and
$$
S_k = S_{k-1} +  S_{k-5} \text{ for } k \geq 8.
$$
This means that the cutting times $S_k$ are even for $k \geq 3$
and  eventually are twice the numbers occurring in the example of
Proposition~\ref{propLY}. Assume also that the critical order of
$\ell$ is so large that $f$ has a wild attractor $\A$. Then $\A$
decomposes into two disjoint Cantor sets $\A_0$ and $\A_1$ which
are permuted by $f$. Note, however, that $f$ is not
renormalizable, see Proposition~\ref{prop:Qinfty}, and therefore
$f$ is topologically mixing on $[c_2, c_1]$.

Let $B_0$ and $B_1$ be disjoint neighborhoods of $\A_0$ and
$\A_1$; for example we can take $B_0 = [c_2, c_{14}] \cup [c_4, c_6]$
and $B_1 = [c_3, c_{15}] \cup [c_5, c_1]$.
Every point in the basin of $\A$ will eventually be
trapped in $B_0 \cup B_1$. But every pair $(x,y)$ with $x \in B_0$
and $y \in B_1$ such that $\orb(x), \orb(y) \subset B_0 \cup B_1$
is distal. On the other hand $f^2|_{\A_0}$ and $f^2|_{\A_1}$
behave like the example of Proposition~\ref{propLY}, so
$\lambda_2$-a.e.\ every pair $(x,y) \in B_0 \times B_0$ (or $(x,y)
\in B_1 \times B_1$) such that $\orb(x), \orb(y) \subset B_0 \cup
B_1$ is Li-Yorke.
\end{example}

\begin{proof}[\textbf{Proof of Theorem~\ref{mainthm:Classification}}]
From Theorem~\ref{mainthm:Asymp} we know that $\lambda_2(\Prox) = 0$.
Also, if some $x \in \bas(\A)$ is approximately periodic, then by
Proposition~\ref{PropApproxPer}, $\A$ is conjugate to an adding
machine, so all points in $\bas(\A)$ are approximately periodic.
By Proposition~\ref{prop:approximately}, $\bas(\A)$ contains no Li-Yorke
pairs.
Therefore (a)-(d) are the only possibilities, and they all occur:\\
(a) The strange adding machine case as wild attractor, see \cite{Btree}.\\
(b) The Fibonacci-like map with kneading map $Q(k) = \max\{ k-d, 0
\}$, $d = 2,3,4$,
see Theorem~\ref{thm:DistalAttractor} and Corollary~\ref{Cor:d=234}.\\
(c) The Fibonacci-like map with kneading map $Q(k) = \max\{ k-5, 0
\}$, see Proposition~\ref{propLY}.\\
(d) See Example~\ref{ExamLYDIS}.\\
Proposition~\ref{prop:scrambled_fibers} implies the existence of $\eps$-scrambled sets in
the fibers $\tilde \pi^{-1}(\tau)$ for all $\tau \in \S^1$ and factor maps $\tilde \pi$ in case
(b), and for cases (c) and (d), the $\eps$-scrambled set is also immediate.
Li-Yorke sensitivity follows as well.
\end{proof}

\begin{remark} In case (a) and (b) of
Theorem~\ref{mainthm:Classification}, $(\A,\mu_{\A}, f)$ is not weakly
mixing. Instead, there is an $f_2$-invariant set of positive
$\mu_2$-measure that is bounded away from the diagonal of $\A_2$.
We expect that case (c) always corresponds to the weakly mixing
case, cf. Proposition~\ref{propLY}, and in particular,
$\mu_\A \times \mu_\A$-a.e.\ pair $(x,y)$ has a dense orbit in $\A \times \A$,
and hence is Li-Yorke.
\end{remark}

\begin{remark}
In case (a), points $x$ in the basin of $\A$ have a distinct
{\em target} point $t_x \in \A$ such that $\dist(f^n(x), f^n(t_x)) \to 0$,
see Proposition~\ref{prop:approximately}.
 In case (b), such target points $t_x \in \A$ do not exist in general,
cf.\ the proof of Proposition~\ref{prop:scrambled_fibers} and
also \cite[Remark 2]{BrHaw}.
\end{remark}

\section{Appendix} \label{sec:appendix}

In this appendix, we prove the more technical results of the
paper, divided into four parts. First we give a result on the
structure of solenoidal sets of Feigenbaum type, next we prove
Theorem~\ref{thm:lemma11}, then we formulate an improved $C^3$
Koebe distortion lemma, and finally we prove
Theorem~\ref{thm:induced}.

\subsection{Neigborhoods of Solenoidal Sets}

Assume that $S$ is a solenoidal set and $\cyc(K)$ is a solenoidal
cycle of period $r$ containing $S$. Let $S_i=S\cap f^i(K)$, $0\leq
i<r$. Since the critical points in $\cyc(K)$ belong to $S$, the
convex hulls $J_i$ of $S_i$ form a cycle of intervals, that is,
$f(J_i)=J_{i+1}$ for every $i$ (with this we also mean
$f(J_{r-1})=J_0$). We call it the \emph{$r$-minimal solenoidal
cycle covering $S$}. We emphasize that the intervals $J_i$ are
pairwise disjoint (because a solenoid contains no periodic
points). Moreover, if they are ordered in the real line as
$J_{i_1}<J_{i_2}<\cdots<J_{i_r}$, then there is a union of
periodic orbits $P=\{p_i\}_{i=1}^{r-1}$ such that
$J_{i_1}<p_1<J_{i_2}<p_2<\cdots <p_{r-1}<J_{i_r}$ (see \cite{MT}
or \cite{AJS}).

The existence of these periodic orbits of smaller period
interlaced among the intervals $J_i$ allows us to prove
immediately that the union $L_i$ of all periodic intervals of
period $r$ containing $J_i$ is also periodic of period $r$.
Clearly, $f(L_i)\subset L_{i+1}$ and $f(\partial L_i)\subset
\partial L_{i+1}$ for every $i$. However, although the intervals
$L_i$ have pairwise disjoint interiors, they do not form a cycle
of intervals because $f(L_i)=L_{i+1}$ need not hold. In fact
$(L_i)_{i=0}^r$ is the pullback chain of $L_r:=L_0$ along
$J_0,J_1,\ldots,J_r=J_0$. Let $M_i=\Inte L_i$. We call
$M=\bigcup_{i=0}^{r-1} M_i$ the \emph{$r$-shell covering $S$}.
Then $M$ is nice and  $J_i\subset M_i$ for every $i$.

\begin{proposition} \label{prop:feigenbaum}
Let $f\in \CIInf(I)$ and let $S$ be a
 solenoidal set of Feigenbaum type. Let $T$ be
 the $r$-minimal solenoidal cycle covering $S$ for some $r$,
 $T=\bigcup_{i=0}^{r-1} f^i(J)$, and let $M=\bigcup_{i=0}^{r-1} M_i$
 be the $r$-shell covering $S$, with $f^i(J)\subset M_i$ for
 every $0\leq i <r$. Then there is $\xi=\xi(f)>0$
 such that $f^i(J)$ is $\xi$-well centered
 in $M_i$ for every $i$.
\end{proposition}

\begin{proof}
Since $S$ is a solenoid, there is a turning point $c\in S$.
 There
 is no loss of generality in assuming that $c\in J$.
 Also, we can assume that $r$ is large enough so that each
 interval $M_i$ contains
at most one critical point of $f$ and there
 are no critical points in $\Cl M$ outside $T$. Hence $(M_i)_{i=0}^r$ is
 the pullback chain of $M_r:=M_0$ along $J,f(J),\ldots,f^r(J)=J$.
 Let $\cyc(R)$ be the $2r$-minimal
 solenoidal cycle covering $S$, with $c\in R$. We can for example
 assume that $[u,v]=R$, $f^r(R)=[w,z]$ and  $c \in R$ is a local maximum,
Note that $J$ is the convex hull of $R \cup f^r(R)$.
 Observe that there is a periodic point $p$ of period $r$ such
 that $f^i(p)$ lies between $f^i(R)$ and
 $f^{i+r}(R)$ for every $i$. Since $T$ is a solenoidal cycle,
 $\orb(p)$ is hyperbolic repelling. In fact, since $f$ is monotone
 on each on the intervals connecting $f^i(R)$ and
 $f^{i+r}(R)$, it is the only
 $r$-periodic orbit in $T$ and $f^r([u,p))=(p,z]$,
 $f^r((p,z])=[u,p)$. Then $(\tau_c(p),p)$ is clearly the only nice
 periodic interval of period $2r$ containing $c$.

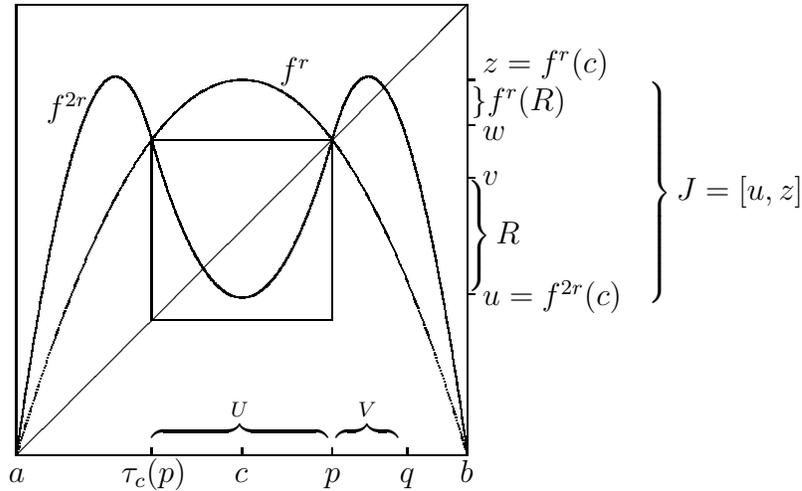
\begin{figure}[htb]
\begin{center}
\setlength{\unitlength}{1.0cm}
\begin{picture}(7,7)(1,-0.4)
\thinlines \put(0,0){\line(1,0){6}} \put(0,0){\line(0,1){6}}
\put(0,6){\line(1,0){6}} \put(6,0){\line(0,1){6}}
\put(0,0){\line(1,1){6}} \put(0,0){\line(0,1){6}}

\put(4.2,4.2){\line(0,-1){2.4}} \put(4.2,4.2){\line(-1,0){2.4}}
\put(4.2,1.8){\line(-1,0){2.4}} \put(1.8,4.2){\line(0,-1){2.4}}

\bezier{500}(0,0)(3,10)(6,0) \bezier{500}(0,0)(1,7.1)(1.8,4.2)
\bezier{500}(1.8,4.2)(3,0)(4.2,4.2)
\bezier{500}(4.2,4.2)(5,7.1)(6,0) \put(3.5,5){$f^r$}
\put(0.4,4.5){$f^{2r}$}
\put(2.9,-0.35){$c$}\put(3,0){\line(0,1){0.1}}
\put(4.1,-0.35){$p$}\put(4.2,0){\line(0,1){0.1}}
\put(5.1,-0.35){$q$}\put(5.2,0){\line(0,1){0.1}}
 \put(4.25,0.2){$\overbrace{\qquad}^{V}$}
 \put(1.8,0.2){$\overbrace{\qquad\!  \! \qquad\qquad}^{U}$}
\put(1.4,-0.35){$\tau_c(p)$}\put(1.8,0){\line(0,1){0.1}}
\put(-0.1,-0.35){$a$}\put(0,0){\line(0,1){0.1}}
\put(5.9,-0.35){$b$}\put(6,0){\line(0,1){0.1}} \put(6.2,5.1){$z =
f^r(c)$}\put(6,5){\line(1,0){0.1}} \put(6.2,2){$u =
f^{2r}(c)$}\put(6,2.15){\line(1,0){0.1}}
\put(6.2,3.6){$v$}\put(6,3.7){\line(1,0){0.1}}
\put(6.2,4.2){$w$}\put(6,4.4){\line(1,0){0.1}}
\put(5.6,2.83){$\left. \begin{array}{l} \\ \\ \\ \end{array}
\right\} R$} \put(6.05,4.6){$\} f^r(R)$} \put(8,3.4){$\left.
\begin{array}{l} \\ \\ \\ \\ \\ \\ \end{array} \right\} J =
[u,z]$}
\end{picture}
\caption{The graph of $f^r:M_0 = (a,b) \to M_0$ with $J$ and $R$
and relevant points} \label{fig:feig}
\end{center}
\end{figure}

 By \cite[Theorem~A'(1)]{vSV} there are an integer $s>0$, a number
 $\xi_0=\xi_0(f)>0$ and nice periodic intervals $N^m\ni c$ of periods  $s2^m$,
 $m\geq 0$, such that $N^{m+1}$ is $\xi_0$-well inside $N^m$ for
 every $m$. We have shown in the previous paragraph
 that if $s2^m$ is large enough, then there is just one
 nice periodic interval of period $s2^m$ containing $c$.
 Thus we can assume that $M_0=N^m$ and $(\tau_c(p),p)=N^{m+1}$
 for some $m$, with $r=s2^m$.

 Let $q>p$ be the closest point to $p$ satisfying
 $f^r(q)=\tau_c(p)$. Then $V=(p,q)$ is just the pullback of $U=(\tau_c(p),p)$ along
 $f^r(R),f^{r+1}(R),\ldots,f^{2r}(R)=R$. Since $U$ is $\xi_0$-well inside $M_0$,
 there is $\xi_1=\xi_1(f)$ such that $V$ is $\xi_1$-well inside $M_0$,  due to
 Corollary~\ref{cor:disjointness}.

 Let $W=(q,b]$ with $b$ the right
 endpoint of $M_0$. Then $\xi_0|U|\leq |V|+|W|$ and $\xi_1|V|\leq
 |W|$, hence
 $$
  \frac{\xi_0\xi_1}{1+\xi_0+\xi_1} (|U|+|V|)\leq |W|.
 $$
 Then $(\tau_c(q),q)$ is
 well centered in $M_0$ and so is any subinterval of
 $(\tau_c(q),q)$. In particular, $J$ is well
 centered in $M_0$.

 To finish the proof we must show that \emph{every}
 interval $f^i(J)$ is well centered
 in the corresponding interval $M_i$. Note that we cannot
 directly use Corollary~\ref{cor:disjointness}
 because it only guarantees that
 $f^i(J)$ is well inside $M_i$. Instead we proceed as
 follows. According to \cite[Lemma~2]{vSV}, $M_r:=M_0$ is well
 inside an interval $G_r$ which contains at most $e=2^{d+1}+3$ of the
 intervals $f^i(M_0)$, $0\leq i<r$, with $d = \# \Crit$. Therefore it also contains at most $e$ of the intervals
 $M_i$. Now \cite[Lemma~3]{vSV} implies that the pullback chain $(G_i)_{i=0}^r$ of
 $G_r$ along $M_0,\ldots,f^r(M_0)$ has order bounded by
 $2(e+d(e+2))+1$. Hence, by Corollary~\ref{cor:disjointness}, $M_i$ is well inside
 $G_i$ for every $i$
 and, additionally, if some iterate $f^l$ maps
 diffeomorphically $M_i$ onto $M_{i+l}$, then this
 diffeomorphism has bounded distortion. Thus, if $f^{i+l}(J)$ is
 well centered in $M_{i+l}$, $f^i(J)$ is well centered in $M_i$.
 Using now Lemma~\ref{lem:regularity} and recalling
 that $J=f^r(J)$ is well centered in $M_0=M_r$,
 we conclude that every interval $f^i(J)$ is well centered in
 $M_i$ as we desired to show.
\end{proof}

\subsection{Proof of Theorem~\ref{thm:lemma11}}

Later on we will apply the lemma below to the set $Q=\Crit$ of
critical points of our map $f\in \CIInf(I)$; recall that we are
assuming that $f$ has no periodic critical points, see
(\ref{nonperiodic}).

\begin{lemma} \label{lem:small}
 Let $f:I\to I$ be a multimodal map without wandering intervals
 and let $Q\subset I\setminus \partial I$ be a finite set
 containing no periodic points. Let $x \in Q$.
 Then there is an arbitrarily small nice interval $J \ni x$
 such that $\orb(Q)\cap \partial J=\emptyset$ and $\dist(\orb(\partial
 J),Q)>0$.
\end{lemma}

\begin{proof}
 Let $Q'=\bigcup_{n=-\infty}^\infty f^n(Q)$.
 We claim that the set $P=\AsPer\setminus Q'$ is dense
 in $I$. This implies the lemma. Indeed, let $\eps>0$ be
 small. Since $x$ is not periodic, there is no loss of generality in assuming that
 $\orb(Q)\cap (x-\eps,x+\eps)$ does not contain any periodic point. Take points
 $\hat a\in (x-\eps,x)\cap P$, $\hat b\in (x,x+\eps)\cap P$, and let $a<x<b$ be the points from
 $\Cl(\orb(\hat a)\cup \orb(\hat b))$ closest to $x$ from both
 sides. We emphasize that both $\hat a$ and $\hat b$ are asymptotically periodic
 and $x$ is not periodic, so $a$ and $b$ are well defined. Moreover, $(a,b)$ is nice.
 Also, $a\notin Q'$. If $a\in \orb(\hat a)\cup
 \orb(\hat b)$ this is obvious because $\hat a, \hat b\notin Q'$.
 If $a\notin\orb(\hat a)\cup \orb(\hat
 b)$, then $a$  must belong to a periodic orbit attracting either $\orb(\hat a)$ or  $\orb(\hat
 b)$, and again $a\notin Q'$ because $Q$ contains no periodic points and
 neither does $\orb(Q)\cap (x-\eps,x+\eps)$. Similarly, $b\notin
 Q'$. We have shown $\orb(Q)\cap \{a,b\}=\emptyset$ and $(\orb(a)\cup\orb(b))\cap
 Q=\emptyset$. Since $a$ and $b$ are asymptotically periodic and $Q$ contains
 no periodic points, the property $(\orb(a)\cup\orb(b))\cap
 Q=\emptyset$ implies in fact $\dist(\orb(a)\cup\orb(b),Q)>0$.
 Thus $J=(a,b)$ is the small nice interval we are looking for.

 We prove that every interval $K$ intersects $P$. If the
 intervals $\{f^n(K)\}_n$ are pairwise disjoint, then
 the absence of wandering intervals for $f$ implies
 that these intervals are attracted by a periodic orbit.
 Since $f$ is multimodal, the set of points in $K \cap Q'$
 is countable. Thus $K$ intersects $P$.

 Now assume that $f^n(K)\cap f^m(K)\neq \emptyset$ for some $n<m$.
 Let $k=m-n$. Using again that $f$ is multimodal, we get that
 $T=\Cl(\bigcup_{r=0}^\infty f^{n+rk}(K))$ is a nondegenerate interval.
 Moreover, it is  invariant for $f^k$. If $f^k|_T$ has finitely many
 periodic points, then all points from $T$ are
 asymptotically (or eventually) periodic, see \eg \cite[p. 127]{BlCo},
 so all points from $K$ are asymptotically periodic as well and
 $K\cap P\neq \emptyset$ as before. If $f^k|_T$ has infinitely
 many periodic points, there is a family of disjoint periodic
 orbits $\{\mathcal{O}_j\}_{j=1}^\infty$ such that
 $\orb(K)\cap \mathcal{O}_j \neq \emptyset$ for every $j$. If
 $j$ is large enough, then $\orb(Q)\cap \mathcal{O}_j=\emptyset$.
 Let $y\in K$ be a preimage of such $\mathcal{O}_j$. Then $y\notin
 \fullorb(Q)$, which finishes the proof.
\end{proof}

The next proposition strengthens
\cite[Proposition~5]{caili}. In what follows, we say that a nice
set $V$ is \emph{$\xi$-nice} if all return domains to $V$ are
$\xi$-well inside the components of $V$ containing them. Sometimes
we say that $V$ is \emph{uniformly nice} if it is $\xi$-nice for
some constant $\xi$ depending only on $f$. We denote
$$
 Z=\{x\in I: x\notin \orb(\Crit),\,\dist(\orb(x),\Crit)>0\}.
$$
Observe that $f^{-1}(Z) = Z$ and that
(reasoning as in the proof of Lemma~\ref{lem:small}) $Z$ is dense
for $f \in \CIInf(I)$.
We denote by $\Feig$ the critical points of Feigenbaum type.

\begin{proposition} \label{prop:caili}
 Let $f\in \CIInf(I)$.
 Then there are $\xi_0=\xi_0(f)>0$ and, for every $\eps>0$ and $c \in \Crit\setminus \Feig$, open
 intervals $c\in V_c$ with $|V_c|<\eps$,
 such that $V=\bigcup_{c\in \Crit\setminus \Feig} V_c$ is
 $\xi_0$-nice and $\partial V\subset Z$.
\end{proposition}

\begin{proof}
 In Cai and Li's version of this proposition, all critical points
 of solenoidal type (not only those of Feigenbaum type) are excluded and no additional
 properties on $\partial V$ are obtained.
 Nevertheless, the proof remains very much the same. We sketch it
 below, emphasizing the specific points where it must be modified.

 First of all, we prove:

 \textbf{Claim 1.} There is $\xi_1=\xi_1(f)$ such that if $c\in
 \Crit\setminus \Feig$, then there are arbitrarily small
 $\xi_1$-nice intervals $J$ containing $c$ such that $\partial
 J\subset Z$.

 This is \cite[Corollary~3]{caili}, but including also solenoidal
 critical points of non-Feigenbaum type. If $c$ is recurrent, then Claim~1 follows
 immediately from Propositions~\ref{macroscopickoebe} and
 \ref{prop:well-inside}. In fact, if $c$ is not solenoidal, and we
 take as the starting interval $I_0$ in
 Proposition~\ref{prop:well-inside} a small nice interval containing
$c$ with $\partial I_0\subset Z$ (which is possible by
 Lemma~\ref{lem:small}), then the central
return domain $I_m$ to $I_{m-1}$ satisfies $\partial I_m\subset Z$ as well
 for all $m$. Thus if $m$ is large enough
 and $I_{m+1}$ is well inside $I_m$, then,
 by Proposition~\ref{macroscopickoebe}, $I_{m+1}$ is the
 uniformly nice interval we need. If $c$ is solenoidal, then
 Propositions~\ref{macroscopickoebe} and
 \ref{prop:well-inside} again imply that there is an arbitrarily small
 uniformly nice interval $J$ containing $c$. Observe that
 if $J$ is sufficiently small, then it is contained in a solenoidal
 cycle very close to $\omega(c)$, which in particular implies that
 the orbits of its endpoints cannot accumulate on any critical
 point outside $\omega(c)$. Also, since $J$ is nice and $\omega(c)$ is minimal,
 they cannot accumulate on $\omega(c)$ either. Thus $\partial
 J\subset Z$.

If $c$ is not recurrent then
the argument from \cite[Corollary~3]{caili} applies without
 any significant changes. Namely, let $I'\owns c$ be a small nice interval
 with $\partial I'\subset Z$. We can assume that $c\notin D(I')$.
 Let $\delta$ be the minimal length
 of the components of $I'\setminus \{c\}$ and take an interval
 $(a',b')\subset (c-\delta/2,c+\delta/2)$ with $a',b'\in Z$.
This is possible by Lemma~\ref{lem:small}. If $a'\notin
 D(I')$, define $a=a'$. Otherwise, let $K$ be the return domain to
 $I'$ containing $a'$ and let $a$ be the endpoint of $K'$ in
 $(a',c)$. The point $b$ is defined similarly. Then $J=(a,b)$ is
 $1/2$-well inside $I'$, $\partial J\subset Z$ and $\partial J\cap
 D(I')=\emptyset$. If $x\in J\cap D(J)$, then the return domain $L$ to
 $J$ containing $x$ is well inside the return domain $L'$
to $I'$
 containing $x$ by  Proposition~\ref{macroscopickoebe}.
 Since $\partial J\cap D(I')=\emptyset$, we obtain
 $L' \subset J$. Then $L$ is well inside $J$ and $J$ is uniformly
 nice.

 \textbf{Claim 2.} Let $c_1,c_2,\ldots,c_k\in \Crit$ and $V_i\owns
 c_i$ be nice intervals such that $V=\bigcup_{i=1}^k V_i$ is a
 $\xi$-nice set and $\partial V\subset Z$.
 Then there is $\xi'=\xi'(\xi)>0$  such that the following hold.
 \begin{itemize}
  \item[(1)] For each $1\leq i\leq k$, there exist nice intervals
  $W_i\supset \tilde V_i\owns c_i$ such that
   \begin{itemize}
   \item $\tilde V_i$ is $\xi'$-well inside $W_i$ and $W_i$ is
   $\xi'$-well inside $V_i$;
   \item $\partial \tilde V_i\cap D(V)=\emptyset$ and $\partial W_i\cap
   D(V)=\emptyset$. In particular, $\bigcup_{i=1}^k \tilde V_i$ is a
   nice set;
   \end{itemize}
  \item[(2)] For each $x\in V_i\setminus \Cl \tilde V_i$, there is a
  nice interval  $J_x\owns x$ such that $J_x$ is
  $\xi'$-well inside $V_i$ and $J_x\cap \tilde V_i=\emptyset$,
  $\partial J_x\cap D(V)=\emptyset$.
  \item[(3)] The endpoints of all intervals above belong to $Z$.
 \end{itemize}

 Claim 2 is exactly \cite[Lemma~3]{caili}, except that we additionally
 request $\partial V\subset Z$ and get the extra property (3) in
 return. The proof requires no changes: only, instead of defining
 the auxiliary interval $(p',q')=
 (a-\xi'(b-a)/4, b+\xi'(b-a)/4)$ for $\tilde V_i=(a,b)$,
 we choose $p',q'\in Z$ with $p'\in (a-\xi'(b-a)/3, a-\xi'(b-a)/4)$ and
 $q'\in (b+\xi'(b-a)/4,b+\xi'(b-a)/3)$.

 We are now in position to prove Proposition~\ref{prop:caili}.
 This is done inductively. Let $\Crit\setminus \Feig=\{c_1,\ldots, c_m\}$.
 If $m=1$, then this is just Claim 1. Assume that we have
 constructed intervals $V_i\owns c_i$ with $|V_i|<\eps$ and
 $\partial V_i\subset Z$,
 $1\leq i\leq k$, such that $\bigcup_{i=1}^k V_i$ is $\xi_k$-nice for
 some constant $\xi_k>0$. We will show that there are smaller
 intervals $\tilde V_i\owns c_i$ with $\partial \tilde V_i\subset Z$,
 and a constant $\xi_{k+1}$
 depending only on $\xi_k$, such that $\bigcup_{i=1}^{k+1} \tilde V_i$ is
 $\xi_{k+1}$-nice.

 The intervals $\tilde V_i$, $1\leq i\leq k$, are those
 from Claim 2. To define $\tilde V_{k+1}$ and conclude the proof,
 two cases must be considered. If $c_{k+1}\in D(\bigcup_{i=1}^k V_i)$,
 then Cai and Li's proof works
 without any changes (it uses Claim 2 in its full extension).

 If $c_{k+1}\notin D(\bigcup_{i=1}^k V_i)$, then we need to find
 intervals $c_{k+1}\subset \tilde V_{k+1}\subset V_{k+1}$ with
 $|V_{k+1}|<\eps$, $\tilde V_{k+1}$ well inside $V_{k+1}$ and
 $\partial \tilde V_{k+1}\subset Z$, such that $\bigcup_{i=1}^{k+1} V_i$ is
 nice and $\partial \tilde V_i\notin D(\bigcup_{i=1}^{k+1} V_i)$ for
 every $1\leq i\leq k+1$.

 We define $V_{k+1}$ and $\tilde V_{k+1}$.
 Since $\tilde V_{k+1}$ is well inside
 $V_{k+1}$ and $\partial \tilde V_{k+1}\subset Z$, we only need
 to show that $\partial V_{k+1},\partial \tilde V_{k+1}
 \notin D(\bigcup_{i=1}^{k+1} V_i)$. (Recall that the endpoints of the intervals $V_i,\tilde
 V_i$, $i\leq k$, belong to $Z$, hence their orbits cannot visit
 $V_{k+1}$ if it is sufficiently small.)

 As in the proof of Claim 1, three possibilities
 arise for $c_{k+1}$. The simplest case is when $c_{k+1}$ is
 solenoidal. Then $V_{k+1}$ is defined as in Claim 1 and $\tilde
 V_{k+1}$ is the return domain to $V_{k+1}$ containing $c_{k+1}$;
 everything works because $c_{k+1}\notin D(\bigcup_{i=1}^k V_i)$.

 Now assume that $c_{k+1}$ is not solenoidal. Starting from a
 small interval $(a',b')\owns c_{k+1}$ with $a',b'\in Z$
 and repeating the reasoning in Case~2 of Cai and Li's proof, we
 find an interval $c_{k+1}\in (a,b)\subset (a',b')$ such that
 $a,b\in Z$ and
 $(a,b)\cup \bigcup_{i=1}^k V_i$ is nice. If $c_{k+1}$ is recurrent,
 we take $I_0=(a,b)$ in  Proposition~\ref{prop:well-inside},
 and define accordingly the intervals $I_m$.
 Then $I_m\cup \bigcup_{i=1}^k V_i$ is nice for every $m$
 because $c_{k+1}\notin D(\bigcup_{i=1}^k V_i)$. Now it suffices to
 fix $m'$ such that $I_{m'}$ is uniformly nice and
 define $V_{k+1}=I_{m'}$ and $\tilde V_{k+1}=I_{m'+1}$. If
 $c_{k+1}$ is non-recurrent, then we find an
 interval $J\owns c_{k+1}$, similarly as we did in
 Claim~1, such that
 $\partial J\subset Z$,
 $\partial J\cap D((a,b)\cup \bigcup_{i=1}^k V_i)=\emptyset$, and
 $J$ is $1/2$-well inside $(a,b)$. Then $V_{k+1}=(a,b)$
 and $\tilde V_{k+1}=J$ are adequate to our purposes.

Let us finally show that with this choice,
 $\tilde V=\bigcup_{i=1}^{k+1} \tilde V_i$ is uniformly nice.
 Let $x\in\tilde V_n \cap D(\tilde V)$, say $\phi_{\tilde V}(x) \in \tilde
V_l$. Let $0 \leq s \leq t$ be the entry time of $x$ to $\tilde
V_l$ and the return time of $x$ to $\tilde V$, respectively, and
denote the return domain to $\tilde V$ and the entry domain to
$V_l$ containing $x$ by $J$ and $K$ respectively, so $x \in J
\subset K$.
 By Proposition~\ref{macroscopickoebe}, the pullback $H$ of
 $\tilde V_l$ along $f^{s}(x),\ldots, f^t(x)$ is well inside $V_l$.
 Note that $J$ is the pullback of $H$ along $x,\ldots,f^s(x)$ and $K$ is
 the pullback of $V_l$ along $x,\ldots,f^s(x)$. Hence $J$ is well
 inside $K$ by Corollary~\ref{cor:disjointness}. Since
 $\partial \tilde V_n\notin D(\bigcup_{i=1}^{k+1} V_i)$, $K$ is
 contained in $\tilde V_n$, so $J$ is well inside $\tilde V_n$.
 This proves that $\tilde V$ is uniformly nice.
\end{proof}

\begin{proof}[\textbf{Proof of Theorem~\ref{thm:lemma11}}]

Fix $\eps>0$. We define the sets $c\in U_c\subset V_c$ and $W_c$
as follows. If $c\in \Crit\setminus \Feig$, then $V_c$ is the
component of the set $V$ from Proposition~\ref{prop:caili}
containing $c$. Write $V_c=(a,b)$ and say, for instance, that
$(c,b)$ is the smallest component of $V_c\setminus\{c\}$. Define
the convex combinations
$$
v_0 := \frac{c +  \xi'_0\ b}{1+\xi_0'}
\quad \text{ and } \quad
v_i := \frac{v_{i-1} + \xi_0'\ b}{1+\xi_0'}, \quad i=1,2,3,
$$
with $\xi_0'=\xi_0(f)/2$ and $\xi_0(f)$ the constant from
Proposition~\ref{prop:caili}. Take $v' \in (v_0, v_1) \cap Z$
close to $v_0$. If $v' \notin D(V)$, then set $v = v'$; otherwise
let $v$ be the right endpoint of the return domain to $V$
containing $v'$, and by having chosen $v'$ close to $v_0$, we
obtain $v_0 < v' \leq v < v_1$. This gives $v\in ((v_0,v_1) \cap Z
) \setminus D(V)$, and, similarly, there is $w \in ((v_2,v_3) \cap
Z ) \setminus D(V)$. If $c \in D(V)$, let $u$ denote the left
endpoint of the return domain to $V$ containing $c$; if $c\notin
D(V)$, take $u_0=c-(v_0-c)=2c-v_0$ and find as before $u\in
(u_0,c)$ belonging to $Z\setminus D(V)$. Finally we define
$U_c=(u, w)$ and $W_c=(v, w )$.

For critical points of Feigenbaum type we rely on
Proposition~\ref{prop:feigenbaum}. More precisely, for every
Feigenbaum solenoidal set $S$, we find a minimal solenoidal cycle
$T=\cyc(\hat J)=\bigcup_{i=0}^{s-1} f^i(\hat J)$ containing $S$
whose constituting intervals are small enough
to ensure that $|V_c| < \eps$ for the components of the corresponding set $V$.
Then the intervals
$U_c$ are those from the shell $M=\bigcup_{i=0}^{s-1} M_i$ of $T$
containing points from $\Crit$ (that is, from $\Feig$). To define
the sets $V_c$ assume, after reordering, that $M_0=M_s$ is the
smallest of all intervals $M_i$. Since the intervals $f^i(\hat J)$
are well centered in the intervals $M_i$ by
Proposition~\ref{prop:feigenbaum}, there is an interval $G_s$ such
that $M_s$ is well inside $G_s$ and $G_s$ intersects no other
interval from the solenoidal cycle $T$ than $\hat J=f^s(\hat J)$.
Pulling back $G_s$ along $f(\hat J),\ldots, f^s(\hat J)$, we
construct similarly intervals $G_i$, $i=1,\ldots,s$, such that
$M_i$ is well inside $G_i$ and $G_i\cap T=f^i(\hat J)$. The
intervals $V_c$ are those intervals $G_i$ containing points from
$\Crit$.

To define the  sets $W_c$ we take a boundary point $q$ of $M_0$
with $f^s(q)=q$ (in fact the proof of
Proposition~\ref{prop:feigenbaum} shows that it is possible that
$f^{s/2}(q)=q$). For every $0\leq i<s$, let $u_i$ denote the
middle point between $f^i(q)$ and the endpoint of $f^i(\hat J)$
closest to $f^i(q)$. If $M_i$ contains a critical point, let $L_i$
be the interval with endpoints $f^i(q)$ and $u_i$; if not, let
$L_i=M_i$. Finally, let $L'_i\subset L_i$ be the largest interval
with endpoint $f^i(q)$ such that $f^j(L'_i) \subset L_{i+j}$ for
all $0 \leq j < s$ with indices taken $\bmod s$, and hence
$f^s|_{L'_i}$ is a diffeomorphism.
%on which the composition
%$f|_{L_{i+s-1}}\circ\cdots \circ f|_{L_{i+1}}\circ f|_{L_i}$
Now the intervals $W_c$ are those intervals $L'_i$ contained in
intervals $U_c=M_i$ intersecting $\Crit$.

We show that (i)-(iii) in Theorem~\ref{thm:lemma11} hold.
Clearly, the construction implies the existence of a number
$\xi=\xi(f)>0$, thus not depending on $\eps$, such that $U_c$ is
$\xi$-well inside $V_c$ for every $c\in \Crit$. Moreover, since
the endpoints of all sets $W_{c'},U_{c'},V_{c'}$, $c'\in
\Crit\setminus \Feig$, belong to $Z\setminus D(V)$, we can assume
that they do not belong to $D(V_c)$, $c\in \Feig$, either. If
$c'\in \Feig$, then the niceness and invariance of shells still
guarantees $\partial U_{c'}\cap D(U_c)=\emptyset$ for every $c\in
\Crit$ and $\partial U_{c'}\cap D(V_c)=\emptyset$ for every $c\in
\Crit$ not belonging to the same solenoidal set as $c'$. In
particular, $U=\bigcup_{c\in \Crit} U_c$ is nice. This proves
Theorem~\ref{thm:lemma11}(i).

Now let $J$ be an entry domain to $U$, say $\phi|_J=f^j|_J$ and
$\phi(J)=U_c$. Then $f^j|_J$ is a diffeomorphism. We want to show
that $f^j|_J$ extends to a diffeomorphism $f^j|_K:K\to V_c$.
Assume by contradiction that this is not the case.
Then there is $K'\supset J$ such that
$f^j|_{K'}$ is a diffeomorphism, one of the endpoints $a$ of $K'$
satisfies $c'=f^n(a)\in \Crit$ for some $1\leq n<j$, and
$f^j(\{a\}\cup K')\subset V_c$.

Let $b\in K'$ such that $f^n(b)\in
\partial U_{c'}$ (here we use that $f^n(J)$ does not intersect $U$). We can
assume that both $c$ and $c'$ belong to the same Feigenbaum
solenoidal set $S$, because otherwise $\partial U_{c'}\cap
D(V_c)=\emptyset$, which contradicts $f^j(b)\in V_c$. Let
$T=\cyc(\hat J)$ be the minimal solenoidal cycle for $S$ we used
earlier to construct the sets $U_c,V_c$. Then there is $i \in \N$
such that $f^j(J)=U_c\supset f^i(\hat J)$. Since
$f^j(a)=f^{j-n}(c')\in S$, there is $i' \in \N$ such that
$f^j(a)\in f^{i'}(\hat J)$. But $f^j|_{\{a\}\cup K'}$ is a
homeomorphism, so $f^i(\hat J)$ and $f^{i'}(\hat J)$ are
different. This is impossible, because by its definition $V_c$
intersects exactly one interval from $T$.

We have shown that if $J$ is an entry domain to $U$, then
$f^j|_J:J\to U_c$ extends to a diffeomorphism  $f^j|_K:K\to V_c$.
Since $U_c$ is well inside $V_c$, there is $\kappa=\kappa(f)$ such
that $f^j|_J$ has distortion bounded by $\kappa$ by the $C^2$
Koebe Principle (Proposition~\ref{macroscopickoebe}).
This finishes the proof of
Theorem~\ref{thm:lemma11}(ii).

It remains to prove Theorem~\ref{thm:lemma11}(iii). If $c\in
\Crit\setminus \Feig$, then the definition of $W_c$ easily implies
that it is not too small compared to $U_c$. On the other hand, it
is not too large compared to the subinterval of $U_c$ between $c$
and $W_c$, so $f|_{W_c}$ has bounded distortion because $c$ is
non-flat. Since $\partial W_c\cap D(U)=\emptyset$,
Theorem~\ref{thm:lemma11}(iii) holds in this case.

Assume now that $c\in \Feig$ and $U_c$ is one of the above
intervals $M_{i_1}$, with $W_c=L'_{i_1}$. Put $k_c=s$ and let
$i_1<i_2<\cdots <i_t < i_1+s$ be the indices $i$ such that
$M_i$ contains some critical point. As shown in the proof of
Proposition~\ref{prop:feigenbaum}, each map
$f^{i_{r+1}-i_r-1}|_{M_{i_r+1}}$ has bounded distortion. Also,
each map $f|_{L_{i_r}}$ has bounded distortion because $|L_{i_r}|$
is less than half the distance from $f^{i_r}(q)$ to the critical
point in $M_{i_r}$ and non-flatness applies. Hence
$f^{k_c}|_{W_c}=f^s|_{L'_{i_1}}$ has bounded distortion. Recall
that all periodic orbits in the intervals $M_i$ are repelling, and
the same is true for $\orb(q)$. Then $f^{k_c}(W_c)\supset W_c$.

Finally, let us show that $W_c$ is not too small compared to
$U_c$. From its definition, $L_{i_1}$ is not too small compared to
$U_c$.
Let $A_j$ denote the largest interval with endpoint $f^j(q)$
such that $f^k(A_j) \subset L_{j+k}$ for all $0 \leq k \leq i_j-i_{j-1}$;
hence $A_1=L_{i_1}$ and $A_t=W_c$.
We claim that $A_{j+1}$ is not too small
compared with $A_j$. Indeed, $L_{i_{j+1}}$ is not
too small compared to $M_{i_{j+1}}$, so \emph{a fortiori} is not
too small compared to $f^{i_{j+1}-i_1}(A_j)$. Since
$f^{i_{j+1}-i_1}|_{A_j}$ has bounded distortion and
$A_{j+1}=A_j\cap f^{-(i_{j+1}-i_1)}(L_{i_{j+1}})$, $A_{j+1}$
cannot be too small compared to $A_j$ either.

This concludes the proof of Theorem~\ref{thm:lemma11}.
\end{proof}

\subsection{A \boldmath $C^3$ \unboldmath Koebe Distortion Lemma}

\begin{lemma} \label{lem:boundedorder}
 Let $f\in \CIIInf(I)$. Then for any $\xi>0$
 and $k\geq 0$, there is $\xi'=\xi'(\xi,k,f)>0$  such that the following statement
 holds: Let $(H_i)_{i=0}^l\subset
 (G_i)_{i=0}^l$ be chains such that $(G_i)_{i=0}^l$ has order at most $k$ and $G_l$
 is a small nice interval close enough to $\Crit$. If $H_l$ is $\xi$-well inside
 $G_l$, then $H_0$ is $\xi'$-well inside $G_0$. Moreover, if $k=0$, then there
 is $\kappa = \kappa(\xi, f)>0$ such that $f^l|_{H_0}$ has distortion bounded
 by $\kappa$.
\end{lemma}

A slightly weaker version of this lemma (requiring that the
intervals $G_i$ are not too close to parabolic periodic points)
appears in \cite{LS}, who in turns refer to
\cite[Theorem~C(2)]{vSV}. Our proof is based on
\cite[Proposition~3]{vSV} and Theorem~\ref{thm:lemma11}.

\begin{proof}[\textbf{Proof of Lemma~\ref{lem:boundedorder}}]
Assuming that the components of $U$ in Theorem~\ref{thm:lemma11}
are sufficiently small, we can conclude that iterates of $f$ one
beyond those mapping into $U$ have negative Schwarzian derivative.
The idea behind obtaining negative Schwarzian derivative goes back
to Kozlovski \cite{Koz}, see also \cite{GSS}, and the precise
statement is as follows.
\begin{quote}
Let $f\in \CIIInf(I)$ and $U$ be as in Theorem~\ref{thm:lemma11}.
If the components $U_c$ of $U$ are sufficiently small, and $x$ is
such that $f^n(x)\in U$ and $f^i(x)\notin \Crit$ for every $0\leq
i\leq n$, then the Schwarzian derivative of $f^{n+1}$ at $x$ is
negative.
\end{quote}
``Sufficiently small'' in this statement should be interpreted as
that for each component $U_c$ of $U$ and each $i \geq 0$, each
component of $f^{-i}(U_c)$ has length $\leq \tau$, where
$\tau=\tau(f)$ is taken from \cite[Proposition~3]{vSV} (choosing,
with the notation in \cite[Proposition~3]{vSV}, $S=1$, $N=0$ and
$\delta$ the number $\xi_0=\xi(f)$ from
Theorem~\ref{thm:lemma11}). If $U_c$ is sufficiently small, then
this holds by Proposition~\ref{prop:contraction}. To prove the
statement, let $(J_i)_{i=0}^n$ be the pullback chain of $U_c$
along $x, f(x), \dots, f^n(x) \in U_c$. Let $t_1 < t_2 < \dots <
t_m$ be the iterates such that $f^{t_j}(x) \in U$, and hence the
intervals $J_i$, $t_j+1 \leq i \leq t_{j+1}$ are pairwise disjoint
and $\sum_{i=t_j+1}^{t_{j+1}} |J_i| \leq 1$. Fix $j>1$ for the
moment. Then $J_{t_j+1}$ is contained in an entry domain $J$ to
$U$, say $f^{t_{j+1}-t_j-1}(J)=U_{c'}$. By
Theorem~\ref{thm:lemma11}(ii), there is $K\supset J$ such that
$f^{t_{j+1}-t_j-1}$ maps diffeomorphically $K$ onto $V_{c'}$,
hence the pullback chain of $V_{c'}$ along $J_{t_j+1}, \ldots,
J_{t_{j+1}}$ has order $0$. Moreover, $J_{t_{j+1}}$ is contained
in $U_{c'}$, so it is $\xi_0$-well inside $V_{c'}$.
Therefore \cite[Proposition~3]{vSV}
implies that $f^{t_{j+1} - t_j}$ has negative Schwarzian
derivative at $f^{t_j+1}(x)$. For $j = 1$ and $t_1 > 0$ (so $x
\notin U$), $f^{t_1+1}$ has negative Schwarzian derivative at $x$
by the same reasoning. If $t_1 = 0$ (so $x \in U$), then $f$ has
negative Schwarzian derivative at $x$ by non-flatness. Since
compositions of maps with negative Schwarzian derivatives have
negative Schwarzian derivative, the statement follows.

Now we continue with the proof of the lemma. Assume first $k=0$.
Fix $\xi>0$ and let $(H_i)_{i=0}^l\subset (G_i)_{i=0}^l$ be chains
such that $(G_i)_{i=0}^l$ has order $0$ and $G_l$ is contained in
a component $U_c$ of the set $U$ above. If $G_0$ is an entry
domain to $G_l$, then the statement is just
Proposition~\ref{prop:koebe} because the intervals $G_i$ are
pairwise disjoint ($G_l$ is nice). If not, again due to the
niceness of $G_l$, there is an interval $G_t\subset G_l$ such that
the intervals $G_{t+1},\ldots, G_l$ are pairwise disjoint, so we
can apply Proposition~\ref{prop:koebe} to the subchains
$(H_i)_{i=t+1}^l\subset (G_i)_{i=t+1}^l$. Also, $f^{t+1}|_{G_0}$
has negative Schwarzian derivative by the above statement, so we
can apply Proposition~\ref{prop:koebenegative} to the subchains
$(H_i)_{i=0}^{t+1}\subset (G_i)_{i=0}^{t+1}$.

The general case $k>0$ follows easily from this one (take also
Lemma~\ref{lem:regularity} into account). Now we need $G_l$ to be
small enough so that every component of every preimage of $G_l$ is
contained in $U$ whenever it intersects $\Crit$ (again by
Proposition~\ref{prop:contraction}).
\end{proof}

\subsection{Proof of Theorem~\ref{thm:induced}}

Let $x\in I$ be such that $\orb(x)$ is disjoint from
$\partial I$.
For every $n \in \N$ and every $0 < \eps < d(f^n, \partial I)$ we
construct the pullback chain of $(f^n(x)-\eps,f^n(x)+\eps)$ along
$x,f(x),\ldots,f^n(x)$. We
define
$$
r_n^k(x) = \sup\{ \eps > 0 : \text{ order of the pullback chain of }
(f^n(x)-\eps,f^n(x)+\eps) \leq k\}.
$$
If
for every $k$ we have $r_n^k(x)\to 0$ as $n\to \infty$, then we
call $x$ a \emph{super-persistent} point. In
\cite[Theorem~2.7]{BlMi} (see also \cite{BlMi2}) it is shown that
for  $f\in \CIInf(I)$, the $\omega$-limit set of any
super-persistent recurrent point of $f$
 is minimal.

\iffalse
\begin{theorem}
 \label{BloMis2}
 Let $f\in \CIInf(I)$.
 Then the $\omega$-limit set of any super-persistent recurrent point of $f$
 is minimal.
\end{theorem}
\fi

%\begin{corollary}\label{cor:afterprop14}
%Let $f \in \CIInf(I)$ and let $V$ be the nice critical
%neighborhood. Let $\tilde H$ be a nice interval that is $\xi$-well
%inside some component $V_c$. Then there is $\xi' > 0$ depending
%only on $f$ and $\xi$ such that every component $H$ of the
%preimage $f^{-n}(\tilde H) \cap V$ is $\xi'$-well inside $V$.
%\end{corollary}
%
%\begin{proof}
%Let $(H_i)_{i=0}^n$ be the pullback chain of $\tilde H$ along $H,
%\dots , f^n(H)$, so $H_0 = H$ and $H_n = \tilde H$. Let $t$ be
%minimal such that $H_t \subset V_c$.
%Proposition~\ref{macroscopickoebe} gives $\xi'' > 0$ such that
%$H_t$ is $\xi''$-well inside $V_c$. Moreover, the intervals $H_0,
%\dots, H_t$ are pairwise disjoint, so that $\sum_{i=0}^t |H_i|
%\leq 1$. Thus Proposition~\ref{prop:koebe} gives a $\xi'$,
%depending only on $\xi''$ and $f$, such that $H_0$ is $\xi''$-well
%inside the first interval of the pullback chain $(G_i)_{i=0}^t$ of
%$G_t := V_c$ along $H_0, \dots, H_t$. Since $V$ is nice, $G_0
%\subset V$, and hence $H_0$ is $\xi'$-well inside its component of
%$V$ as well.
%\end{proof}

\begin{proof}[\textbf{Proof of Theorem~\ref{thm:induced}}]
By definition of $E$ and type (2) attractors,
$\lambda$-a.e.\ $x \in E$ has a dense orbit
in some cycle $\cyc(K)$ and therefore cannot be super-persistently
recurrent, nor map into $\partial I$.
%There is no loss of generality in assuming that the
%orbits of points from $E$ do not accumulate on $\partial I$.
We can then find a sequence $n_j\to \infty$ and $N_x \in \N$ and
$\delta_x > 0$ such that the pullback chain of
$(f^{n_j}(x)-\delta_x, f^{n_j}(x)+\delta_x)$ along $x,\ldots,
f^{n_j}(x)$ has order at most $N_x$. Clearly we can take
$\delta_{f(x)} \geq \delta_x$ and $N_{f(x)} \leq N_x$ (in fact, we
have equality for large $j$ whenever $x$ is not a critical point).
Since Lebesgue measure has only finitely many ergodic components
\cite[Theorem~2]{Lypre}, it follows that there are a single
$\delta > 0$ and $N \in \N$ such that $\delta_x \geq \delta$ and
$N_x \leq N$ for $\lambda$-a.e.\ $x \in E$.

Next choose $\eps > 0$ in  Theorem~\ref{thm:lemma11} so small that the intervals
$U_c\subset V_c$, $c\in \Crit'$, satisfy:
\begin{itemize}
 \item if $c$ is in the interior of the metric attractor
 $\cyc(K)$r, then $\dist(V_c,\partial\cyc(K))>2\delta$;
 moreover, if $d\in \Crit\setminus \Crit'$, then either $d\in
 \partial \cyc(K)$ or $\dist(d,\cyc(K))>\delta$;
 \item every component of every preimage of $V_c$, and every
 image $f^n(V_c)$, $n\leq 2r$ (with $r$ the period of $\cyc(K)$),
 has length less than $\delta$.
\end{itemize}
Recall that $U'=\bigcup_{c\in \Crit'}U_c$, $V'=\bigcup_{c\in
\Crit'}V_c$. From their definition, both sets are nice. Moreover,
$\partial U'\cap D(V')=\emptyset$ where as before $D(V') = \cup_{n
> 0} f^{-n}(V')$.

Since $\orb(x)$ is dense in $\cyc(K)$, we can assume that the
numbers $n_j$ are large enough so that $f^{n_j}(x)\in \cyc(K)$ for
every $j$. Let $J$ be the entry domain to $U$ (or the component of
$U$) containing $f^{n_j}(x)$, say $f^{m_j}(J)=U_c$. Two
possibilities arise. If $J\subset \cyc(K)$, then $c\in \Crit'$.
Moreover, if $K\supset J$ is such that $f^{m_j}$ maps $K$
diffeomorphically onto $V_c$, then $K\subset
(f^{n_j}(x)-\delta,f^{n_j}(x)+\delta)$, so the pullback chain of
$V_c$ along $x,\ldots, f^{n_j+m_j}(x)$ has order at most $N$. The
second possibility is that $J$ contains some point from $\partial
\cyc(K)$, call it $a$. It is not possible that $a$ belongs to a
periodic orbit contained in $\partial \cyc(K)$, because this would
imply that one of the points of this periodic orbit belongs to
$\Crit\setminus \Crit'$, which contradicts that $\cyc(K)$ contains
dense orbits. Hence there is $r_j\leq 2r$ (with $r$ the period of
$\cyc(K)$) such that $|f^{n_j-r_j}(x)-c'|<\delta$ for some $c'\in
\Crit'$. In fact, $f^{n_j-r_j}(x)\in U_{c'}$, because otherwise
$f^{n_j-r_j}(x)$ would belong to an entry domain $J'$ to $U$
contained in $\cyc(K)$, which leads to the contradiction
$f^{r_j}(J')=J\subset \cyc(K)$. Since $f^{r_j}(V_{c'})\subset
(f^{n_j}(x)-\delta,f^{n_j}(x)+\delta)$, the pullback chain of
$V_{c'}$ along $x,\ldots, f^{n_j-r_j}(x)$ has order at most $N$.

We have proved that there is a number $N=N(f)$ such that, for a.e
$x\in D$, there are $c\in \Crit'$ and a sequence $s_j\to \infty$
such that $f^{s_j}(x)\in U_c$ and the pullback of $V_c$ along
$x,\ldots, f^{s_j}(x)$ has order at most $N$. However, we need
critical order $0$, not $N$, for this theorem, so a further
argument is required.

Let $k_x \geq 1$ be the smallest integer for which there are $G_x
\supset H_x \owns x$ and $c \in \Crit'$ such that $f^{k_x}$ maps
$G_x$ diffeomorphically onto $V_c$ and $f^{k_x}(H_x) = U_c$. If
$x\notin U'$, $\omega(x)=\cyc(K)$ and the entry domain $H_x\owns
x$ to $U'$ is contained in $\cyc(K)$, then $k_x$ exists; it is the
first entry time $k_x = r_{U'}(x)$. But if $x\in U'$, then $k_x$
is more difficult to find.

{\bf Claim 1.} The set $B := \{ x \in E : k_x \text{ exists}\}$
has full Lebesgue measure in $E$.

Assume by contradiction that this claim fails, and that $x$ is a
density point of $E \setminus B$. We can find the sequence
$(s_j)_{j \in \N}$ and $c\in \Crit'$  as above, and let
$(G^j_i)_{i=0}^{s_j}$ and $(H^j_i)_{i=0}^{s_j}$ be the pullback
chains of $G^j_{s_j} = V_c$ and $H^j_{s_j} = U_c$ along $x, \dots,
f^{s_j}(x)$, respectively. The chains $(G^j_i)_{i=0}^{s_j}$ have
order $\leq N$.

Let $(W_t)_{t \in \N}$ be a denumeration of the return domains to
$V'$ within $V_c \setminus U_c$ and also take $W_0 := U_c$. Recall
that $\partial U_c\cap D(V')=\emptyset$, so $(W_t)_{t \geq 0}$ is
a family of pairwise disjoint intervals whose union has full
measure in $V'$. Since $f^{s_j}|_{G^j_0}$ has at most $N$ critical
values, we can arrange the enumeration of the $W_t$s such that no
critical value of $f^{s_j}|_{G^j_0}$ is contained in $W_t$ if
$t>N$. By Proposition~\ref{prop:caili} there is $\xi=\xi(f)$
such that each $W_t$ is $\xi$-well inside $V_c$. By
Lemma~\ref{lem:boundedorder}, there is $\zeta=\zeta(\xi,N,f)$ such
that each component of $G^j_0 \cap f^{-s_j}(\bigcup_{t=0}^{N}
W_t)$ is $\zeta$-well inside $G^j_0$. This holds in particular for
the at most $M \leq 2^{N+1}-1$ components of
$f^{-s_j}(\bigcup_{t=0}^N W_t)$ in $G^j_0$. Let us denote these
components as $Y_i$, numbered in order of decreasing size.

{\bf Claim 2.} There are $b = b(\zeta, N) > 0$ and $B^j \subset
G^j_0$ such that $\lambda(B^j) \geq b \lambda(G^j_0)$ and for
every $y \in B^j$, there are $k_y \geq s_j$, $G_y \supset H_y
\owns y$ and $c(y) \in \Crit'$ such that $f^{k_y}$ maps $G_y$
diffeomorphically onto $V_{c(y)}$ and $f^{k_y}(H_y) = U_{c(y)}$.

Take $b = \zeta^M /(M+1)!$. The components $Y_i$ are pairwise
disjoint, and have $\zeta$-collars around them in $G^j_0$, which
may intersect other components $Y_{i'}$ or their collars. We will
show that $\lambda(\bigcup_i Y_i) \leq (1-b)\lambda(G^j_0)$. If
$\lambda(Y_1) <\lambda(G^j_0)/(M + 1)$, then we get directly
$$\lambda(\bigcup_i Y_i)<\lambda(G^j_0) M /(M +
1)=(1-1/(M+1))\lambda(G^j_0)<(1-b)\lambda(G^j_0).$$

Thus we can assume $\lambda(Y_1) \geq \lambda(G^j_0)/(M +1)$ and
the $\zeta$-collar of $Y_1$ has two components of length $\geq
\zeta \lambda(G^j_0)/(M+1)$. If the second largest $Y_2$ satisfies
$Y_2 <\zeta \lambda(G^j_0)/((M+1)M)$, then there is a set  in the
$\zeta$-collar of $Y_1$ of measure $>\zeta
\lambda(G^j_0)/((M+1)M)$ which do not intersect $\bigcup_i Y_i$.
Hence
$$\lambda(\bigcup_i Y_i)<(1-\zeta/((M+1)M))\lambda(G^j_0)<(1-b)\lambda(G^j_0)
$$
again. Continuing this way, we find that at least one component
$Y_k$ has a $\zeta$-collar disjoint from $\bigcup_i Y_i$ and its
size is $\geq  \zeta^{M} \lambda(G^j_0)/(M+1)! = b
\lambda(G_0^j)$.

Now take $B^j:= E\cap(G^j_0\setminus \bigcup_i Y_i)$. Then clearly
$\lambda(B^j) \geq b \lambda(G^j_0)$, and if $y \in B^j$, we find
$k_y$ as follows. We have $z := f^{s_j}(y) \in W_t$ for some $t >
N$. There is $r_t$ such that $f^{r_t}$ maps $W_t$
diffeomorphically onto a component of $V'$. If $f^{r_t}(z) \in
U'$, then $k_y = s_j +r_t$. Otherwise $f^{r_t}(z) \in V' \setminus
U'$ and in fact $f^{r_t}(z)$ belong to another return domain to
$V'$. We continue iterating diffeomorphically until finally $z$
falls into $U'$, and we choose $k_y$ accordingly. This proves the
Claim 2. Obviously $B^j \subset B$, so $\lambda(B \cap G^j_0) \geq
b\lambda(G^j_0)$ independently of $j$. Since $\lambda(G^j_0) \to
0$ as $j \to \infty$, this contradicts that $x$ is a density point
of $I \setminus B$, proving Claim 1.

From Claim 1 the first part of Theorem~\ref{thm:induced} easily
follows: if $x,y\in B$ and $k_x\leq k_y$, then either $H_x\cap
H_y=\emptyset$ or $H_y\subset H_x$. Now an easy maximality
argument allows to redefine these sets if necessary so that either
$H_x\cap H_y=\emptyset$ or $H_y=H_x$.

It remains to prove the second part of Theorem~\ref{thm:induced}.
Since a.e. point from $U'$ belongs to $B$, $F:U'\to U'$ is well
defined a.e. It is important to note that if $F$ is defined on
$x$, then $H_x\subset G_x\subset U'$ because $\partial(U')\cap
D(V')\neq \emptyset$. Now it is immediate so show by induction on
$n$ that every branch $F^n|_H=f^j|_H:H\subset U\to U_c$ of $F^n$
admits a diffeomorphic extension to an interval $H\subset G\subset
U$ with $f^j(G)=V_c$. Then $F^n|_H$ has distortion bounded by a
constant $\kappa$ depending neither on $n$ nor on $H$ by
Lemma~\ref{lem:boundedorder}.
\end{proof}

\medskip
\noindent
Department of Mathematics\\
University of Surrey\\
Guildford, Surrey, GU2 7XH\\
UK\\
\texttt{h.bruin@surrey.ac.uk}\\
\texttt{http://personal.maths.surrey.ac.uk/st/H.Bruin/}

\medskip \noindent
Departamento de Matem\'aticas\\
Universidad de Murcia\\
Campus de Espinardo,
30100 Murcia \\
Spain \\
\texttt{vjimenez@um.es}\\
\texttt{http://www.um.es/docencia/vjimenez}\\


\begin{thebibliography}{BKNS}

\bibitem[AJS]{AJS} L.\ Alsed\`a, V.\ Jim\'enez L\'opez, L. Snoha,
{\em  All solenoids of piecewise smooth maps are period doubling,}
Fund. Math. {\bf 157} (1998) 121--138.

\bibitem[BJ1]{BaJi1} F.\ Balibrea, V.\ Jim\'enez L\'opez,
{\em A structure theorem for $C^2$ functions verifying the
Misiurewicz condition,} in: Proceedings of the European Conference
on Iteration Theory (Lisbon, 1991),  12--21, World Sci.
Publishing, Singapur, 1992.

\bibitem[BJ2]{BaJi2} F.\ Balibrea, V.\ Jim\'enez L\'opez,
{\em The measure of scrambled sets: a survey,} Acta Univ. M. Belii
Ser. Math. No. {\bf 7} (1999) 3--11.

\bibitem[BD]{BD} M.\ Barge, B.\ Diamond,
{\em Proximality in Pisot tiling spaces,} Fund. Math. {\bf 194}
(2007) 191--238.

\bibitem[Ba]{Bar} J. A. \ Barnes,
{\em Conservative exact rational maps of the sphere,} J. Math.
Anal. Appl. {\bf 230} (1999) 350--374.

\bibitem[BrJ]{BrJi} A.\ Barrio Blaya, V.\ Jim\'enez L\'opez,
{\em An almost everywhere version of Sm\'{\i}tal's order-chaos
dichotomy for interval maps,} J. Austral. Math. Soc. {\bf 85}
(2008) 29--50.

\bibitem[BHS]{BHS} F.\ Blanchard, W.\ Huang, L.\ Snoha,
{\em  Topological size of scrambled sets, } Colloq. Math. {\bf
110} (2008) 293--361.

%\bibitem[BGT]{renewal} N.\ Bingham, C.\ Goldie, J.\ Teugels,
%{\em Regular variation,}
%Cambridge Univ. Press, Cambridge, 1987.

\bibitem[Bl]{Blo} L. S.\ Block, {\em Stability of periodic orbits
in the theorem of \v{S}arkovskii,}  Proc. Amer. Math. Soc. {\bf
81}  (1981) 333--336.

\bibitem[BC]{BlCo} L. S.\ Block, W. A.\ Coppel,
{\em Dynamics in one dimension,} Lecture Notes in Mathematics,
1513, Springer-Verlag, Berlin, 1992.

\bibitem[BK]{BK} L.\ Block, J.\ Keesling,
{\em A characterization of adding machine maps,} Topology Appl.
{\bf 140} (2004) 151--161.

\bibitem[BKM]{BKM} L.\ Block, J.\ Keesling, M.\ Misurewicz,
{\em Strange adding machines,} Ergodic Theory Dynam. Systems  {\bf
26}  (2006) 673--682.

\bibitem[BL]{BL} A.\ Blokh, M.\ Lyubich,
{\em Measurable dynamics of S-unimodal maps of the interval,} Ann.
Sci. \'Ecole Norm. Sup. (4)  {\bf 24} (1991) 545-573.

\bibitem[BM1]{BlMi}
A.\ Blokh, M.\ Misiurewicz, {\em Wild attractors of polymodal
negative Schwarzian maps,} Comm. Math. Phys. {\bf 199} (1998)
397--416.

\bibitem[BM2]{BlMi2}
A.\ Blokh, M.\ Misiurewicz, {\em Typical limit sets of critical
points for smooth interval maps,} Ergodic Theory Dynam. Systems
{\bf 20} (2000) 15--45. Erratum in Ergodic Theory Dynam. Systems
{\bf 23} (2003) 661.

\bibitem[BH]{BrHu} A. M.\ Bruckner, T.\ Hu, {\em On scrambled sets for chaotic
functions,} Trans. Amer. Math. Soc. {\bf 301} (1987) 289--297.

\bibitem[Br1]{Bruin} H.\ Bruin.
{\em Invariant measures of interval maps,}
Ph.D.~Thesis, University of Delft, 1994.

\bibitem[Br2]{Bknead} H.\ Bruin,
{\em Combinatorics of the kneading map,} Internat. J. Bifur. Chaos
Appl. Sci. Engrg. {\bf 5} (1995) 1339--1349.

\bibitem[Br3]{BTams} H.\ Bruin,
{\em Topological conditions for the existence of absorbing Cantor
sets,} Trans. Amer. Math. Soc. {\bf 350}  (1998) 2229--2263.

\bibitem[Br4]{Btree} H.\ Bruin,
{\em (Non)invertibility of Fibonacci-like unimodal maps restricted
to their critical omega-limit sets,} Preprint 2008.

\bibitem[BHa]{BrHaw} H.\ Bruin, J.\ Hawkins,
{\em  Exactness and maximal automorphic factors of unimodal
interval maps,} Ergodic Theory Dynam. Systems {\bf 21} (2001)
1009--1034.

\bibitem[BKNS]{BKNS} H.\ Bruin, G.\ Keller, T.\ Nowicki, S.\ van Strien,
{\em Wild Cantor Attractors exist,} Ann. of Math. (2) {\bf 143}
(1996) 97--130.

\bibitem[BKS]{BKS}  H.\ Bruin, G.\ Keller, M.\ St.-Pierre,
{\em Adding machines and wild attractors,} Ergodic Theory Dynam.
Systems {\bf 17} (1997) 1267--1287.

\bibitem[BSS]{BSS} H.\ Bruin, W.\ Shen, S.\ van Strien,
{\em Existence of unique SRB-measures is typical for real
unicritical polynomial families,}  Ann. Sci. \'Ecole Norm. Sup.
(4) {\bf 39} (2006) 381--414.

%\bibitem[BNT]{BNT} H.\ Bruin, M.\ Nicol, D.\ Terhesiu,
%{\em On Young towers associated with infinite measure preserving
%transformations,}
%Preprint 2009, to appear in Stoch. Dyn.

\bibitem[BRSS]{BRSS} H.\ Bruin, J.\ Rivera-Letelier, W.\ Shen, S.\ van Strien,
{\em Large derivatives, backward contraction and invariant
densities for interval maps,} Invent. Math. {\bf 172} (2008)
509--533.

\bibitem[CL]{caili} H.\ Cai, S. Li,
{\em Distortion of interval maps and applications,} Nonlinearity
(to appear).

\bibitem[D]{Day} R. H.\ Day, {\em Complex economic dynamics. Vol. I. An introduction
to dynamical systems and market mechanisms,} MIT Press, Cambridge,
1994.

\bibitem[E]{Engel} R.\ Engelking,
{\em Dimension theory,} PWN-Polish Scientific Publishers, Warsaw,
1978.
\bibitem[G]{Ged} T.\ Gedeon,
{\em There are no chaotic mappings with residual scrambled sets,}
Bull. Austral. Math. Soc. {\bf 36} (1987) 411--416.

\bibitem[GKR]{GKR} C.\ Good, R.\ Knight, B.\ Raines,
{\em Nonhyperbolic one-dimensional invariant sets with a countably
infinite collection of inhomogeneities,} Fund. Math. {\bf 192}
(2006) 267--289.

\bibitem[GLT]{GLT} P. J.\ Grabner, P.\ Liardet, R. F.\ Tichy,
{\em Odometers and systems of enumeration,} Acta Arith. {\bf 70}
(1995) 103--125.

\bibitem[GSS]{GSS} J.\ Graczyk, D.\ Sands, G. \'Swi\c atek,
{\em Metric attractors for smooth unimodal maps,} Ann. of Math.(2)
{\bf 159} (2004) 725--740.

\bibitem[Gu]{Guc} J.\ Guckenheimer, {\em Sensitive dependence to
initial conditions for one-dimensional maps,}  Comm. Math. Phys.
{\bf 70} (1979) 133--160.

\bibitem[HS]{HawSil} J.\ Hawkins, C.\ Silva,
{\em Characterizing mildly mixing actions by orbit equivalence of
products,} New York J. Math, {\bf 3A} (1998) 99-115.

\bibitem[H]{hofbauer} F.\ Hofbauer,
{\em The topological entropy of a transformation $x \mapsto
ax(1-x)$,} Monatsh. Math. {\bf 90} (1980) 117--141.

\bibitem[HK]{HK} F.\ Hofbauer, G.\ Keller.
{\em Quadratic maps without asymptotic measures,} Commun. Math.
Phys. {\bf 127} (1990) 319--337.

\bibitem[HY]{HuYe} W.\ Huang, X.\ Ye, {\em Homeomorphisms with
the whole compacta being scrambled sets,} Ergodic Theory Dynam.
Systems {\bf 21} (2001) 77--91.

\bibitem[Ja]{Jak} M.\ V.\ Jakobson,
{\em Absolutely continuous invariant measures for one-parameter
families of one-dimensional maps,}
Commun. Math. Phys. {\bf 81} (1981) 39-88.

\bibitem[JK]{JaSm} K.\ Jankov\'a, J.\ Sm\'{\i}tal,
{\em A characterization of chaos,}  Bull. Austral. Math. Soc.
{\bf} 34 (1986) 283--292.

\bibitem[J1]{Jim1} V. Jim\'enez L\'opez, {\em $C^1$ weakly chaotic
functions with zero topological entropy and non-flat critical
points,} Acta Math. Univ. Comenian. (N.S.) {\bf 60} (1991)
195--209.

\bibitem[J2]{Jim2} V. Jim\'enez L\'opez, {\em  Large chaos in
smooth functions of zero topological entropy,} Bull. Austral.
Math. Soc. {\bf 46} (1992) 271--285.

\bibitem[J3]{Jim3} V. Jim\'enez L\'opez, {\em Order and chaos for
a class of piecewise linear maps,} Internat. J. Bifur. Chaos Appl.
Sci. Engrg. {\bf 5} (1995) 1379--1394.

\bibitem[Jo]{Joh} S. D.\ Johnson, {\em Singular measures without
restrictive intervals,}  Comm. Math. Phys. {\bf 110}  (1987)
185--190.

\bibitem[K]{Kan} I.\ Kan, {\em A chaotic function possessing a scrambled set with
positive Lebesgue measure,}  Proc. Amer. Math. Soc. {\bf 92}
(1984) 45--49.

\bibitem[Ke]{keller} G.\ Keller.
{\em Exponents, attractors and Hopf decompositions for interval
maps,} Ergodic Theory Dynam. Systems {\bf 10} (1990) 717--744.

\bibitem[Ko]{Koz} O.\ S.\ Kozlovski,
{\em Getting rid of the negative Schwarzian derivative condition,}
Ann. of Math. (2) {\bf 152} (2000) 743--762.

\bibitem[Le]{Ledrap} F.\ Ledrappier,
{\em Some properties of absolutely continuous invariant measures
on an interval, }  Ergodic Theory Dynam. Systems {\bf 1}  (1981)
77--93.

\bibitem[LS]{LS} S.\ Li, W.\ Shen,
{\em Hausdorff dimension of Cantor attractors in one-dimensional dynamics,}
Invent. Math. {\bf 171} (2008) 1629--1643.

\bibitem[LY]{LY} T.\ Li, J. A.\ Yorke,
{\em Period three implies chaos,} Amer. Math. Monthly {\bf 82}
(1975) 985-992.

\bibitem[Ly1]{Lypre} M.\ Lyubich, {\em Ergodic theory for smooth one dimensional
dynamical systems,} Stony Brook preprint 1991/11.

\bibitem[Ly2]{Lyu} M.\ Lyubich, {\em Combinatorics, geometry and attractors of
quasi-quadratic maps,} Ann. of Math. (2) {\bf 140} (1994)
347--404.

\bibitem[LM]{LM} M.\ Lyubich, J.\ Milnor,
{\em The Fibonacci unimodal map,} J. Amer. Math. Soc. {\bf 6}
(1993) 425--457.

\bibitem[MMN]{MMN} M.\ Majumdar, T.\ Mitra, K.\ Nishimura (eds.),
{\em Optimization and chaos,} Studies in Economic Theory {\bf 11},
Springer-Verlag, Berlin, 2000.

\bibitem[M]{mane} R.\ Ma\~n\'e,
{\em Hyperbolicity, sinks and measure in one-dimensional
dynamics,} Comm. Math. Phys. {\bf 100} (1985) 495--524. Erratum in
Comm. Math. Phys. {\bf 112}  (1987) 721--724.

\bibitem[Ma]{martens} M.\ Martens,
{\em Interval Dynamics,} Ph.D.~Thesis, Delft University of
Technology, 1990.

\bibitem[MMS]{MMS} M.\ Martens, W.\ de Melo, S.\ van Strien, {\em
Julia-Fatou-Sullivan theory for real one-dimensional dynamics,}
Acta Math. {\bf 168} (1992) 273--318.

\bibitem[MT]{MT} M.\ Martens, C.\ Tresser, {\em Forcing of
periodic orbits for interval maps and renormalization of piecewise
affine maps,} Proc. Amer. Math. Soc. {\bf 124} (1996) 2863--2870.

\bibitem[MS]{dMvS} W.\ de Melo, S.\ van Strien,
{\em One dimensional dynamics,} Ergebnisse Series {\bf 25},
Springer-Verlag, Berlin, 1993.

\bibitem[Mi]{Mis} M.\ Misiurewicz, {\it Chaos almost everywhere,} in: Iteration
theory and its functional equations (Lochau, 1984), 125--130,
Lecture Notes in Math., 1163, Springer, Berlin, 1985.

\bibitem[P]{Pio} J.\ Pi\'orek, {\em On the generic chaos in
dynamical systems,}  Univ. Iagel. Acta Math. {\bf 25} (1985)
293--298.

\bibitem[S1]{Smi1} J.\ Sm\'{\i}tal,
{\em A chaotic function with some extremal properties,} Proc.
Amer. Math. Soc. {\bf 87} (1983) 54--56.

\bibitem[S2]{Smi2} J.\ Sm\'{\i}tal, {\em A chaotic function with a
scrambled set of positive Lebesgue measure,} Proc. Amer. Math.
Soc. {\bf 92} (1984) 50--54.

\bibitem[S3]{Smi3} J.\ Sm\'{\i}tal, {\em Chaotic functions with zero topological
entropy,} Trans. Amer. Math. Soc.  {\bf 297}  (1986) 269--282.

\bibitem[SS]{SmSt} J.\ Sm\'{\i}tal, M. \v{S}tef\'ankov\'a, {\em Omega-chaos almost
everywhere,} Discrete Contin. Dyn. Syst.  {\bf 9}  (2003)
1323--1327.

\bibitem[St]{Str} E. Straube,
{\em On the existence of invariant, absolutely continuous
measures,} Comm. Math. Phys. {\bf 81} (1981) 27--30.

\bibitem[vS]{vS} S.\ van Strien,
{\em Transitive maps which are not ergodic with respect to
Lebesgue measure,} Ergodic Theory Dynam. Systems {\bf 16} (1996)
833--848.

\bibitem[SV]{vSV} S.\ van Strien, E.\ Vargas,
{\em Real bounds, ergodicity and negative Schwarzian for
multimodal maps,}  J. Amer. Math. Soc.  {\bf 17}  (2004) 749--782.
Erratum in  J. Amer. Math. Soc.  {\bf 20}  (2007) 267--268.


\bibitem[Y]{Y}  L.-S. Young,
{\em Recurrence times and rates of mixing,} Israel. J. Math. {\bf
110} (1999) 153--188.

\end{thebibliography}
\end{document}